\pdfoutput=1
\documentclass[12pt]{amsart}
\usepackage{amsmath,thmtools}
\usepackage{amsfonts,color,amssymb}
\usepackage{geometry}
\usepackage[T1]{fontenc}
\usepackage[utf8]{inputenc}

\usepackage{mathrsfs,dsfont}
\usepackage[mathscr]{euscript}
\usepackage{mathtools}
\usepackage{mdwlist}
\usepackage{graphicx}
\usepackage{enumerate}
\usepackage[shortlabels]{enumitem}
\usepackage{a4wide,url}
\usepackage{amscd}
\newcommand{\chapter}{\part}
\makeatletter
\@addtoreset{section}{part}
\makeatother

\newtheorem{theorem}{Theorem}[section]
\newtheorem{lemma}[theorem]{Lemma}
\newtheorem{prop}[theorem]{Proposition}
\newtheorem{cor}[theorem]{Corollary}

\theoremstyle{definition}
\newtheorem{definition}[theorem]{Definition}
\newtheorem{example}[theorem]{Example}

\newtheorem{quest}{Problem}

\theoremstyle{remark}
\newtheorem{remark}{Remark}

\newcounter{smallromans}

\newenvironment{romanenumerate}
{\begin{list}{{\normalfont\textrm{(\roman{smallromans})}}}%
  {\usecounter{smallromans}\setlength{\itemindent}{0cm}%
   \setlength{\leftmargin}{5.5ex}\setlength{\labelwidth}{5.5ex}%
   \setlength{\topsep}{.5ex}\setlength{\partopsep}{.5ex}%
   \setlength{\itemsep}{0.1ex}}}%
{\end{list}}

\newcounter{smallromansdash}

{\end{list}}

\newcounter{bigromans} 
  {\end{list}}

\def\N{{\mathbb N}}
\let\<\langle
\let\>\rangle

\def\span{\mathop{\rm span}}

\def\cal{\fam\symsymbols}
\def\Lc{{\cal L}}

\def\U {\mathscr{U}}

\def\injtp{\widehat\otimes_\varepsilon}
\def\projtp{\widehat\otimes_\pi}

\author[M.\ Gonz\'alez]{Manuel  Gonz\'alez}
\address[M.\ Gonz\'alez]{Departamento de Matem\'aticas, Facultad de Ciencias, Universidad de Can\-ta\-bria, Avda. de los Castros s/n, 39071-Santander, Spain}
\email{manuel.gonzalez@unican.es}

\author[T.~Kania]{Tomasz Kania}
\address[T.~Kania]{Mathematical Institute\\Czech Academy of Sciences\\\v Zitn\'a 25 \\115 67 Praha 1\\Czech Republic  and  Institute of Mathematics and Computer Science\\ Jagiellonian University\\ {\L}ojasiewicza 6, 30-348 Krak\'{o}w, Poland
}
\email{kania@math.cas.cz, tomasz.marcin.kania@gmail.com}

\title{Grothendieck spaces: the landscape and perspectives}
\thanks{Research of M. Gonz\'alez was partially supported by MICINN (Spain), Grant PID2019-103961GB-C22. T. Kania acknowledges with thanks funding received from SONATA 15 No. 2019/35/D/ST1/01734. }
\date{\today}

\keywords{Grothendieck space, G-space, property $(V)$, property $(V_\infty)$, Banach space, Grothendieck operator, C*-algebra, Banach lattice, positive Grothendieck property, quantitative Grothendieck property, pseudo-intersection number, locally convex Grothendieck space, twisted sum, push-out.}
\subjclass[2010]{Primary: 46A35; 46B20.}

\begin{document}
\begin{abstract}
In 1973, Diestel published his seminal paper \emph{Grothendieck spaces and vector measures} that drew a connection between Grothendieck spaces (Banach spaces for which weak- and weak*-sequential convergences in the dual space coincide) and vector measures. This connection was developed further in his book with J.~Uhl Jr.~\emph{Vector measures}. Additionally, Diestel's  paper included a section with several open problems about the structural properties of Grothendieck spaces, and only half of them have been solved to this day.\smallskip

The present paper aims at synthetically presenting the state of the art at subjectively selected corners of the theory of Banach spaces with the Grothendieck property, describing the main examples of spaces with this property, recording the solutions to Diestel's problems, providing generalisations/extensions or new proofs of various results concerning Grothendieck spaces, and adding to the list further problems that we believe are of relevance and may reinvigorate a~better-structured development of the theory.

\end{abstract}
\maketitle
\pagebreak

\tableofcontents

\chapter{Introduction}\label{chap:intro}

A Banach space $X$ is \emph{Grothendieck}
whenever every weakly* convergent sequence in the dual space $X^*$ converges weakly. The notion of a Grothendieck space was coined after Grothendieck proved that for any index set $\Gamma$ the space of bounded functions on $\Gamma$, $\ell_\infty(\Gamma)$, and more generally, the space $C(K)$ of continuous functions on a compact, extremally disconnected (\emph{Stonean}) space $K$ have this property (\cite[Th\'eor\`eme 9]{Grothendieck:53}).

Without a doubt, reflexive spaces are Grothendieck, and from the point of view of the dual space, the Grothendieck property may be indeed viewed as a form of `sequential reflexivity'. The Grothendieck property may be regarded as complementary to being weakly compactly generated---in particular, these two properties together imply reflexivity. Grothendieck property is of relevance in summability theory in Banach spaces (see \cite{GLR:21} and references therein) or in the theory of $C_0$-semigroups of operators, where many features available for semigroups of operators on reflexive spaces may be emulated. 

In \cite{Diestel:73}, Diestel explained the relevance of the Grothendieck property in the study of finitely additive vector-valued measures; the paper itself includes a section of open problems concerning this property, most of which to this day remain open. Even though many particular examples (and non-examples) of Grothendieck spaces have been identified since, little progress on stability properties of Grothendieck spaces has been obtained so far. 

The present paper aims at synthetically presenting various corners of the theory of Banach spaces with the Grothendieck property with the perspective of putting it on a~more systematic footing. For this reason, we describe the main isomorphic properties of Grothendieck spaces, collect the answers that have been obtained for some of Diestel's problems, revisit the stubbornly open ones as well as raise a number of new ones. Occasionally, we generalise certain facts concerning Grothendieck spaces. In order to properly motivate the problems by related results, we have stated them along the text, as well as we have included a comprehensive list of open problems in the final section. Let us briefly describe the contents of the paper. 

Chapter \ref{chap:prelim} is devoted to introducing concepts that will be used throughout our exposition.  
Chapter \ref{chap:properties} includes several characterisations and properties of Grothendieck spaces.  In Chapter \ref{chap:examples} we exemplify  Grothendieck spaces within certain familiar classes of Banach spaces such as $C(K)$-spaces, $C^*$-algebras, $\Lc_\infty$-spaces, Banach lattices, \emph{etc.} as well as we provide sufficient conditions for spaces in these classes to be Grothendieck. Chapter \ref{chap:stability} touches upon preservation of the Grothendieck property via various constructions and discusses methods of building new Grothendieck spaces from the already known ones. This includes twisted-sums, tensor products, ultraproducts, and more. In Chapter \ref{chap:miscellanea} we briefly describe some additional results on the Grothendieck property, and Chapter \ref{chap:problems} provides a~concise list of the problems that we have collected in the previous chapters.

\chapter{Preliminaries}\label{chap:prelim}
The Banach spaces we consider are either real or complex, however when we talk about $C^*$-algebras, the underlying Banach spaces are always complex. An \emph{operator} is a continuous linear operator between Banach spaces (unless we explicitly state otherwise). We denote by $B_X$ the closed unit ball of a Banach space $X$. The symbol $\mathscr{B}(X,Y)$ stands for the space of all operators between Banach spaces $X$ and $Y$; the spaces of compact operators $\mathscr{K}(X,Y)$ and weakly compact operators $\mathscr{W}(X,Y)$ are closed subspaces thereof; and we write $\mathscr{B}(X)$, $\mathscr{K}(X)$, and $\mathscr{W}(X)$ when $X=Y$. We write $X\cong Y$, whenever two Banach spaces $X$ and $Y$ are isomorphic and $X\equiv Y$ in the case where they are isometrically isomorphic. 

When $X$ and $Y$ are isomorphic, we denote the (multiplicative) Banach--Mazur distance between $X$ and $Y$ by $$d_{\rm BM}(X,Y) = \inf\big\{\|T^{-1}\|\colon T\in \mathscr{B}(X,Y)\text{ is a norm-one isomorphism}\big\}.$$

The \emph{density character} of a Banach space $X$, denoted ${\rm dens}\,X$, is the least cardinal $\lambda$ for which $X$ has a dense subset of cardinality $\lambda$. \smallskip

A Banach space $X$ is \emph{weakly compactly generated} (WCG, for short) if it contains a~weakly compact subset that spans a dense subspace. Every separable Banach space, every reflexive Banach space, and spaces of the form $c_0(\Gamma)$ for any set $\Gamma$ are WCG.\smallskip

A Banach space $X$ has \emph{weak$^*$ sequentially compact dual ball} (has W$^*$SC dual ball, for short) whenever every sequence in $B_{X^*}$ has a weak$^*$-convergent subsequence; we refer to \cite[Chapter XIII]{Diestel:84} for systematic study of spaces with this property.
Each subspace of a WCG Banach space has W$^*$SC dual ball  \cite[Theorem XIII.4]{Diestel:84}. Also the spaces whose dual spaces contain no copies of $\ell_1$, and the spaces isomorphic to the dual of a separable space containing no copies of $\ell_1$. Clearly, a Grothendieck space with weak$^*$ sequentially compact dual ball is reflexive. We refer to \cite[Chapter XIII]{Diestel:84} for additional information.  \smallskip 

An operator $T\colon X\to Y$ is \emph{weakly pre-compact} if for every bounded sequence $(x_n)_{n=1}^\infty$ in $X$, $(Tx_n)_{n=1}^\infty$ has a weakly Cauchy subsequence; equivalently, by Rosenthal's $\ell_1$-theorem \cite[Chapter XI]{Diestel:84}, if there exists no subspace of $X$ isomorphic to $\ell_1$ on which the operator $T$ is an isomorphism. We denote by $\mathscr{R}(X,Y)$ the space of weakly pre-compact operators from $X$ into $Y$. \smallskip

A series $\sum_{n=1}^\infty x_n$ in $X$ is \emph{weakly unconditionally converging} if for every $f\in X^*$ one has  $\sum_{n=1}^\infty |\langle f, x_n\rangle|<\infty$.   
An operator $T\colon X\to Y$ is \emph{unconditionally converging} if for each  weakly unconditionally converging series $\sum_{n=1}^\infty x_n$ in $X$, the series $\sum_{n=1}^\infty Tx_n$ is unconditionally converging in $Y$; equivalently, if there is no subspace of $X$ isomorphic to $c_0$ on which $T$ is an isomorphism \cite[Chapter  III, Lemma 3.3.A]{HarmandWW:93}. We denote by $\mathscr{U}(X,Y)$ the space of unconditionally converging operators from $X$ into $Y$. \smallskip

A Banach space $X$ has \emph{property $(V)$}, whenever every unconditionally converging operator $T\colon X\to Y$ is weakly compact; \emph{i.e.}, when $\mathscr{U}(X,Y) \subset \mathscr{W}(X,Y)$ for every Banach space $Y$. Reflexive spaces and $C(K)$-spaces have property $(V)$. Moreover, property $(V)$ is inherited by quotients (but not by closed subspaces). 
We refer to \cite[Section III.3]{HarmandWW:93} for additional information on this property.\smallskip

A Banach space $X$ has the \emph{Dunford--Pettis property}, when for every Banach space $Y$ every weakly compact operator $T\colon X\to Y$ takes weakly convergent sequences of $X$ into convergent sequences of $Y$. It was proved by Grothendieck \cite{Grothendieck:53} that $C(K)$-spaces and $L_1(\mu)$ for any measure $\mu$ have the Dunford--Pettis property. Additionally, $\Lc_\infty$-spaces and $\Lc_1$-spaces have this property. A Banach space $X$ has the Dunford--Pettis property, whenever $X^*$ has the Dunford--Pettis property. We refer to \cite{Diestel:80} for a complete survey on the Dunford--Pettis property. \smallskip

A \emph{Schauder decomposition} of a Banach space $X$ is a sequence $(X_i)_{i=1}^\infty$ of closed non-zero subspaces of $X$ such that every $x\in X$ has a unique representation $x=\sum_{i=1}^\infty x_i$ with $x_i\in X_i$ for every $i$ \cite[Definition 1.g.1]{LT1}.  It determines (and is determined by) a sequence $(P_i)_{i=1}^\infty$ of non-zero projections on $X$ such that $P_iP_j = 0$ for $i\neq j$ and $x=\sum_{i=1}^\infty P_i x$ for each $x\in X$; note that each subspace $M_i$ corresponds with the range of the projection $P_i$. This concept is equivalent to the notion of a Schauder basis when the projections have one-dimensional ranges.\smallskip

Let $X$ be a Banach space. A \emph{Markuschevich basis} (\emph{M-basis}, for short) of $X$ is a~family $\{(x_\gamma,f_\gamma)\colon \gamma\in \Gamma\}$ in $X\times X^*$, which is biorthogonal ($\langle f_{\gamma_1}, x_{\gamma_2}\rangle = \delta_{\gamma_1,\gamma_2}$), fundamental ($\overline{\textrm{span}}\{x_\gamma\colon \gamma\in \Gamma\}=X)$, and total ($\langle f_\gamma, x\rangle =0$ for each $\gamma\in \Gamma$ implies $x=0$). Obviously, a~Schauder basis with the associated  sequence of biorthogonal functionals is an $M$-basis; however, the concept of an M-basis is much more general: WCG Banach spaces such as $c_0(\Gamma)$ for any set $\Gamma$, or $C(K)$-spaces, where $K$ is a Valdivia compact space, admit M-bases. For these and other examples, consult, \emph{e.g.}, \cite[Corollary 5.2]{HajekMVZ:08}. \smallskip

\section{{$\mathcal{L}_p$}-spaces and related concepts}
Let $1\leqslant\lambda<\infty$ and $1\leqslant p \leqslant \infty$. A Banach space $X$ is a~\emph{$\mathcal{L}_{p, \lambda}$-space} whenever every finite-dimensional subspace $F$ of $X$ is contained in another finite-dimensional space $G$ of $X$ whose Banach--Mazur distance to $\ell_p^{\dim G}$ is at most $\lambda$. A Banach space $X$ is said to be a \emph{$\Lc_p$-space} if it is a \emph{$\Lc_{p,\lambda}$-space} for some $\lambda$.
We will be primarily interested in the $\Lc_\infty$-spaces and the $\Lc_1$-spaces.
A Banach space $X$ is a~\emph{Lindenstrauss space}, whenever it is a $\mathcal{L}_{\infty, \lambda}$-space for every $\lambda>1$; equivalently, if $X$ is an isometric predual of some $L_1(\mu)$-space. Spaces of the form $C_0(K)$ for some locally compact space $K$ are natural examples of Lindenstrauss spaces.

We say that  a Banach space $X$ is a  $\widetilde{\mathcal{OL}}_{ \infty, \lambda}$-space  whenever every finite-dimensional subspace of $X$ is contained in another finite-dimensional subspace whose Banach--Mazur distance to some finite-dimensional $C^*$-algebra is at most $\lambda$. The  $\widetilde{\mathcal{OL}}_{\infty}$-spaces are the space which are $\widetilde{\mathcal{OL}}_{\lambda,\infty}$-spaces for some $\lambda \geqslant 1.$

The just-introduced terminology concerning $\widetilde{\mathcal{OL}}_{\infty}$-spaces is not standard as the non-commutative analogues of $\mathcal{L}_{\infty, \lambda}$-spaces---usually termed $\mathcal{OL}_{\infty, \lambda}$-spaces---are defined in terms of the \emph{completely bounded} Banach--Mazur distance to the set of finite-dimensional $C^*$-algebras. $\mathcal{OL}_{\infty, \lambda}$ are $\widetilde{\mathcal{OL}}_{\infty, \lambda}$-spaces. A $C^*$-algebra is nuclear if and only if it is $\mathcal{OL}_{\infty, \lambda}$ for some $\lambda > 6$ (\cite[Theorem 1.2]{Junge:2003}). We refer to \cite{Junge:2003} for additional information and a more precise definition of these spaces.

\section{Direct sums and ultraproducts}\label{sect:direct-sums}
Let $\Gamma$ be a (possibly uncountable) set and let $E$ be a Banach space that has a~normalised, 1-un\-con\-di\-tio\-nal basis $(e_\gamma)_{\gamma\in \Gamma}$. Given a collection $(X_\gamma)_{\gamma\in \Gamma}$ of Banach spaces, the \emph{$E$-sum of $(X_\gamma)_{\gamma\in \Gamma}$,}  denoted by $(\bigoplus_{\gamma\in \Gamma} X_\gamma)_E$, is the set of all tuples $(x_\gamma)_{\gamma\in \Gamma}$ with $x_\gamma \in X_\gamma$ for each $\gamma\in \Gamma$, and such that $\sum_{\gamma\in \Gamma} \|x_\gamma\| e_\gamma \in E$. Formally, the $E$-sum depends on the choice of the basis but when the basis is clear from the context we can afford this abuse of notation.

If the basis $(e_\gamma)_{\gamma\in \Gamma}$ is shrinking, the coordinate functionals $(e^*_\gamma)_{\gamma\in \Gamma}$ associated to that basis form a normalised, 1-unconditional basis of $E^*$. In this case, the map $$\Lambda_E\colon (\bigoplus_{\gamma\in \Gamma} X_\gamma^*)_{E^*}\to \Big((\bigoplus_{\gamma\in \Gamma} X_\gamma)_{E}\Big)^*$$ given by
$$\big\langle (x_\gamma)_{\gamma\in \Gamma}, \Lambda_E (f_\gamma)_{\gamma\in \Gamma}\big\rangle = \sum_{\gamma\in \Gamma} \langle x_\gamma, f_\gamma\rangle
$$
for $(x_\gamma)_{\gamma\in \Gamma}\in (\bigoplus_{\gamma \in \Gamma} X_\gamma)_{E}$ and $(f_\gamma)_{\gamma\in \Gamma} \in (\bigoplus_{\gamma\in \Gamma} X^*_\gamma)_{E^*}$, is an isometric isomorphism (see \cite[Section 4]{Laustsen:2001} for details). When the basis $(e_\gamma)_{\gamma\in \Gamma}$ is additionally boundedly complete (which implies that $E$ is reflexive), the bidual of $(\bigoplus_{\gamma\in \Gamma} X_\gamma)_E$ is naturally isometrically isomorphic to $(\bigoplus_{\gamma\in \Gamma} X_\gamma^{**})_E$.\bigskip

Let $\Gamma$ be an infinite set and let $\{X_\gamma\colon \gamma\in \Gamma\}$ be a~family of Banach spaces. The \emph{$\ell_\infty$-sum of $(X_\gamma)_{\gamma\in \Gamma}$,} denoted  $(\bigoplus_{\gamma\in \Gamma} X_\gamma)_{\ell_\infty(\Gamma)}$, is the set of all tuples $(x_\gamma)_{i\in \Gamma}$ such that $x_\gamma\in X_\gamma$ ($\gamma\in \Gamma$) and $\|(x_\gamma)_{\gamma\in \Gamma}\|_{\ell_\infty(\Gamma)} = \sup_{\gamma\in \Gamma}\|x_\gamma\|< \infty$; it is a Banach space with respect to this norm.
When $X_\gamma=Z$ for each $\gamma\in \Gamma$, we write $X = \ell_\infty(\Gamma,Z)$.

Let $\U$ be a non-trivial ultrafilter on $\Gamma$ (assuming that $\Gamma$ is infinite). Then
$$
N_\U\big((X_\gamma)_{\gamma\in \Gamma}\big)=\big\{(x_\gamma)_{\gamma\in \Gamma}\in E\colon \lim_{\gamma\to\U}\|x_\gamma\| =0\big\}
$$
is a closed subspace of $X:=(\bigoplus_{\gamma \in \Gamma} X_\gamma)_{E}$. The \emph{ultraproduct} of $\{X_\gamma\colon \gamma\in \Gamma\}$ along $\U$, denoted $[X_\gamma]_\U$, is defined as the quotient space $X/N_\U((X_\gamma)_{\gamma\in \Gamma})$; when $X_\gamma = Z$ for all $\gamma\in \Gamma$, then we call the quotient space an \emph{ultrapower} of $Z$ (along $\U$).

\section{Ordered normed spaces}\label{sect:ordered}
Let $E$ be a normed space. A convex subset $P\subset E$ is a \emph{cone}, whenever $\lambda P = P$ for all $\lambda > 0$ and $P\cap (-P) = \{0\}$. An \emph{ordered normed space} is a~normed space $E$ with a distinguished cone $P$; we will denote by $(E, P)$ an ordered normed space or simply by $E$ if the choice of the cone is clear from the context. 

For $x,y\in E$ we write $x\leqslant y$ as long as $y-x\in P$. If there exists a~number $\lambda > 0$ such that whenever $0\leqslant x \leqslant y$ ($x,y\in E$) we have $\|x\| \leqslant \lambda \|y\|$, the cone $P$ is called \emph{normal}. An~\emph{order unit} of $(E,P)$ is an element $e\in P$ such that for every $x\in P$ there is some $n\in \mathbb N$ with $x\leqslant n\cdot e$. A convex subset $B$ of $P$ is called a \emph{base} for $P$, whenever for each non-zero $x\in P$ there is a unique real number $f (x) > 0$ such that $f(x)^{-1}x\in B$. A~cone $P$ is \emph{well-based} in $X$ if it has a bounded base $B$ defined by a some $f\in X^*$, that is, $B = \{x\in P\colon \langle f,x\rangle = 1\}$. An element $x\in P$ is a  \emph{quasi-interior point} of $P$, when the order interval $[0, x] = \{ z\in E\colon 0 \leqslant z \leqslant x\}$ is linearly dense in $X$. Whenever $E$ is a~Banach lattice or a $C^*$-algebra, by default we consider the (canonical) positive cone $E_+$ therein.

\begin{definition}
Let $(E,P)$ be an ordered normed space.
\begin{itemize}
\item $(E,P)$ has the  \emph{countable (Riesz) interpolation property}, whenever given two sequences $(x_n)_{n=1}^\infty$, $(y_n)_{n=1}^\infty$ with $x_n \leqslant y_m$ ($n,m\in \mathbb N$), there exists $z\in E$ such that $x_n \leqslant z \leqslant y_n$ ($n\in \mathbb N$).
\item $(E,P)$ has the \emph{countable monotone (Riesz) interpolation property}, whenever for any non-increasing sequence $(x_n)_{n=1}^\infty$ in $E$ and any non-decreasing sequence $(y_n)_{n=1}^\infty$ in $E$ with $x_n \leqslant y_n$ ($n\in \mathbb N$), there exists $z\in E$ such that $x_n \leqslant z \leqslant y_n$ ($n\in \N$).
\end{itemize}
\end{definition}

These two properties are equivalent in vector lattices but in general they are different as witnessed by the $C^*$-algebra of $2\times 2$ matrices.\smallskip

In $C^*$-algebras the positive cone is automatically normal with $\lambda = 1$ (\cite[Theorem 2.2.5]{Murphy:90}). Every $C^*$-algebra $A$ decomposes into the real direct sum $A = A_{\rm sa} \oplus_{\mathbb R} iA_{\rm sa}$, where $A_{\rm sa}$ is the self-adjoint part of $A$, which is naturally a real ordered Banach space. By definition, a $C^*$-algebra $A$ has the countable (monotone)  interpolation property, whenever the ordered space $ A_{\rm sa}$ has this property. The countable monotone interpolation property of $C^*$-algebras has been studied in \cite{SmithWilliams:86, SmithWilliams:88}.\smallskip

Let $A$ be a $C^*$-algebra decomposed into $A_{\rm sa} \oplus_{\mathbb R} iA_{\rm sa}$. A functional $f\in A^*$ is \emph{self-adjoint} ($f\in (A^*)_{\rm sa}$), whenever $\langle f,x\rangle = \langle f, x^*\rangle$ for every $x\in A$ and this happens if and only if $f$ restricted to $A_{\rm sa}$ takes real values only. We have the canonical isometric identification $(A_{\rm sa} \oplus_{\mathbb R} iA_{\rm sa})^* \equiv (A_{\rm sa})^* \oplus_{\mathbb R} i (A_{\rm sa})^*$.

\section{Exact sequences, twisted sums, and three-space properties}\label{sect:3SP}

Let $Y$ and $Z$ be Banach spaces. We say that a~Banach space $X$ is a \emph{twisted sum of $Y$ and $Z$} if there exists a closed subspace $M$ of $X$ so that $M$ is isomorphic to $Y$ and $X/M$ is isomorphic to $Z$.
Note that the notion of a twisted sum is not symmetric.

When $M$ is complemented in $X$, we say that the twisted sum is \emph{trivial}; otherwise, the twisted sum is \emph{non-trivial}.
When the twisted sum is trivial, $X$ is naturally isomorphic to the direct sum $Y\oplus Z$.

An \emph{exact sequence} is a diagram
\begin{equation}\label{exact-seq}
\begin{CD}
0@>>> Y @>j>> X @>q>>Z@>>> 0
\end{CD}
\end{equation}
in which $X$, $Y$ and $Z$ are Banach spaces, and $j$ and $q$ are continuous operators with $j$ injective, $q$ surjective and $j(Y)=\ker q$; thus $j(Y)$ is a closed subspace of $X$ and $X/j(Y)$ is isomorphic to $Z$. \smallskip

A twisted sum of $Y$ and $Z$ can be identified with an exact sequence like (\ref{exact-seq}), and the twisted sum is trivial precisely when the exact sequence splits. \smallskip

Let $\wp$ be a property of Banach spaces. Then $\wp$ is a \emph{three-space property} (or the class of spaces satisfying $\wp$ has \emph{the three-space property}), when $\wp$ is stable under twisted sums; meaning that if $Y$ and $Z$ in (\ref{exact-seq}) have $\wp$, then so has $X$. We refer to \cite{Castillo-Gonzalez:97} for additional information.

\section{Spaces of continuous functions}\label{sect:C(K)}
Let $K$ be a compact space and let $L$ be a non-empty proper closed subset of $K$. Then the difference $K\setminus L$ is a locally compact space, and $C_0(K\setminus L)$ has codimension one in $C((K\setminus L)_\infty)$, where $(K\setminus L)_\infty$ is the one-point compactification of $K\setminus L$.

Thus, if we consider the continuous map $\varphi\colon K\to (K\setminus L)_\infty$ defined by $\varphi(t)= t$ for $t\in K\setminus L$ and $\varphi(t)= \infty$ for $t\in L$, we obtain a natural exact sequence
\begin{equation}
\begin{CD}\label{C(K)-seq}
0@>>> C_0(K\setminus L) @>\phi>> C(K) @>R>>C(L)@>>> 0,
\end{CD}
\end{equation}
where $\phi g = g\circ \varphi$ and $R$ is the restriction map $f\to f|_L$, which is surjective by the Tietze--Urysohn extension theorem.

For a given set $\Gamma$, the \v Cech--Stone compactification $\beta \Gamma$ of $\Gamma$ comprises all ultrafilters on $\Gamma$; this space is topologised by the base $\{\mathscr U\in \beta \Gamma\colon A \in \mathscr U\}$ ($A\subseteq \Gamma$). By taking $K = \beta\N$ and $L=\beta\N\setminus\N$, the exact sequence (\ref{C(K)-seq}) becomes $0\to c_0\to \ell_\infty\to \ell_\infty/c_0\to 0$.

More generally, for a given filter $\mathscr{F}$ on $\Gamma$ the set $U_{\mathscr{F}}$ consisting of all ultrafilters \emph{not} extending $\mathscr{F}$ is open. Set $K_{\mathscr U} = \beta \Gamma \setminus U_{\mathscr{F}}$. By the Tietze--Urysohn extension theorem, the closed subspace 
\begin{equation}\label{c0F-def}c_{0,\mathscr{F}} = \{f\in \ell_\infty(\Gamma)\colon \forall \varepsilon > 0\;\exists A\in \mathscr{F}\; \forall \gamma\in A\; |f_\gamma| < \varepsilon \}\subseteq \ell_\infty(\Gamma)\end{equation}
is isometrically isomorphic to $C_0(U_{\mathscr{F}})$, the space of continuous functions vanishing at infinity on the locally compact space $U_{\mathscr{F}}$. In particular, we have the following short exact sequence
\begin{equation}
\begin{CD}\label{filt-seq}
0@>>> c_{0,\mathscr{F}} @>\phi>> \ell_\infty(\Gamma) @>R>>C(K_{\mathscr{F}})@>>> 0.
\end{CD}
\end{equation}

\section{Semigroups of operators}
Let $E$ be a Banach space. A \emph{semigroup of operators on $E$}  is a map $T(\cdot)\colon [0,\infty)\to \mathscr{B}(E)$ such that
\begin{itemize}
 \item $T(0)=I_E$, the identity on $E$;
 \item $T(s+t)=T(s)T(t)$ ($s,t\geqslant 0$); 
\end{itemize}

A semigroup $T(\cdot)$ is a \emph{$C_0$-semigroup} whenever it is strongly continuous in the sense that
for each $x_0\in X$, $\lim\limits_{t\to 0^+} \|T(t)x_0-x_0\|=0$. 

A $C_0$-semigroup $T(\cdot)$ is \emph{uniformly continuous} if $\lim_{t\to 0+} \|T(t)-I_E\|=0$. 

A semigroup $T(\cdot)$ is \emph{locally integrable}, whenever for each $x\in E$ the map $t\mapsto T(t)x$ is strongly measurable and $\int_0^t \|T(s)x\|\,{\rm d}s < \infty$ for all $t> 0$ and $x\in E$. 

The \emph{infinitesimal generator} $A$ of a strongly continuous semigroup $T(\cdot)$ is defined by
$$ A\,x = \lim_{t\to 0^+} \frac1t\,(T(t)- I)\,x \quad (x\in E)$$
as long as the limit exists. \smallskip

$C_0$-semigroups of operators are a standard tool for solving abstract Cauchy problems. We refer to Lotz \cite{Lotz:85} for additional information.

\section{The pseudo-intersection number}
Grothendieck spaces and their properties often appear to be related to the following cardinal number $\mathfrak{p}$, that is known in the literature as the \emph{pseudo-intersection number}. It is defined as the smallest cardinal number $\lambda$ with the property that for every family $\mathscr{A}$ of subsets of $\mathbb N$ such that:
\begin{itemize}
    \item $|\mathscr{A}|<\lambda$,
    \item for all $A_1, \ldots, A_n\in \mathscr{A}$ ($n\in \mathbb N$), the set $A_1 \cap \ldots \cap A_n$ is infinite
\end{itemize}
there exists an infinite set $B$ such that $B\setminus A$ is finite for every $A\in \mathscr{A}$. 
Here $B$ is a~\emph{pseudo-intersection} of $\mathscr{A}$, hence the symbol $\mathfrak{p}$.

It is known that $\mathfrak{p}$ is an uncountable cardinal number between $\omega_1$ and the continuum and that Martin's Axiom is equivalent to $\mathfrak{p} = \mathfrak{c}$; in particular the value of $\mathfrak{p}$ cannot be determined in ZFC in terms of the aleph numbers. We refer to \cite{Fremlin:1984} for further information concerning $\mathfrak{p}$ and its applications in Analysis.

\chapter{Isomorphic properties of Grothendieck spaces}\label{chap:properties}

In the present chapter we collect the main results concerning Banach spaces with the Grothendieck property as well as some related problems.

\begin{definition}
A Banach space $X$ is \emph{Grothendieck} (or has the \emph{Grothendieck property}) whenever every weakly* convergent sequence in its dual  $X^*$ is weakly convergent.
\end{definition}

\section{Characterisations and first  properties}

We begin with a portmanteau characterisation theorem of Grothendieck spaces. Most of these equivalences  were proved in \cite[Theorem 1]{Diestel:73}. We present an easier proof. \medskip

Let $S_n, S\in \mathscr{B}(X,Z)$ ($n\in \N$). We say that the sequence $(S_n)_{n=1}^\infty$ converges to $S$ in the \emph{strong operator topology} if $(S_nx)_{n=1}^\infty$ converges to $Sx$ for each $x\in X$; and we say that $(S_n)_{n=1}^\infty$ converges to $S$ in the \emph{weak operator topology} if $(\langle S_nx,z^* \rangle)_{n=1}^\infty$ converges to $(\langle Sx,z^* \rangle)_{n=1}^\infty$ for each $x\in X$ and $z^*\in Z^*$. 

\begin{theorem}\label{Groth-sp}
For a Banach space $X$, the following assertions are equivalent:
\begin{enumerate}
  \item\label{Groth-sp:1} $X$ is a Grothendieck space;
  \item\label{Groth-sp:2} every operator $T\colon X\to c_0$ is weakly compact;
  \item\label{Groth-sp:3} for each separable space $Y$, every operator $T\colon X\to Y$ is weakly compact.
  \item\label{Groth-sp:4} for each WCG space $Y$, every operator $T\colon X\to Y$ is weakly compact.
  \item\label{Groth-sp:5} for each $Y$ with W$^*$SC dual ball, every operator $T\colon X\to Y$ is weakly compact; 
  \item\label{Groth-sp:6} for each space $Z$, $\mathscr{W}(X,Z)$ is sequentially closed in $\mathscr{B}(X,Z)$ with respect to the weak operator topology.  
  \item\label{Groth-sp:7} for each space $Z$,  $\mathscr{W}(X,Z)$ is sequentially closed in $\mathscr{B}(X,Z)$ with respect to the strong operator topology.  
\end{enumerate}
\end{theorem}
\begin{proof}
The equivalence between clauses \eqref{Groth-sp:1} and \eqref{Groth-sp:2} follows easily from the description of the operators $T\colon X\to c_0$ by means of weak$^*$-null sequences in the dual space $X^*$; the implications \eqref{Groth-sp:5} $\Rightarrow$ \eqref{Groth-sp:4}  $\Rightarrow$ \eqref{Groth-sp:3} $\Rightarrow$ \eqref{Groth-sp:2} are trivial because WCG Banach spaces have W$^*$SC dual balls  \cite[Theorem XIII.4]{Diestel:84}. \smallskip

For \eqref{Groth-sp:1} $\Rightarrow$ \eqref{Groth-sp:5} suppose that $X$ is Grothendieck, $Y$ has  weak$^*$ sequentially compact dual ball, and $T\colon X\to Y$ is an operator. Every bounded sequence $(f_n)_{n=1}^\infty$ in $Y^*$ has a~weakly* convergent subsequence  $(f_{n_k})_{n=1}^\infty$. Since 
$(T^*f_{n_k})_{n=1}^\infty$ is  weakly* convergent in X$^*$, it is weakly convergent. Consequently $T^*$ is weakly compact, and by Gantmacher's theorem, so is $T$. \smallskip

For \eqref{Groth-sp:1} $\Rightarrow$ \eqref{Groth-sp:6}, let $(T_n)_{n=1}^\infty$ be a sequence of operators in $\mathscr{W}(X,Z)$ such $(T_nx)_{n=1}^\infty$ is weakly convergent in $Z$ for each $x\in X$. By the uniform boundedness principle, $(T_n)_{n=1}^\infty$ is norm-bounded and $Tx = \textrm{weak-lim}_n T_n x$ defines $T\in  \mathscr{B}(X,Z)$. We have to show that $T$ is weakly compact. \smallskip

Let $Y=\overline{\span}\bigcup_{n=1}^\infty T_n(X)$ and $W=\bigcup_{n=1}^\infty 2^{-n} T_n(B_X)$. Then $W$ is a relatively weakly compact subset of $Z$ that generates $Y$, hence $Y$ is WCG. By \eqref{Groth-sp:4}, since $T$ takes values in $Y$, it is weakly compact. \smallskip

The implication \eqref{Groth-sp:6} $\Rightarrow$ \eqref{Groth-sp:7} is trivial. For \eqref{Groth-sp:7} $\Rightarrow$ \eqref{Groth-sp:2}, let us consider an arbitrary operator $T\colon X\to c_0$. Let $P_n$ denote the projection on $c_0$ onto the subspace generated by the $n$ first terms of the unit vector basis ($n\in \mathbb N$), then $P_nT$ is compact and converges to $T$ in the strong topology; hence $T$ is weakly compact. \end{proof}

A companion characterisation of Grothendieck spaces may be stated in terms of convex block subsequences in the dual space (see, \emph{e.g.}, \cite[Proposition 4.2]{ChenKaniaRuan:2021}). 

\begin{prop}\label{prop:convblock}
A Banach space $X$ has the Grothendieck property if and only if every weak*-null sequence in $X^*$ admits a convex block subsequence that converges in norm  to zero.
\end{prop}

The above-stated characterisation of Grothendieck spaces still alludes to the properties of $X^*$ rather than $X$ itself. The first problem stated in \cite{Diestel:73} remains unsolved:

\begin{quest}
Does there exist an internal characterisation of Grothendieck spaces?
\end{quest}
Here `internal' means depending only on the properties of the space $X$, not on those of the dual space $X^*$, or the operators on $X$.\medskip

In general,  convergence in the strong operator topology does not imply convergence in the operator norm. However for Grothendieck $C(K)$-spaces the following result is available (\cite{Coulhon:84}; see also \cite[Theorem 2.1]{BarcenasMarmol:2005}).
\begin{prop}
Let $K$ be a compact space such that $C(K)$ is Grothendieck. Let $(T_n)_{n=1}^\infty$ be a sequence of contractions that converges in the strong operator topology. Then $(T_n)_{n=1}^\infty$ converges in the norm topology.
\end{prop}

In \cite[Problem 8]{Diestel:73},  Diestel asked if the class of Banach spaces $Y$ such that the equality $\mathscr{B}(X,Y)= \mathscr{W}(X,Y)$ holds for every Grothendieck space $X$ coincides with WCG spaces. However, in the Addendum at the end of the paper, he remarked that the space $\ell_1(\Gamma)$ with $\Gamma$ uncountable is a counterexample, and modified the problem as follows. 

\begin{quest}
What class of Banach spaces $Y$ is characterised by the equality  $$\mathscr{B}(X,Y)= \mathscr{W}(X,Y)$$ for every Grothendieck space $X$? 
\end{quest}

For future reference, let us record the following simple permanence properties of Grothendieck spaces. The proofs can be directly derived from the definition of  Grothendieck space.

\begin{prop}\label{basic}
Let $X$ and $Y$ be Banach spaces.
\begin{enumerate}
\item\label{simple0} If $X$ and $Y$ are  Grothendieck spaces, then so is the direct sum $X\oplus Y$.
\item\label{simple1} If $X$ is a Grothendieck space and there exists a surjective operator  $T\colon X\to Y$, then $Y$ is Grothendieck. In particular, quotients and complemented subspaces of Grothendieck spaces are Grothendieck.
\end{enumerate}
\end{prop}

The subsequent result demonstrates that the Grothendieck property is separably determined. 

\begin{prop}\label{basic2}
If every separable subspace $W$ of a Banach space $X$ is contained in a~subspace $Z$ of $X$ that is a~Grothendieck space, then $X$ is a Grothendieck space.
\end{prop}
\begin{proof}
Observe that an operator $T$ is weakly compact if and only if all restrictions of $T$ to separable subspaces of $X$ are weakly compact. Let $T\colon X\to c_0$ be an~operator. Then $T$ restricted to a separable subspace $W$ of $X$ must be weakly compact, because the restriction of $T$ to a Grothendieck superspace of $W$ is weakly compact. The conclusion then follows from the equivalence between \eqref{Groth-sp:1} and \eqref{Groth-sp:2} in Theorem~\ref{Groth-sp}.
\end{proof}

The subsequent result is interesting in its own and will be helpful later. 
\begin{prop}\label{prop:NQc0}
A Banach space $X$ has no quotient isomorphic to $c_0$ if and only if every weakly* convergent sequence in $X^*$ has a weakly Cauchy subsequence.
\end{prop}
\begin{proof}
If there exists a weakly* convergent sequence $(f_n)_{n=1}^\infty$ in $X^*$ which has no weakly Cauchy subsequence then, by Rosenthal $\ell_1$ theorem and passing to a subsequence, we can assume that $(f_n)_{n=1}^\infty$ is equivalent to the usual basis of $\ell_1$. Thus $(g_n)_{n=1}^\infty$ defined by $g_n = f_{2n+1}- f_{2n}$ ($n\in \mathbb N$) is a weak$^*$-null sequence equivalent to the usual basis of $\ell_1$, and $Tx =(\langle x,g_n\rangle)_{n=1}^\infty$ defines a surjective operator from $X$ onto $c_0$. 

Conversely, if $q\colon X\to c_0$ is a quotient map, the adjoint $q^*$  takes the standard unit basis of $\ell_1$ into a weak$^*$-null sequence in $X^*$ without a weakly Cauchy subsequence.
\end{proof}
It is not too difficult to show that if every weakly* convergent sequence in $X^*$ is weakly Cauchy subsequence, then $X$ is Grothendieck.\medskip

The following characterisation of Grothendieck spaces is due to R\"abiger \cite{Rabiger:85}. 
\begin{theorem}\label{th:char-G}
A Banach space $X$ is Grothendieck if and only if $X^*$ is weakly sequentially complete and $X$ has no quotient isomorphic to $c_0$. 
\end{theorem}
\begin{proof}
For the direct implication, note that weakly Cauchy sequences are weakly* convergent, and surjective operators onto $c_0$ are not weakly compact. 

The converse implication easily follows from Proposition \ref{prop:NQc0}. 
\end{proof}

R\"abiger  \cite{Rabiger:85} noticed that for spaces isomorphic to complemented subspaces of a~Banach lattice, being Grothendieck is equivalent to having no quotient isomorphic to $c_0$; let us show that this characterisation extends to spaces having a local unconditional structure.

\begin{prop}
Suppose that $E$ is a Banach space with a local unconditional structure. Then $E$ is Grothendieck if and only if $E$ has no quotient isomorphic to $c_0$.
\end{prop}
\begin{proof}
Suppose that $E$ has a local unconditional structure. This is equivalent to the existence of a Banach lattice $F$ into which $E^{**}$ embeds as a complemented subspace \cite[Theorem 2.1]{FigielJohnsonTzafriri:1975}. Since $E^*$ is complemented in $E^{***}$, it is also complemented in the dual Banach lattice $F^*$. Thus, $E^*$ is weakly sequentially complete if and only if it does not contain any subspace isomorphic to $c_0$ \cite[Theorem 1.c.7]{LT2}.\smallskip

If $E$ has no quotient isomorphic to $c_0$, then $E$ contains no complemented copy of $\ell_1$, hence $E^*$ contains no copies of $c_0$ \cite[Theorem V.10]{Diestel:84}, thus $E^*$ is weakly sequentially complete.
\end{proof}

Let us revisit some consequences of the following result due to  Hagler and Johnson (\cite[Theorem 1]{HaglerJ:77}). 
\begin{theorem}\label{th:HJ77}
Let $X$ be a Banach space. Suppose that $X^*$ has an infinite-dimensional closed subspace without normalised weak$^*$-null sequences. Then $X$ contains a subspace isomorphic to $\ell_1$. 
\end{theorem}

A version of the following result appears in \cite{DiestelS:78}, mentioning that its proof is similar to that of Theorem \ref{th:HJ77}. We can afford a more transparent argument. 
\begin{theorem}\label{th:noWinG}
If $X$ is Grothendieck and $T\colon X\to Y$ is a weakly pre-compact operator, then it is weakly compact.
\end{theorem}
\begin{proof}
Given an operator $T\colon X\to Y$, the DFJP factorisation of $T$ introduced  in \cite{DFJP:74}  provides an intermediate Banach space $\Delta(T)$ and two operators $j\colon \Delta(T)\to Y$ and $A\colon X\to \Delta(T)$ such that $T=jA$. We will require the following two properties of the DFJP  factorisation: 
\begin{enumerate}[label=(\alph*)]
 \item\label{precompacta} $T$ is weakly pre-compact if and only if $\Delta(T)$ contains no copies of $\ell_1$; 
 \item\label{precompactb} $T^*=A^*j^*$ is equivalent to the DFJP factorisation of $T^*$; in particular,  $\Delta(T^*)$ is isomorphic to $\Delta(T)^*$; 
\end{enumerate}
which are \cite[Theorem 2.3]{Heinrich:80} and \cite[Theorem 1.5]{Gonzalez:93}.
\newline\indent
Suppose that $T\colon X\to Y$ is not weakly compact. Then by Gantmacher's theorem, the adjoint $T^*$ is not weakly compact either, so there exists a bounded sequence $(g_i)_{i=1}^\infty$ in $Y^*$ such that  $(T^*g_i)_{i=1}^\infty$ does not have a weakly convergent subsequence. Since weakly Cauchy sequences in dual spaces are weakly* convergent and $X$ is Grothendieck, $(T^*g_i)_{i=1}^\infty$ fails to have weakly Cauchy subsequences; hence $T^*$ is not weakly pre-compact. Passing to a subsequence, by Rosenthal's $\ell_1$-theorem we may suppose that both sequences $(g_i)_{i=1}^\infty$, $(T^*g_i)_{i=1}^\infty$ are equivalent to the unit vector basis of $\ell_1$. In this case, the operator $T^*$ is an isomorphism on the closed subspace $N$ generated by $(g_i)_{i=1}^\infty$. 

Property \ref{precompactb} allows for taking $A^*j^*$ as the DFJP factorisation of $T^*$. Then $j^*(N)$ is an~infinite-dimensional closed subspace of $\Delta(T^*)$ without normalised weak$^*$-null sequences. Indeed, otherwise if $(d_i)_{i=1}^\infty$ is such a sequence and $X$ is Grothendieck,  $(A^*d_i)_{i=1}^\infty$ would be a  semi-normalised weakly null sequence in $T^*(N)$, which is not possible because $T^*(N)$ is isomorphic to $\ell_1$. 
Since $\Delta(T)^*\cong \Delta(T^*)$, Theorem \ref{th:HJ77} implies that $\Delta(T)$ contains a copy of $\ell_1$; hence $T$ is not weakly pre-compact by property \ref{precompacta}. 
\end{proof}
\begin{cor}\label{cor:l1}
Every non-reflexive Grothendieck space contains a copy of $\ell_1$.
\end{cor}

Corollary \ref{cor:l1} was strengthened in \cite{HaydonLevyOdell:87}, where it was proved that under Martin's Axiom and the negation of the Continuum Hypothesis (or, actually under $\mathfrak{p}>\omega_1$), every non-reflexive Grothendieck space contains a copy of $\ell_1(\mathfrak{p})$. We may then record the following result.

\begin{prop}\label{prop:densityp}
A non-reflexive Grothendieck space has density at least $\mathfrak{p}\geqslant \omega_1$.
\end{prop}

When $\mathfrak{p} = \mathfrak{c}$, then one can deduce that every non-reflexive Grothendieck space has $\ell_\infty$ as a quotient.\smallskip

Corollary \ref{cor:l1} implies that \cite[Problem 10]{Diestel:73} has a negative answer: indeed, the dual of a~Grothendieck space cannot be isomorphic to $\ell_1(\Gamma)$. Even more is true, as Haydon proved in \cite{Haydon:87} that the dual of a non-reflexive Grothendieck space contains a copy of the non-separable space $L_1 ([0,1]^{\mathfrak{p}})\equiv L_1 (\{0,1\}^{\mathfrak{p}})$. However, when $\mathfrak{p}<\mathfrak{c}$ one may ask whether a non-reflexive Grothendieck space of density $\mathfrak{p}$ exists. This may be indeed so as proved by Brech \cite{Brech:06}, who used forcing to construct a compact space $K$ for which the space $C(K)$ is Grothendieck and has density $\mathfrak{p}<\mathfrak{c}$. The space $K$ arises as the Stone space of a~certain Boolean algebra, hence it is totally disconnected.\smallskip 

In \cite[Problem 5]{Diestel:73}, Diestel asked  whether a non-reflexive Grothendieck space contains a~copy of $\ell_\infty$.
A negative answer was provided by Haydon in \cite{Haydon:81} (see Example \ref{ex:Haydon}), but a weaker problem is in place.

\begin{quest}\label{Grothc_0}
Does a non-reflexive Grothendieck space contain a copy of $c_0$?
\end{quest}

For spaces with property $(V)$ a convenient characterisation of the Grothendieck property is available (that was noted, \emph{e.g.}, in \cite[Theorem 28]{GhenciuLewis:2012}; a further characterisation may be found in \cite[Corollary 9]{CiliaEmmanuele:2015}). 

\begin{prop}\label{Groth-V-c0}
Let $X$ be a Banach space with property $(V)$. Then $X$ is Grothendieck if and only if it contains no complemented copies of $c_0$.
\end{prop}
\begin{proof}
If $X$ is not Grothendieck, every non-weakly compact operator $T\colon X\to c_0$ is bounded below on some subspace $M$ of $X$ isomorphic to $c_0$. Moreover, $M$ is complemented in $X$ because $T(M)$ is complemented in $c_0$.
\end{proof}
\begin{cor}\label{cor:dualV}
If $X$ is a space with property $(V)$ isomorphic to a complemented subspace of a dual space, then it is a Grothendieck space.
\end{cor}
\begin{proof}
Dual spaces contain no complemented copies of $c_0$.
\end{proof}

We are now going to describe certain characterisations of Banach spaces with the Grothendieck property obtained in \cite{GonzalezG:95} in terms of polynomials. \smallskip

Let $X$ be a Banach space and let $k\in\N$. We denote by $\mathcal{P}(^k X)$ the space of all $k$-homogeneous scalar polynomials defined on $X$. We have $\mathcal{P}(^1 X)=X^*$. The space $\mathcal{P}(^k X)$ endowed with the norm $\|P\|=\sup_{x\in B_X}|P(x)|$ is a dual Banach space. Moreover, every $P\in \mathcal{P}(^k X)$ has a canonical extension $\widehat P\in \mathcal{P}(^k X^{**})$, which is an extension of $P$ by weak$^*$ continuity. \cite[Theorem]{GonzalezG:95} yields a characterisation of Grothendieck spaces in terms of homogeneous polynomials.

\begin{prop} 
For a Banach space $X$, the following assertions are equivalent:
\begin{enumerate}[label=(\alph*)]
\item $X$ has the Grothendieck property; 
\item\label{polynomialb} for every $k\in\N$, given a sequence $(P_n)\subset  \mathcal{P}(^k X)$ with $P_n(x)\to 0$ for all $x\in X$, then $\widehat P_n(x^{**})\to 0$ for all $x^{**} \in X^{**}$;
\item the same statement as \ref{polynomialb} is true for some $k$. 
\end{enumerate}
\end{prop}
Further characterisations of the Grothendieck property in terms of vector-valued  homogeneous polynomials $\mathcal{P}(^k X,Y)$ may be found in \cite{GonzalezG:95}. 

Some characterisations in terms of weak-Riemann integrability of weak*-continuous functions are also available (\cite{Lone:2017}).

\section{Further properties of Grothendieck spaces} 

We begin with an auxiliary lemma from \cite[Theorem 2.7]{GO:86} that can be seen as a `lifting result for sequences' associated to the Grothendieck property.
\begin{lemma}\label{liftingG}
Let $M$ be a closed subspace of a Banach space $X$ such that $X/M$ has the Grothendieck property, let $J\colon M\to X$ denote the embedding, and let $J^*\colon X^* \to M^*$ be its adjoint. Then every weakly* convergent sequence $(f_n)_{n=1}^\infty$ in $X^*$ such that $(J^*f_n)_{n=1}^\infty$ is weakly convergent has a weakly convergent subsequence $(f_{n_k})_{k=1}^\infty$.
\end{lemma}

The following result was obtained in \cite[Corollary 2.8]{GO:86}. We shall employ it for constructing new examples of Grothendieck spaces. 

\begin{prop}\label{prop:3SP} 
The class of Grothendieck spaces has the three-space property.
\end{prop}
\begin{proof}
Suppose that a Banach space $X$ has a closed subspace $M$ such that both $M$ and $X/M$ are  Grothendieck spaces.
If $J\colon M\to X$ denotes the embedding, then the adjoint operator $J^*\colon X^*\to M^*$ is surjective by the Hahn--Banach theorem. 

Given a weakly* convergent sequence $(f_n)_{n=1}^\infty$ in $X^*$, $(J^*f_n)_{n=1}^\infty$ is weakly* convergent; hence weakly convergent in $M^*$ because $M^*$ is Grothendieck.
Since $X/M$ is Grothendieck, each subsequence of  $(f_n)_{n=1}^\infty$ has a weakly convergent subsequence by Lemma \ref{liftingG}, hence $(f_n)_{n=1}^\infty$ is weakly convergent, and we conclude that $X$ is Grothendieck.
\end{proof}

As we will see later, many Grothendieck spaces satisfy the following property that we term $(V_\infty)$.

\begin{definition}
A Banach space $X$ \emph{has property $(V_\infty)$} when every non-weakly compact operator $T\colon X\to Y$ is an isomorphism on a subspace of $X$ isomorphic to $\ell_\infty$.
\end{definition}

Spaces with property $(V_\infty)$ satisfy the following results.

\begin{prop}
Let $X$ and $Y$ be Banach spaces.
\begin{romanenumerate}
\item\label{v-infty:1} If $X$ has property $(V_\infty)$, then it is Grothendieck.
\item\label{v-infty:2} If $X$ and $Y$ have property $(V_\infty)$, then so has $X\oplus Y$.
\item\label{v-infty:3} If $X$ has property $(V_\infty)$ and there exists a surjective operator $T\colon X\to Y$, then $Y$ has property $(V_\infty)$. In particular, quotients and complemented subspaces of spaces with property $(V_\infty)$ have also this property.
\end{romanenumerate}
\end{prop}
\begin{proof}
\eqref{v-infty:1} is a direct consequence of Theorem \ref{Groth-sp}, and the proofs of \eqref{v-infty:2} and \eqref{v-infty:3} are easy.
\end{proof}

The space $C[0,1]$ has property $(V)$ but it is not Grothendieck because it contains complemented copies of $c_0$. (Actually, a separable Banach space is Grothendieck if and only if it is reflexive). That cannot happen for dual spaces:

\begin{prop}\label{prop:V-dual}
If $X$ is isomorphic to a complemented subspace of a dual space
and it has property $(V)$, then $X$ has property $(V_\infty)$. In particular, $X$ is Grothendieck.
\end{prop}
\begin{proof}
Suppose that a space $X$ is complemented in a dual space and has property $(V)$, and let $T\colon X\to Y$ be a non-weakly compact operator. Then $T$ is is an~isomorphism on a~copy of $c_0$ in $X$, hence it is an~isomorphism on a copy of $\ell_\infty$ in $X$ by a result of Rosenthal \cite[Theorem 1.3]{Rosenthal:70}.
\end{proof}

It was proved by Pfitzner (\cite{Pfitzner:94}; see Theorem~\ref{thm:pf}) that $C^*$-algebras have property $(V)$ (see \cite{FPPeralta:2010} for an alternative proof). Consequently, the previous result implies that the space $\mathscr{B}(H)$ of bounded operators on a Hilbert space $H$ (and every other von Neumann algebra) is a Grothendieck space. Thus \cite[Problem 2]{Diestel:73} has positive answer. It may look coincidental that known examples of Grothendieck spaces have property $(V)$, hence it is natural to ask if it is always the case. This problem, which we record below, was originally raised by Diestel in \cite{Diestel:80}.

\begin{quest}\label{Q:property_V}
Do Grothendieck spaces have property $(V)$?
\end{quest}

We observe that property $(V)$ is not a three-space property \cite{Cast-Gonz:94}. Moreover, in light of Proposition \ref{prop:V-dual}, one may propose a weaker problem.

\begin{quest}\label{dualGrothV}
Do dual spaces with the Grothendieck property have property $(V)$?
\end{quest}

In view of Theorem \ref{th:char-G}, the next  result could be helpful to solve Problem \ref{Q:property_V}.

\begin{prop}\label{prop:Ch-V}
A Banach space $X$ has property $(V)$ if and only if $X^*$ is weakly sequentially complete and for every unconditionally converging operator $T\colon X\to Y$, the adjoint $T^*$ is weakly pre-compact. 
\end{prop}
\begin{proof}
Suppose that $X$ has property $(V)$. Then $X^*$ is weakly sequentially complete by  \cite[Propositions 3.3.D and 3.3.F]{HarmandWW:93}.  
Moreover, if $T\colon X\to Y$  is not an isomorphism on a~copy of $c_0$, then $T$ is weakly compact; hence $T^*$ is weakly compact and cannot act as an~isomorphism on a copy of $\ell_1$.

Conversely, if $X$ fails property $(V)$, then there exists a non-weakly compact operator $T\colon X\to Y$ which  is not an isomorphism on any copy of $c_0$ in $X$. Since $T^*$ is not weakly compact, there exists a bounded sequence $(g_n)_{n=1}^\infty$ in $Y^*$ such that $(T^*g_n)_{n=1}^\infty$ has no weakly convergent subsequence. Thus  $X^*$ is not weakly sequentially complete or $(T^*g_n)_{n=1}^\infty$ has no weakly Cauchy subsequence. In the latter case, $(g_n)_{n=1}^\infty$ has a subsequence generating a~subspace isomorphic to $\ell_1$ on which $T^*$ is an isomorphism, and the result is proved.
\end{proof}


R\"abiger in \cite{Rabiger:85} gave the following partial positive answer to  Problem \ref{Q:property_V}. 

\begin{prop}\label{prop:Gr->V}
Every Grothendieck space $X$  isomorphic to a complemented subspace of a~Banach lattice $E$ has property $(V)$.
\end{prop}
\begin{proof}
Since $X^*$ is weakly sequentially complete (Theorem \ref{th:char-G}), we may suppose that the norm in $E^*$ is order continuous (\cite[Theorem 5.1.15]{Meyer-Nieberg:91}). Let $P\colon E\to E$ be a projection onto $X$. 

By Proposition \ref{prop:Ch-V}, it is enough to show that for every unconditionally converging operator  $T\colon X\to Y$, the adjoint $T^*$ is weakly pre-compact. But $TP$  unconditionally converging implies that $P^*T^*$ is weakly pre-compact (\cite[Theorem 3.4.20]{Meyer-Nieberg:91}, hence $T^*$ is weakly pre-compact.  \end{proof}

The subsequent results characterise reflexivity of dual spaces with property $(V)$.

\begin{prop}\label{prop:XandX*V} If $X$ is a Grothendieck  space and $X^*$ has property $(V)$, then $X$ is reflexive.
\end{prop}
\begin{proof}
If $X^*$ is non-reflexive and has property $(V)$, then it contains a copy of $c_0$ \cite[Chapter III, Corollary 3.3.c]{HarmandWW:93}. Thus $X^*$ is not weakly sequentially complete, hence $X$  is not Grothendieck).
\end{proof}



The following problem was posed by Diestel \cite[Problem 6]{Diestel:73} is then naturally related with the previously presented ones.

\begin{quest}\label{XandX*-Gr}
Suppose that $X$ and $X^*$ are  Grothendieck. Is $X$ reflexive?
\end{quest}

However, this would be indeed the case had Problem~\ref{dualGrothV} had affirmative solution.\smallskip

Diestel \cite[Problem 7]{Diestel:73} also asked if $X$ is reflexive when both $X$ and $X^*$ are weakly sequentially complete. A negative answer was provided by Bourgain and Delbaen in \cite{BD:80}, by showing the existence of an infinite-dimensional  $\Lc_\infty$-space satisfying the Schur property.\smallskip 



For each compact Hausdorff space $K$, the bidual $C(K)^{**}$ is an injective space, hence it is Grothendieck. More generally, for every $C^*$-algebra $A$, the bidual $A^{**}$ is a von Neumann algebra, hence a Grothendieck space. In this light, Diestel's \cite[Problem 4]{Diestel:73} appears even more natural:

\begin{quest}\label{bidual-Gr}
Let $X$ be a Grothendieck space. Is $X^{**}$ Grothendieck?
\end{quest}
\begin{remark}\label{nonsums}
It is to be noted that in general property $(V)$ does not pass to biduals, because $X = (\bigoplus_{n\in\mathbb N} \ell_1^n)_{c_0}$ has property $(V)$, yet $X^{**} = (\bigoplus_{n\in\mathbb N} \ell_1^n)_{\ell_\infty}$ contains a complemented copy of $\ell_1$ \cite[p.~303]{Johnson:84}, hence it fails property $(V)$.\end{remark}
Curiously, even if we knew that (dual) Grothendieck spaces have property $(V)$, it would not be automatic that biduals of such spaces still had property $(V)$.

\section{M-bases and Schauder decompositions} 
The present section is devoted to results showing that the Grothendieck property prevents Banach spaces from admitting certain structures generalising Schauder bases. 
\newline\indent 
The following result is due to W.~B.~Johnson \cite{Johnson:70}. 
\begin{theorem}\label{th:johnson}
Every Grothendieck space admitting an M-basis is reflexive. 
\end{theorem}
\begin{proof}
Let $\{(x_i,f_i)\colon i\in I\}$ be an M-basis in a Banach space $X$, and let $Y$ denote the closed subspace of $X^*$ generated by $\{f_i\colon i\in I\}$. It is enough to show that $Y$ is reflexive. Indeed, since $Y$ is total over $X$, $Y$ is weak$^*$-dense in $X^*$. Therefore, if the unit ball $B_Y$ is weakly compact, then it is $\sigma(X^*,X)$-compact, so that it follows from the Krein--Smulian theorem \cite[Theorem 2.7.11]{Megginson} that $Y$ is weak$^*$-closed, hence $Y=X^*$. 
\newline\indent
Let $(y_n)_{n=1}^\infty$ be a sequence in $B_Y$. Since each $y\in Y$ is the norm limit of a sequence in ${\textrm{span}}\{f_i\colon i\in I\}$, for each $n$ the set $A_n=\{i\in I\colon \langle y_n, x_i\rangle\neq 0\}$ is countable; thus so is $\bigcup_{n=1}^\infty A_n$. A standard diagonalisation argument shows that there exists a subsequence $(y_{n_k})_{k=1}^\infty$ such that $(\langle y_{n_k}, x_i\rangle)_{k=1}^\infty$ converges for each $i\in I$. Since $(y_{n_k})_{k=1}^\infty$ is bounded and $\{x_i\colon i\in I\}$ generates a dense subspace of $X$, $(\langle y_{n_k}, x\rangle)_{k=1}^\infty$ converges for each $x\in X$. Thus $(y_{n_k})_{k=1}^\infty$ is weak$^*$-convergent to some $y\in X^*$. Since $X$ is Grothendieck, $(y_{n_k})_{k=1}^\infty$ is weakly convergent to $y$. Note that $y\in Y$ because $Y$ is (weakly) closed. Thus $Y$ is reflexive and the proof is complete.
\end{proof}

In relation with Theorem \ref{th:johnson}, Diestel \cite[Problem 9]{Diestel:73} asked if a Banach space  with an M-basis can contain $\ell_\infty$. The answer is that it can contain $\ell_\infty$, but not $\ell_\infty(\Gamma)$ for any uncountable set $\Gamma$ (see \cite[Corollary 5.4]{HajekMVZ:08}).
\medskip

Subsequently, we discuss the non-existence of Schauder decompositions in Banach spaces with the Grothendieck property. Observe that $\ell_2$ and  $\ell_2(\ell_\infty)$ are Grothendieck spaces (see Proposition \ref{prop:unc-sum}) admitting natural Schauder decompositions, so we have to impose an additional condition to that of being Grothendieck that would prevent the existence of a Schauder decomposition. 

\begin{prop}\label{prop:SchD1}
Let $(P_i)_{i=1}^\infty$ be a Schauder decomposition in a Banach space $X$ with the Grothendieck property. Then the sequence of adjoint projections $(P_i^*)_{i=1}^\infty$ forms a~Schauder decomposition in $X^*$. 
\end{prop}\begin{proof}
Let $S_n = P_1+\cdots+P_n$ ($n\in\N$). Then $S_n x\to x$ as $n\to \infty$ for each $x\in X$. Therefore, given $x\in X$ and $f\in X^*$, 
$$
\lim_{n\to \infty} \langle S^*_n f,x\rangle = \lim_{n\to \infty} \langle f,S_n x \rangle =\langle f,x \rangle. 
$$
Since $(S^*_n f)_{n=1}^\infty$ is weakly* convergent to $f$ and the space $X$ is Grothendieck, 
$(S^*_n f)_{n=1}^\infty$ is weakly  convergent to $f$; and to finish the proof, it is enough to show that it is norm convergent. 

The union of ranges $D=\bigcup_{n=1}^\infty S_n^*(X^*)$ is a dense subspace of $X^*$ such that $(S^*_n g)_{i=1}^\infty$ is norm  -convergent to $g$ for each $g\in D$, and  $C=\sup_{n\in N}\|S_n\|< \infty$ by the Banach--Steinhaus theorem. Thus, given $f\in X^*$, since we can choose $g\in D$ arbitrarily close to $f$, the inequality 
$$
\|S_k^* f-S_l^* f\|\leqslant \|S_k^*(f -g)\|+ \|S_k^* g-S_l^* g\|+ \|S_l^*(g -f)\|
$$
shows that $(S^*_n f)_{n=1}^\infty$ is norm-convergent to $f$. 
\end{proof}
The next two results will be the key in the proof of Proposition \ref{prop:SchD2}. 
\begin{lemma}\label{lem:seq}
Let $(P_i)_{i=1}^\infty$ be a Schauder decomposition in a Banach space $X$, and let $(x_i)_{i=1}^\infty$ be a sequence with $x_i \in P_i(X)$ for each $i\in\N$. Then  $(x_i)_{i=1}^\infty$ is a basic sequence. 
\end{lemma}
\begin{proof}
This is a standard argument. If  $x\in\overline{\span}\{x_i\colon i\in\N\}$, then $x=\sum_{i=1}^\infty P_i(x)$; so it is enough to show that $P_i(x)$ is a multiple of $x_i$ for each $i$.
\newline\indent
Otherwise, for some $j\in\N$ we could  find $f\in P_j^*(X^*)$ such that $\langle f,x_j\rangle=0$, however $\langle f,P_jx\rangle\neq 0$. 
Since $\langle f,x_i\rangle=0$ for each $i\in\N$, $0=\langle f,x\rangle= \langle P_j^*f,x\rangle= \langle f,P_jx\rangle$; a~contradiction.
\end{proof}
\begin{lemma}\label{lem:Grot}
Let $(P_i)_{i=1}^\infty$ be a Schauder decomposition in a Banach space $X$ with the Grothendieck property. Then 
\begin{romanenumerate}
 \item\label{lem:Grot:1} every bounded sequence  $(x_i)_{i=1}^\infty$ in $X$ with $x_i \in P_i(X)$ for each $i\in\N$ is weakly null. 
 \item\label{lem:Grot:2} every bounded sequence $(f_i)_{i=1}^\infty$ in $X^*$ with $f_i \in P_i^*(X^*)$ for each $i\in\N$ is weakly null, and $\overline{\span}\{f_i\colon i\in\N\}$ is a reflexive subspace. 
\end{romanenumerate}
\end{lemma}
\begin{proof}
\eqref{lem:Grot:1} Recall that  $S_n = P_1+\cdots+P_n$. If $x_i\in P_i(X)$, then $x_i=(I-S_{i-1})x_i$ for $i>1$. Thus, for every $f\in X^*$ we have
$$
|\langle f, x_i\rangle|= |\langle (I-S_{i-1}^*)f, x_i\rangle|\leqslant \|(I-S_{i-1}^*)f\|\cdot\|x_i\|,
$$
and, by Proposition \ref{prop:SchD1}, $ \|(I-S_{i-1}^*)f\|\to 0$ as $i\to\infty$. Consequently,  $(x_i)_{i=1}^\infty$ is a weakly null sequence. 
\newline\indent
\eqref{lem:Grot:2} Proceeding as in \eqref{lem:Grot:1}, we can show that $(f_i)_{i=1}^\infty$ is weak$^*$ null, hence it is weakly null by the Grothendieck property of $X$. 

On the other hand, if $\overline{\span} \{f_i\colon i\in\N\}$ were not reflexive, then, by Rosenthal's $\ell_1$-theorem, it would contain a subspace isomorphic to $\ell_1$. Since $(f_i)_{i=1}^\infty$ is a basic sequence by Lemma \ref{lem:seq}, we could construct a bounded sequence of successive blocks $g_j =\sum_{i\in A_j}c_i f_i$ ($j\in N, A_j\subset \N$) with no weakly convergent subsequence. Taking $Q_j =\sum_{i\in A_j}P_i$ we obtain another Schauder decomposition $(Q_j)_{j=1}^\infty$ of $X$. Since $g_j\in Q^*_j(X^*)$ for each $j$, we would get that $(g_j)_{j=1}^\infty$ is weakly null; a contradiction. 
\end{proof}
The following result was proved by Dean \cite{Dean:67}. To see that Dean's statement and ours are equivalent, observe that it is an easy exercise to show that a Banach space $Z$ has the Dunford--Pettis property if and only if given weakly compact operators $S\colon Z\to Y$ and $T\colon X\to Z$, the product $ST$ is compact.  
\begin{prop}\label{prop:SchD2}
Grothendieck spaces with the Dunford--Pettis property do not admit Schauder decompositions.
\end{prop}
\begin{proof}
If $X$ is Grothendieck and $(P_i)_{i=1}^\infty$ is a Schauder decomposition in $X$, then we can select a normalised sequence $(f_i)_{i=1}^\infty$ in $X^*$ with $f_i\in P^*_i(X^*)$. Since $P^*_if_i=f_i$, we can find a bounded sequence $(x_i)_{i=1}^\infty$ in $X$ with $x_i\in P_i(X)$ such that $\langle f_i, x_i\rangle=1$ for each $i\in\N$. By Lemma \ref{lem:Grot}, both sequences $(x_i)_{i=1}^\infty$ and $(f_i)_{i=1}^\infty$ are weakly null; hence $X$ fails the Dunford--Pettis property. Indeed, $Tx =(\langle f_i, x\rangle)_{i=1}^\infty$ defines a weakly compact operator $T\colon X\to c_0$ and $(Tx_i)_{i=1}^\infty$ does not converge in norm to $0$.
\end{proof}



We say that a Banach space $X$ has the \emph{surjective Dunford--Pettis property} if every surjective operator from $X$ onto a reflexive Banach space takes weakly convergent sequences into convergent sequences. Clearly, spaces with the Dunford--Pettis property have the surjective Dunford--Pettis property, but the example described in Proposition \ref{prop:ex-Leung} shows that the converse implication fails.

Leung \cite{Leung:88} improved Proposition \ref{prop:SchD2} as follows.  
\begin{prop}\label{prop:SchD3}
Grothendieck spaces with the surjective Dunford--Pettis property do not admit Schauder decompositions.
\end{prop}
\begin{proof}
Assume that $X$ is a Grothendieck space that has a Schauder decomposition $(P_i)_{i=1}^\infty$. As in the proof of Proposition \ref{prop:SchD2}, we select bounded sequences $(f)_{i=1}^\infty$ and $(f_i)_{i=1}^\infty$ with $f_i \in P^*_i(X^*)$,  $x_i\in P_i(X)$ and  $\langle f_i, x_i\rangle=1$ for each $i\in\N$. 
\newline\indent
By Lemma \ref{lem:Grot}, $N=\overline{\span} \{f_i\colon i\in\N\}$ is reflexive, hence $(^\perp N)^\perp = N$ so $X/^\perp N$ is   reflexive too. Let $Q\colon X\to X/^\perp N$ denote the quotient map. Then the sequence  $(x_i)_{i=1}^\infty$ is weakly null, but $(Qx_i)_{i=1}^\infty$ does not converge in norm to $0$, because $\langle f_i,Qx_i\rangle=1$ for each $i\in\N$. Consequently, $X$ fails the surjective Dunford--Pettis property.
\end{proof}

\section{Norm-attaining functionals on Grothendieck spaces}

For a Banach space $X$, we consider  $$\textrm{At}(X) = \{f\in X^*\colon \| f \| = \langle f,x \rangle \text{ for some }x\in S_X\},$$
the set of all norm-attaining functionals in $X^*$. A well-known result of James establishes that $X$ is reflexive if and only if $\textrm{At}(X)=X^*$. 

Debs, Godefroy, and Saint-Raymond \cite{DGSR:1995} proved that $\textrm{At}(X)$ is not a weak$^*$-$G_\delta$ subset of $X^*$ when $X$ is separable and non-reflexive. Acosta and Kadets extended this result (\cite[Theorem 2.5]{AcostaKadets:2011}) as follows. 

\begin{prop}\label{prop:At}
If $X$ is a Banach space and ${\rm At}(X)$ is a weak$^*$-$G_\delta$ subset of $X^*$, then $X$ is Grothendieck. 
\end{prop}

We do not know if the sufficient condition for being a Grothendieck space, presented in Proposition \ref{prop:At}, is actually also necessary. 
\begin{quest}\label{pr:8}
Let $X$ be a Grothendieck space. Is ${\rm At}(X)$ a weak$^*$-$G_\delta$ subset of $X^*$?
\end{quest}

\noindent \emph{\textbf{Added in proof.}} As explained to us by the referee, Problem~\ref{pr:8} has negative answer. We are indebted for her/his permission to include the argument here.

Let us consider the renorming of $\ell_\infty$ given by
\[
    \|(\xi_k)_{k=1}^\infty \|^\prime =  \|(\xi_k)_{k=1}^\infty \|_{\ell_\infty} + \limsup_{k\to \infty} |\xi_k|\quad \big( (\xi_k)_{k=1}^\infty \in \ell_\infty \big), 
\]
and let us denote $X_0=(\ell_\infty, \|\cdot\|^\prime)$. 

It is easy to see that the restriction of the norm in $X_0^*$ to $\ell_1$ coincides with the canonical norm on $\ell_1$, and that an element of $\ell_1$, regarded as a functional on $X_0$, attains its norm if and only if it is finitely supported. 
It follows from the Baire category theorem that the set of norm-attaining functionals in $X_0^*$ is not $G_\delta$ in the norm topology, hence not in the weak* topology either.

\chapter{Classes of Banach spaces with the Grothendieck property}\label{chap:examples}

\section{Spaces of continuous functions} \label{sect:ck}
It was Grothendieck himself \cite{Grothendieck:53} who proved that for a discrete set $\Gamma$, $\ell_\infty(\Gamma) \equiv C(\beta \Gamma)$ is a Grothendieck space. It follows immediately from the fact that if $K$ is \emph{Stonean space}, that is, compact and extremally disconnected, then $C(K)$ is injective (hence complemented in $\ell_\infty(\Gamma)$, where $\Gamma$ is the unit ball of $C(K)^*$), so also Grothendieck. 

We will use the following terminology introduced by Seever \cite{Seever:68}. 
\begin{definition}
A compact (Hausdorff) space $K$ is a \emph{G-space} whenever $C(K)$ has the Grothendieck property. 
\end{definition}

Thus we may say that Stonean  spaces are G-spaces. Note that if $K$ is a G-space and $L$ is a continuous image of $K$, then $L$ need not be a G-space. Indeed, $\N_\infty$ is a continuous image of $\beta\N$; and more generally,  every compact space $K$ is a continuous image of $\beta K_d$, where $K_d$ is $K$ endowed with the discrete topology. Actually, every Stonean space $K$ continuously surjects onto a space $L$, which is not a G-space, yet every weakly* convergent sequence of purely atomic measures on $L$ is weakly convergent \cite[Theorem 7.3]{KakolSobotaZdomsky:2020}.\smallskip

Convergent sequences in a compact space $K$ that are non-trivial, in the sense that they are not eventually constant, give rise to complemented copies of $c_0$, hence a G-space cannot have non-trivial convergent sequences. Such a condition is however not sufficient as the product of two infinite compact spaces $K$ and $L$ is never a G-space (a more general variant of this result is discussed in Section~\ref{Injten}; and a further strengthening may be found in \cite[Theorem 11.3]{KakolSobotaZdomsky:2020}; see also \cite{KakolMarciszewskiSobotaZdomsky:2020}). In \cite[Theorem 7]{KakolMolto:2020} it was noticed that $C(K)$ is \emph{not} a Grothendieck space if and only if it is isomorphic to a space $C(L)$ with $L$ containing a~non-trivial convergent sequence. A strengthening of this observation is available and may be found. \emph{e.g.}, in \cite[Proposition 6.12]{Koszmider:2010}.

\begin{prop}
Let $K$ be a compact space. Suppose that $(\mu_n)_{n=1}^\infty$ is a weakly* convergent sequence in $C(K)^*$ with no weakly convergent subsequence. Then $C(K)$ contains a~complemented subspace isomorphic to $c_0$.
\end{prop}
\begin{proof}
Let $(\mu_n)_{n=1}^\infty$ be a weakly* convergent sequence in $M(K)$, the space of Borel measures on $K$ identified with the dual of $C(K)$, that does not have any weakly convergent sequence. As $(\mu_n)_{n=1}^\infty$ does not have weakly convergent subsequences, it follows from the Eberlein--Smulian theorem together with the Dieudonn\'e--Grothendieck theorem (\cite[Theorem 14 in Chapter VII]{Diestel:84}) that there are pairwise disjoint open subsets $U_n\subset K$ ($n\in \mathbb N$) and $\delta > 0$ such that $|\mu_n(U_n) | \geqslant \delta$. Since $(\mu_n)_{n=1}^\infty$ is weakly* convergent (to some measure $\mu$), we may assume passing to a subsequence if necessary, that $|(\mu-\mu_n)(U_n)| > \delta / 2$. Using Urysohn's lemma, for some $M > 0$ and each $n\in \mathbb N$ we may find a function $f_n\in C(K)$ of norm at most $M$ whose support is contained in $U_n$ and $\langle \mu-\mu_n, f_n \rangle = 1$. Since the functions $f_n$ have pairwise disjoint supports $(n\in\mathbb N)$, they span an isometric copy of $c_0$. Moreover, the expression 
\[
    Pf = \sum_{n=1}^\infty \langle \mu-\mu_n, f\cdot \mathds{1}_{U_n}\rangle f_n\quad \big(f\in C(K)\big)
\]
defines a projection on $C(K)$ onto $\overline{\span}{\{f_n\colon n\in \mathbb N\}}$.
\end{proof}

In an unpublished note, Plebanek proved that there exists a~non-separable compact Hausdorff space $K$ that is not a G-space, yet every closed, separable subspace $L\subset K$ is a~G-space (an~exposition of this construction may be found in \cite{Bielas:11} or \cite{KakolMolto:2020}).\medskip


As already mentioned, every Stonean space is a G-space, because $C(K)$ is then complemented in $C(\beta \Gamma)$ for some set $\Gamma$. (A compact space $K$ is Stonean if and only if disjoint open sets in $K$ have disjoint closures; Stonean spaces are precisely Stone spaces of complete Boolean algebras.)\smallskip

P\'erez Hern\'andez asked during the Winter School in Abstract Analysis 2017 held in Svratka, Czech Republic, for a characterisation of filters $\mathscr F$ on $\mathbb N$ for which the space $c_{0,\mathscr F}$ (see Section \ref{sect:C(K)}) is complemented in $\ell_\infty$ (hence isomorphic to $\ell_\infty$ by \cite{Lindenstrauss:67}). Leonetti \cite{Leonetti:2018} proved that for every filter $\mathscr F$ such that there exists an uncountable family $\mathscr B\subset \mathscr{P}(\mathbb N)\setminus \mathscr{F}$ with the property that for any two distinct sets $N,M\in \mathscr{B}$ the union $N\cup M$ is in $\mathscr{F}$, the subspace $c_{0,\mathscr{F}}$ is \emph{not} complemented in $\ell_\infty$. In \cite{Kania:2019}, the result was strengthened and the question of whether the space $c_{0,\mathscr{F}}$ for a filter satisfying the property distilled by Leonetti is \emph{not} a~Gro\-then\-dieck space was asked. However, for the intersection of finitely many ultrafilters, the corresponding space has finite codimension in $\ell_\infty$, and as such, it is a~Grothendieck space. 

\begin{quest}
 Characterise filters $\mathscr{F}$ for which $c_{0,\mathscr{F}}$ is a Grothendieck space.
\end{quest}

In 1964, Lindenstrauss \cite{Lindenstrauss:64} studied the $C(K)$-spaces with the property that every operator from $C(K)$ to a separable Banach space is weakly compact; hence, Grothendieck $C(K)$-spaces, and---using Seever's terminology---he noticed that  F-spaces are G-spaces. The same result was later independently found again by Seever \cite{Seever:68}.

\begin{remark}\label{rem:list}Let us list some further examples of G-spaces:
\begin{itemize}
\item $\sigma$-Stonean spaces (Stone spaces of $\sigma$-complete Boolean algebras): \cite{Ando:61}.
\item Basically disconnected spaces, also known as Rickart spaces (open $\sigma$-compact subsets have open closures): \cite{Semadeni:1964}.
\item F-spaces (disjoint open $F_\sigma$-sets have disjoint closures; equivalently, the Banach lattice $C(K)$ has the countable monotone interpolation property): \cite{Lindenstrauss:64, Seever:68}; see also \cite[Theorem 4.6]{SmithWilliams:86} and \cite{SmithWilliams:88}.
\item {Weakly Koszmider spaces}, which are infinite compact spaces for which every operator $T\colon C(K)\to C(K)$ is \emph{centripetal}, that is, for any bounded sequence $(f_n)_{n=1}^\infty$ of disjointly supported functions in $C(K)$ and any sequence $(x_n)$ in $K$ with $f_n(x_n) = 0$ for all $n\in \mathbb N$, we have $ (Tf_n)(x_n) \to 0$ as $n\to \infty$. (See \cite[Theorem  4.4]{Schlackow:2008} for the proof that weakly Koszmider compact spaces  are G-spaces.) 

When $K$ is a weakly Koszmider space such that $K\setminus F$ is connected for any finite set $F\subset K$, the Banach space $C(K)$ is \emph{indecomposable} in the sense that each complemented subspace of $C(K)$ is either finite-dimensional or has finite codimension; in particular such spaces do not contain complemented copies of $c_0$. 

In the literature one can find several constructions of (weakly) Koszmider spaces; see \emph{e.g.}, \cite{Koszmider:04, KoszmiderShelahSwietek:2018, Plebanek:04}.
\end{itemize}
\end{remark}

We remark in passing that there exist connected F-spaces (\emph{e.g.}, the \v{C}ech--Stone remainder $\beta [0,\infty)\setminus [0,\infty)$ as observed by Seever \cite{Seever:68}), hence not every G-space arises as a Stone space of a certain Boolean algebra. Such observations led Diestel  \cite[Problem 3]{Diestel:73} to restate a  problem posed by Lindenstrauss in  \cite[p.~224]{Lindenstrauss:64} as follows:

\begin{quest}
Is there an intrinsic characterisation of G-spaces? More precisely, can G-spaces be characterised topologically?
\end{quest}

A topological space $X$ is a $\Delta$-\emph{space}, whenever for every non-increasing sequence $(D_n)_{n=1}^\infty$ of subsets of $X$ with empty intersection, there exists a non-increasing sequence $(V_n)_{n=1}^\infty$ of open subsets of $X$ whose intersection is empty and $D_n\subseteq V_n$ for every $n\in \mathbb N$. The notion of a $\Delta$-space was introduced by Knight \cite{Knight:93}. K\k{a}kol and Leiderman observed that infinite compact $\Delta$-spaces are \emph{not} G-spaces (\cite[Corollary 3.14]{KakolLeiderman:2021}).\smallskip

Next we briefly describe an important example due to Haydon \cite{Haydon:81}.
\begin{example}\label{ex:Haydon}
In ZFC, there exists a G-space $K$ such that $C(K)$ does not contain subspaces isomorphic to $\ell_\infty$, however it admits a quotient isomorphic to $\ell_\infty$.
\end{example}
The compact space $K$ is the Stone space of a certain algebra $\mathscr{A}$ of subsets of $\N$ that  contains the finite sets; the space $C(K)$ may be identified with a closed sub-$C^*$-algebra of $\ell_\infty$ generated by the indicator functions of the sets in $\mathscr{A}$. As $\mathscr{A}$ contains the finite sets, the algebra contains the natural copy of $c_0$ in $\ell_\infty$. \medskip

Assuming the Continuum Hypothesis, Talagrand constructed in  \cite{Talagrand:80} a G-space $L$ such that $C(L)$ does not admit any quotients isomorphic to $\ell_\infty$. The latter condition is equivalent to the fact that  $C(L)$ does not contain subspaces isomorphic to $\ell_1(\Gamma)$ with $\Gamma$ uncountable. \medskip

Let $(\Omega, \Sigma, \mu)$ be a~measure space. Then the space $L_\infty(\mu)$ of all $\mu$-essentially bounded scalar-valued functions on $\Omega$ is naturally a (complex) Banach lattice (even an AM-space) as well as a commutative $C^*$-algebra, where in either case the lattice/algebra operations are defined pointwise. It follows from  Kakutani's representation theorem (\cite[Theorem 1.b.3]{LT2}) for AM-spaces, or from the Gelfand--Naimark theorem for commutative $C^*$-algebras that $L_\infty(\mu)$ is isometric to a $C(K)$-space. When $L_\infty(\mu)$ is naturally representable as the dual space of $L_1(\mu)$, it is injective, hence Grothendieck because it is complemented in $\ell_\infty(\Gamma)$ for some set $\Gamma$.
\smallskip

There is an exact measure-theoretic condition for $\mu$ characterising when $L_\infty(\mu)$ is a~dual space; measures with this property are called \emph{strictly localisable}, however we do not require to invoke the details of this condition here.

\begin{example}\label{ex:Pel-Sud}
Let $\nu$ be the counting measure on an uncountable set $\Gamma$. Then $L_\infty(\nu)$ is \emph{not} a dual space. Indeed, $L_\infty(\nu)$ is the linear span of  $\ell_\infty^c(\Gamma)$, the subspace of all countably supported functions on $\Gamma$, and the constant function $\mathds{1}$ in $\ell_\infty(\Gamma)$. 

The space $L_\infty(\nu)$ is not injective, but it is Grothendieck. 
\end{example}

That $\ell_\infty^c(\Gamma)$ is a non-injective Grothendieck space was first remarked by Pe{\l}czy\'nski and Sudakov (\cite[p.~87]{PelczynskiSudakov:1962}; see also \cite[Proposition 3.7]{JohnsonKaniaSchechtman:2016}).

\begin{prop}\label{prop:Linfty}For every (non-negative) measure $\mu$, the space $L_\infty(\mu)$ is Grothendieck.
\end{prop}
\begin{proof} 
As $L_\infty(\mu)$ is lattice-isometric to $C(K)$ for some compact space $K$, and every bounded disjoint sequence in $L_\infty(\mu)$ has a supremum, the same happens in $C(K)$. This means that $C(K)$ is a Dedekind $\sigma$-complete Banach lattice, which translates into the fact that the compact $K$ is $\sigma$-Stonean. Thus, the Grothendieck property of $L_\infty(\mu)$ follows from the previous Remark 2. See also  \cite{Ando:61}.
\end{proof}

 Let $\mathscr{A}$ be a field of sets (a concrete Boolean algebra). The Stone space ${\rm St}\, \mathscr{A}$ of $\mathscr{A}$ coincides with the maximal ideal space of the Banach algebra $B(\mathscr{A})$ of all bounded, scalar-valued, $\mathscr{A}$-measurable functions (endowed with the supremum norm). As $B(\mathscr{A})$ is an abstract $M$-space/commutative $C^*$-algebra, the space $C({\rm St}\, \mathscr{A})$ is isometrically isomorphic to $B(\mathscr{A})$, so that we can freely interchange between the two descriptions of the same object. 
 
 One of the most interesting examples of algebras whose Stone space is \emph{not} a G-space is the algebra of Jordan-measurable subsets of the unit interval, that is sets whose boundary has Lebesgue measure zero (see, \emph{e.g.}, \cite[Corollary 5.8]{GravesWheeler:1983} for an extension of this fact to more general Jordan algebras). Schachermayer \cite[Proposition 4.6]{Schachermayer:82} observed that Stone spaces of algebras expressible as strictly increasing unions of countably many subalgebras are not G-spaces. \smallskip

In light of the first clause above one may ask about characterisation of those Boolean algebras (or more concretely, fields of sets) whose Stone spaces are G-spaces. This is indeed an active area of research with a strong set-theoretic flavour.

Let us list some examples of algebras whose Stone spaces are G-spaces:
\begin{itemize}
\item subsequentially complete Boolean algebras (every disjoint sequence in the algebra has a subsequence that has a least upper bound): \cite{Haydon:2001}.
\item Boolean algebras with Molt\'o's property (f): \cite[Corollary 1.4]{Molto:81}.
\end{itemize}

In \cite[Question 7.3]{KakolSobotaZdomsky:2020}, the authors raised the posed the following problem.
\begin{quest}
 Let $\mathscr{A}$ be a Boolean algebra whose Stone space is a G-space. Does there exist a Boolean subalgebra $\mathscr{B}$ of $\mathscr{A}$ whose Stone space $K$ fails to be a G-space, yet every weakly* convergent sequence of purely atomic measures on $K$ converges weakly?
\end{quest}

\begin{definition}
Let $X$ be a topological space.
The 0-\emph{Baire} functions are the continuous, scalar-valued functions on $X$; and for an ordinal number $\alpha \geqslant 1$, the $\alpha$-\emph{Baire} functions are the pointwise limits of sequences of $\beta$-Baire functions with $\beta < \alpha$.
\end{definition}

For an ordinal number $\alpha$, the space $B_\alpha(X)$ of all bounded Baire-$\alpha$ functions on a~topological space $X$ is a~commutative $C^*$-algebra when endowed with the supremum norm and operations defined pointwise. In particular, $B_\alpha(X)$ is isometric to some $C(K)$-space. Dashiell Jr.~studied the spaces $B_\alpha(X)$ 
and proved that they are Grothendieck (see \cite{Dashiell:81}; or  \cite[Theorem 3.3.9]{Dales:16}).
\begin{prop}
For every topological space $X$ and non-zero $\alpha<\omega_1$, the space $B_\alpha(X)$ has the  Grothendieck property.  
\end{prop}

Let us introduce a related notion for Banach spaces. 

\begin{definition}
Let $E$ be a Banach space considered naturally a subspace of $E^{**}$. We denote by $E_w$ the subspace of $E^{**}$ comprising all weak*-limits in $E^{**}$ of weakly Cauchy sequences in $E$.
\end{definition}

It is not difficult to show that if $\Gamma$ is an infinite set, then $c_0(\Gamma)_w= \ell_\infty^c(\Gamma)$. In particular, $(c_0)_w= \ell_\infty$. Moreover, $E_w$ satisfies the following properties:

\begin{prop}
Let $E$ be a Banach space.
\begin{romanenumerate}
\item\label{w-1} $E_w$ is a closed subspace of $E$.
\item\label{w-2} $E=E_w$ if and only if $E$ is weakly sequentially complete.
\item\label{w-3} If $E$ is separable, then $E^{**}=E_w$ if and only if $E$ contains no copies of $\ell_1$.
\end{romanenumerate}
\end{prop}
\begin{proof}
\eqref{w-1} was proved by McWilliams \cite{mcwilliams:62}, \eqref{w-2} is clear, and \eqref{w-3} was proved by Odell and Rosenthal \cite{OdellR:75}.
\end{proof}

The following problem arises.

\begin{quest}\label{quest:Ew}
Characterise Banach spaces $E$ for which $E_w$ is Grothendieck.

Moreover, is $E_w$ Grothendieck when so is $E$?
\end{quest}

The space $C(K)_w$ has a natural representation in terms of 1-Baire functions (\cite[Theorem 3.3.9]{Dales:16}).

\begin{prop}
For every compact space $K$, the space $C(K)_w$ may be identified with the space of bounded Baire-1 functions $B_1(K)$ on $K$, and it is a Grothendieck space.
\end{prop}

Observe that for every compact space $K$, we have two Grothendieck spaces $B_1(K)$ and $B_1(K)/C(K)$. In fact, we have potentially many more if we consider the spaces $B_\alpha(K)$ of bounded $\alpha$-Baire functions on $K$ ($1\leqslant \alpha \leqslant \omega_1$).


\section{{$C^*$}-algebras}\label{sect:c-star-algs} In Section~\ref{sect:ck}, we discussed in detail the question of when commutative (unital) $C^*$-algebras (that is $C(K)$-spaces) are Grothendieck. As for possibly non-commutative $C^*$-algebras, Problem~\ref{quest:Aw} as well as Section~\ref{sect:c-star-tensor} concerning $C^*$-tensor products, and Proposition~\ref{prop:scriptLinfty} touch upon this topic directly. In this context, Pfitzner's theorem (Theorem~\ref{thm:pf}, \cite{Pfitzner:94}; see also \cite{FPPeralta:2010}) is of fundamental importance.
\begin{theorem}[Pfitzner's theorem]\label{thm:pf}
Let $A$ be a $C^*$-algebra and let $K \subseteq A^*$ be a bounded set. Then
$K$ is \emph{not} relatively weakly compact if and only if there are $\delta > 0$ and a sequence $(a_n)_{n=1}^\infty$ of of pairwise orthogonal, norm-one self-adjoint elements in $A$ such that
\[
    \sup_{f\in K} |\langle f, a_n\rangle| \geqslant \delta.
\]
In particular, $C^*$-algebras have property $(V)$.
\end{theorem}

As a consequence, Pfitzner deduced the following result, answering  \cite[Problem 2]{Diestel:73}.  

\begin{cor}
Each von Neumann algebra (dual $C^*$-algebra) is a Grothendieck space; in particular $\mathscr{B}(H)$ is Grothendieck for every Hilbert space $H$.
\end{cor}

Fern\'andez-Polo and Peralta extended Pfitzner's theorem to JB*-triples \cite{FPPeralta:2009}.\smallskip

Let us say that a \emph{masa} is a maximal Abelian sub-$C^*$-algebra of a $C^*$-algebra. 
The second-named author observed that one can deduce from Theorem~\ref{thm:pf} the following result  (\cite[Proposition~2.5]{Kania:15}).
\begin{prop}\label{prop:masa}
Let $A$ be a $C^*$-algebra. Suppose that each masa in $A$ is a Gro\-then\-dieck space. Then $A$ is a Grothendieck space.
\end{prop}
\begin{proof}
Let $T\colon A\to c_0$ be an operator. Assume contrapositively that $T$ is \emph{not} weakly compact. Then the adjoint $T^*$ is not weakly compact either, so the image $T^*(B_{c_0^*})\subset A^*$ is not relatively weakly compact, where $B_{c_0^*}$ denotes the unit ball of $c_0^*$. By Theorem~\ref{thm:pf} there are $\delta > 0$ and a sequence $(a_n)_{n=1}^\infty$
of pairwise orthogonal, norm-one self adjoint elements in $A$ such that
\[
    \sup_{f\in T^*(B_{c_0^*})} |\langle f, a_n\rangle| = \sup_{ y\in B_{c_0^*} } |\langle T^*y, a_n\rangle| = \sup_{ y\in B_{c_0^*} } |\langle y, Ta_n\rangle| \geqslant \delta.
\]

Since the elements $a_n$ ($n\in \N$) are self-adjoint and orthogonal, they generate a commutative sub-$C^*$-algebra $C$ of $A$. Let $B$ be a masa containing $C$. Then $T$ restricted to $C$ (hence also to $M$) is not weakly compact. In particular, $M$ is not Grothendieck.\end{proof}

A $C^*$-algebra is said to be  \emph{monotone} $\sigma$\emph{-complete} if each upper-bounded, monotone increasing sequence of self-adjoint elements has a supremum. The spectrum of a commutative $\sigma$-complete $C^*$-algebra is $\sigma$-Stonean (\emph{cf}.~Remark~\ref{rem:list}) and, \emph{vice versa}, if $K$ is a~$\sigma$-Stonean compact space, then $C(K)$ is monotone $\sigma$-complete. Sait\^o and Wright \cite{SaitoWright:2003} term a $C^*$-algebra \emph{Rickart}, whenever every masa $M$ of $A$ is monotone $\sigma$-complete; in particular, $M$ is Grothendieck because $K$ is a~G-space. They noted that each von Neumann algebra, each AW${}^*$-algebra, and each monotone $\sigma$-complete $C^*$-algebra is Rickart (another proof is presented in \cite{BrooksSaitoWright:2005}). Using Proposition~\ref{prop:masa}, we have thus subsumed the main result of \cite{SaitoWright:2003}.
\begin{cor}
Every Rickart $C^*$-algebra is a Grothendieck space.
\end{cor}
We have already noted that $B_1(K)$ is a commutative $C^*$-algebra. However, the following problem is still open (\cite[Chapter 6, Question 2]{Dales:16}):

\begin{quest}\label{quest:Aw}
 Let $A$ be a $C^*$-algebra. Is $A_w$ a Grothendieck space?
\end{quest}

In the non-commutative case, it is not clear whether $A_w$ is a sub-$C^*$-algebra of the enveloping von Neumann algebra $A^{**}$. \smallskip

We remark in passing that Chetcuti and Hamhalter \cite{ChetcutiHamhalter:2009} characterised $C^*$-algebras satisfying the Brooks--Jewett theorem as precisely those $C^*$-algebras that are Grothendieck spaces and whose  irreducible representations are finite-dimensional.

\section{{$\mathcal{L}_\infty$}-spaces}
The $C(K)$-spaces are the best known examples of $\mathcal{L}_\infty$-spaces. They are Lindenstrauss spaces; \emph{i.e.}, they are $\mathcal{L}_{\infty, \lambda}$ for every $\lambda>1$, and they have property $(V)$.
That is not true for all $\mathcal{L}_\infty$-spaces.

\begin{example}\label{ex:BD}
Bourgain and Delbaen \cite{BD:80} constructed two families of  separable, non-reflexive $\mathcal{L}_\infty$-spaces; hence they are not Grothendieck.

The spaces in the first family are hereditarily reflexive, so that they contain no copies of $c_0$; hence they fail property $(V)$.

The spaces $E$ in the second one have the Schur property; hence they fail property $(V)$; moreover they  satisfy $E=E_w$.
\end{example}

So there are $\mathcal{L}_\infty$-spaces $E$ for which $E_w$ is not Grothendieck. The next question is a special case of  Problem~\ref{quest:Ew}.

\begin{quest}
Let $E$ be a Grothendieck  $\mathcal{L}_\infty$-space. Is $E_w$ a  Grothendieck space?
\end{quest}

For the following result we refer to \cite[Proposition 1.9]{ACCGM:16}.

\begin{prop}\label{dual-Linfty}
A $\mathcal{L}_\infty$-space isomorphic to a dual space is injective; hence it has the  Grothendieck property.
\end{prop}

Using this result and the principle of local reflexivity, we can prove provide the following characterisation:

\begin{prop}\label{prop:bidual-inj}
A Banach space $X$ is a $\mathcal{L}_\infty$-space if and only if $X^{**}$ is injective.
\end{prop}

Haydon \cite{Haydon:78} proved that if the space $X^{**}$ is injective then it is isomorphic to $\ell_\infty(\Gamma)$ for some set $\Gamma$.

Since $X^*$ is a complemented subspace in $X^{***}$, from Proposition \ref{prop:bidual-inj} we derive the following consequence:

\begin{prop}\label{prop:Linfty-quot}
Let $Y$ be a closed subspace of a Banach space $X$. If $Y$ and $X$ are  $\mathcal{L}_\infty$-spaces then so is $X/Y$.
\end{prop}
\begin{proof}
Since $Y^{**}\equiv Y^{\perp\perp}$ and $X^{**}$ are injective, $X^{**}= Y^{\perp\perp}\oplus Z$. Thus $(X/Y)^{**}$, which is isomorphic to $Z$, is injective.
\end{proof}

Some $\mathcal{L}_\infty$-spaces, such as the ones described in Example \ref{ex:BD},  fail property $(V)$. Therefore we do not have a good characterisation of those which are Grothendieck. So the following problem was posed in \cite{ACCGM:16}.

\begin{quest}\label{quest:l-infty}
Is every $\mathcal{L}_\infty$-space without infinite-dimensional separable complemented subspace a Grothendieck space?
\end{quest}

The following result, which is implicit in \cite[Appendix]{Lindenstrauss:64},  may be relevant for solving Problem~\ref{quest:l-infty}. Let us note that every infinite-dimensional, Grothendieck $\mathcal{L}_\infty$-space is  non-separable.

\begin{prop}
A $\mathcal{L}_\infty$-space $X$ is Grothendieck if and only if every operator from $X$ into a separable space can be extended to any superspace of $X$.
\end{prop}
\begin{proof}
If $X$ is a Grothendieck $\mathcal{L}_\infty$-space, $Y$ is separable and $T\colon X\to Y$ is an operator, then $T$ is weakly compact. Thus $T^{**}\colon X^{**}\to Y$ is an extension of $T$ with $X^{**}$ injective, and $T^{**}$ can be extended to any super-space of $X^{**}$.

Conversely, let $T\colon X\to c_0$. Since $X$ is contained in  $\ell_\infty(\Gamma)$ for some set $\Gamma$ and the hypothesis implies that $T$ can be extended to $\widehat T\colon \ell_\infty(\Gamma)\to c_0$, $T$ is weakly compact because so is $\widehat T$. Hence $X$ is Grothendieck.
\end{proof}

We state a problem that is a special case of the previous one.

\begin{quest}
Let $X$ be a Grothendieck $\mathcal{L}_\infty$-space. Does $X$ have property $(V)$?
\end{quest}

\subsection*{Separably injective spaces}
In the present section we consider certain subclasses  of $\Lc_\infty$-spaces defined by various extension properties.

\begin{definition}
Let $Z$ be a Banach space and let $1\leqslant \lambda<\infty$. We say that $Z$ is
\begin{enumerate}[(1)]
\item \emph{$\lambda$-injective} if for a given a closed subspace $Y$ of a Banach space $X$, every operator $T\colon Y\to Z$ admits an extension $\widehat T\colon X\to Z$ with $\|\widehat T\|\leqslant \lambda\|T\|$;
\item \emph{universally $\lambda$-separably injective} if it satisfies (1) when $Y$ is separable;
\item \emph{$\lambda$-separably injective} if it satisfies (1) when $X$ is separable.
\end{enumerate}
\end{definition}

\begin{remark}
We can define the classes of \emph{injective}, \emph{universally separably injective}, and \emph{separably injective} spaces in the same way, but omitting the reference to the bound of the norm of the extension. However, it is not difficult to see that every injective space is $\lambda$-injective for some $\lambda$, and the same happens for the other classes; see \cite{ACCGM:16}.
\end{remark}

Obviously injective spaces are universally separably injective, and the latter spaces are separably injective. Additionally, the following result follows from \cite[Proposition 2.8]{ACCGM:16} and Proposition \ref{Groth-V-c0}. 

\begin{prop} 
Separably injective spaces are $\Lc_\infty$-spaces and have property $(V)$. Consequently, they are Grothendieck spaces if and only if they contain no complemented copies of $c_0$.
\end{prop}

It is well known that a Banach space is $1$-injective if and only if it is isometric to a $C(K)$-space with $K$ extremely disconnected compact, and it is a long-standing open problem if every injective space is isomorphic to a $1$-injective space. 

\begin{prop}
Universally separably injective spaces have property $(V_\infty)$, hence they are Grothendieck.
\end{prop}

Universally separably injective spaces were introduced in \cite{ACCGM:13b}; there is an abundance of examples of spaces in this class that are not injective spaces. We will point out a few and refer to \cite{ACCGM:16} for further examples and additional information.

\begin{example}
Let $\Gamma$ be an uncountable set.
The space $\ell_\infty^c(\Gamma)$
is universally separably injective because every $x\in\ell_\infty^c (\Gamma)$ has countable support. Therefore, if $Y$ is separable and $T\colon Y\to\ell_\infty^c(\Gamma)$ is an operator, then the range of $T$ is contained in a subspace isometric to  $\ell_\infty$, which is $1$-injective.
\end{example}

The subsequent proposition (\cite[Proposition 2.11]{ACCGM:16}) provides further examples of universally separably injective spaces.

\begin{prop}
Let $Y$ be a closed subspace of $X$. If $X$ is universally separably injective and $Y$ is separably injective, then $X/Y$ is universally separably injective.
\end{prop}

\begin{example}
The space $\ell_\infty/c_0$ is universally separably injective because  $\ell_\infty$ is injective and $c_0$ is separably injective.
\end{example}

Separably injective spaces are not necessarily Grothendieck: by Sobczyk's theorem, the space $c_0$ is $2$-separably injective. However, every infinite-dimensional separable and separably injective space is isomorphic to $c_0$ (\cite{Zippin:1977}) and it is $\lambda$-separably injective for no $\lambda<2$ (see \cite{ACCGM:13b}). In fact, we have the following result (\cite[Proposition 2.31]{ACCGM:13b}).

\begin{prop}
Each $1$-separably injective space has the Grothendieck property.
\end{prop}

The following problem was stated in \cite{ACCGM:16}.

\begin{quest}
Suppose that $X$ is a $\lambda$-separably injective Banach space for some $\lambda<2$. Is $X$ Grothendieck?
\end{quest}

Observe that if a $C(K)$-space is $\lambda$-separably injective for some $\lambda<2$, then it is $1$-separably injective \cite[Proposition 2.34]{ACCGM:16}. 

\section{Spaces of analytic and differentiable functions} Here we describe some results of Bourgain in
\cite{Bourgain:83} and \cite{Bourgain:83a}.  
\smallskip

Let $\mathbb{D}$ denote the open unit disc in the complex plane and let $\mathbb{T}$ be the unit circle. We denote by $A$ the disc algebra, which is the closed subspace of $C(\mathbb{T})$ comprising all functions that admit an analytic extension in $\mathbb{D}$. Let $m$ denote the Lebesgue measure on $\mathbb{T}$, $L_1\equiv L_1(\mathbb{T})$, and $L_\infty\equiv L_\infty(\mathbb{T})$. 

By the F.~and M.~Riesz theorem (see \cite{Pelczynski:77}), the annihilator $A^\perp$ can be identified with a~closed subspace $H^1_0$ of $L_1(\mathbb{T})$,
\[
    A^\perp=\{\mu\in C(\mathbb{T})^* \colon \mu=h\cdot m, h\in H^1_0\},
\]
and $(L_1/H^1_0)^*\equiv H^\infty$ is a closed subspace of $L_\infty\equiv L_1^*$ that can be  identified with the space of bounded analytic functions on the unit disc $\mathbb{D}$.
\begin{example}Bourgain proved in \cite{Bourgain:83} that the space $H^\infty$ has property $(V_\infty)$; hence it is a Grothen\-dieck space, yet it is not a $\Lc_\infty$-space. 
\end{example}

For the latter part, observe that  $H^\infty$ is not an injective space, because it is not complemented in $L_\infty$; hence, since $H^\infty$ is a dual space, it cannot be a $\Lc_\infty$-space by Proposition \ref{dual-Linfty}.

\begin{example}
The second dual of the disc algebra  $A^{**}$ has property $(V_\infty)$, yet it is not a $\Lc_\infty$-space. 
\end{example}

Indeed, $C(\mathbb{T})^*=L_1\oplus_1 V_{\rm sing}$, where $V_{\rm sing}$ is the space measures $\mu$ on $\mathbb{T}$ that are singular with respect to $m$. Then we have $A^*=L_1/H^1_0\oplus_1 V_{\rm sing}$;  hence $A^{**}=H^\infty\oplus_\infty V_{\rm sing}^*$, where $V_{\rm sing}^*$ is a dual $\Lc_\infty$-space, so it is injective.
\medskip

Let $E_n = C^1([0,1]^n)$ denote the space of continuously differentiable functions on the $n$-dimensional cube. 

In \cite{Bourgain:83a}, Bourgain proved that for every $n\in \mathbb N$ the dual space of $E_n$ is weakly sequentially complete. Observe that, by Taylor's theorem, $E_1\cong C[0,1]\oplus \mathbb R \cong C[0,1]$; hence $E_1^{**}$ is a Grothendieck space. However it is not known whether the spaces $E_n$ are isomorphic to $C(K)$-spaces for $n\geqslant 2$. These observations lead to the following problem.

\begin{quest} Let $n\geqslant 2$. Is $E_n^{**}$ a Grothendieck space?
\end{quest}

Should the solution be affirmative, $E_n^{***}$ (hence also $E_n^*)$ would be weakly sequentially complete, so that Bourgain's result would have been subsumed.


\section{Banach lattices and ordered vector spaces}
The fact that a $C(K)$-space with $K$ a Stonean compact is Grothendieck has a lattice-theoretic flavour since these spaces may be abstractly described as Dedekind-complete AM-spaces. Here we collect several examples of Banach spaces that have been proved to be Grothendieck using Banach-lattice techniques, and some results discussing when an ordered Banach space---such as a Banach lattice or a $C^*$-algebra (or rather, the real part thereof)---satisfying certain separation property is a~Grothendieck space. \smallskip

Throughout this section we fix a complete measure space $(X,\Sigma,\mu)$. \smallskip

In \cite[Theorem 1]{Lotz:10}, Lotz proved that a Banach lattice $E$ with weakly sequentially continuous dual and satisfying some technical conditions is a Grothendieck space. In particular, the following result (\cite[Theorem 3]{Lotz:10}) was proved.

\begin{example}
The space $L^{p,\infty}(\mu)$ is Grothendieck for $1<p<\infty$.
\end{example}

Let $1/p + 1/q=1$. The space $L^{p,\infty}(\mu)$, which is also called a \emph{weak $L^p$-space}, consists of all real-valued (equivalence classes with respect to the relation of being equal almost everywhere of) $\mu$-measurable functions $f$ such that $\{\omega\colon |f(\omega)|>0\}$ is $\sigma$-finite and
\[
    \|f\|=\sup\left\{ \frac{1}{\mu(B)^{1/q}} \int_B |f|\, {\rm d}\mu\colon B\in\Sigma, 0<\mu(B)<\infty \right\}<\infty.
\]

De Pagter and Sukochev \cite{dePagter:20}
showed that Lotz' criterion \cite[Theorem 1]{Lotz:10} can be applied to a large  class of Marcinkiewicz spaces $M_\Psi$ that are defined as follows.

Let $0<r\leqslant \infty$ and let $\Psi\colon [0,r)\to [0,\infty)$ be a non-zero increasing concave function, continuous on $(0,r)$ and satisfying $\Psi(0)=0$.
The Marcinkiewicz space $M_\Psi(0,r)$ is the space of all real-valued measurable functions $f$ on $(0,r)$ such that
$$
\|f\|_{M_\Psi}=\sup_{0<t<r} \frac{1}{\Psi(t)} \int_0^t f^*(s)\, {\rm d}s<\infty,
$$
where $f^*$ is the decreasing rearrangement of $|f|$.

The subsequent result is due to de Pagter and Sukochev (\cite[Theorem 5.6]{dePagter:20}).

\begin{prop} Let $\Psi\colon [0,r)\to [0,\infty)$ be a function as above. Then the Marcinkiewicz space $M_\Psi(0,r)$ is Grothendieck if and only if one of the following conditions is satisfied:
\begin{enumerate}[label=(\Alph*)]
\item $r=\infty$, $\liminf\limits_{t\to 0} \Psi(2t)/\Psi(t)>1$ and $\liminf\limits_{t\to \infty} \Psi(2t)/\Psi(t)>1$, or
\item $r<\infty$ and $\liminf\limits_{t\to 0} \Psi(2t)/\Psi(t)>1$.
\end{enumerate}
\end{prop}

These results admit some extensions to general measure spaces. We refer to \cite{dePagter:20} for additional information.
\medskip

Polyrakis and Xanthos generalised the fact that for a compact F-space $K$ the space $C(K)$ is Grothendieck to the framework of general ordered Banach spaces as follows (\cite[Theorem 9]{PolyrakisXanthos:2013}). 

\begin{prop}\label{prop:cones}
Let $E$ be an ordered Banach space whose cone is closed and normal. If $E$ has an order unit and the countable interpolation property, then $E$ is a Grothendieck space. 
\end{prop}

As already mentioned in Section~\ref{sect:ordered}, the countable interpolation property and the countable monotone interpolation  property are equivalent for Banach lattices, but in general they are not. It is thus natural to ask whether countable monotone interpolation property is sufficient for the Grothendieck property of a given ordered space.

\begin{quest}\label{quest:cmip}
Can we replace `the countable interpolation property' by `the countable monotone interpolation property' in Proposition \ref{prop:cones}? 

What happens in the case $E$ is a $C^*$-algebra?
\end{quest}

This problem is particularly relevant for $C^*$-algebras and related structures. Indeed, let $A$ be a $C^*$-algebra. As we have the canonical isometric identification of the dual space of $A = A_{\rm sa} \oplus_{\mathbb R} iA_{\rm sa}^*$ with $(A_{\rm sa})^* \oplus_{\mathbb R} i (A_{\rm sa})^*$, $A$ is a Grothendieck space if and only if so is the real Banach space $A_{\rm sa}$. Consequently, had Problem~\ref{quest:cmip} had affirmative solution, we would have found a way alternative to invoking Pfitzner's theorem (Theorem~\ref{thm:pf}) for establishing the Grothendieck property of $C^*$-algebras with countable monotone interpolation property. Such algebras include von Neumann algebras (\cite[p.~117]{SmithWilliams:88}) and corona algebras of $\sigma$-unital, non-unital algebras (\cite[Theorem 3.14.2]{Pedersen:2018}).\smallskip 

It turns out that the Grothendieck property may be characterised in terms of well-based cones (\cite[Theorem 15]{PolyrakisXanthos:2011}).\smallskip

\begin{prop}\label{polyrakis2}
Let $X$ be a Banach space. Then $X$ is \emph{not} a Grothendieck space if and only if there exists a well-based cone $P$ of $X^*$ such that the set
$$P_0 = \{x\in X\colon \langle f, x\rangle \geqslant 0 \quad (f\in P)\, \}$$
has non-empty interior and there exists $x_0\in P_0$ such that the set $\bigcup_{n=1}^\infty [-nx_0, nx_0]$ is dense in $X$ with respect to the seminorm $d_p(x) = \sup_{f\in V} \langle f, x\rangle$, where $V$ is the convex hull of $(B_{X^*}\cap P)\cup -(B_{X^*}\cap P)$.

\end{prop}



\section{Dual Grothendieck spaces}\label{sect:dual-Groth}
We consider the class of Banach spaces $X$ such that the dual space $X^*$ is a Grothendieck space. Clearly reflexive spaces are in this class, as well as $\Lc_1$-spaces, because $X$ is a $\Lc_1$-space if and only if $X^*$ is a~$\Lc_\infty$-space; equivalently, $X^*$ is injective (Proposition \ref{dual-Linfty}).
Moreover, preduals of von Neumann algebras are in this class too.

\begin{quest}
 Find a characterisation of dual Grothendieck spaces.
\end{quest}
Certainly, weak sequential completeness of $X^{**}$ is a~necessary condition for $X^*$ being a Grothendieck space.\smallskip

Let $X^{(k)}$ denote the dual of order $k$ ($k\in\N\cup\{0\}$) of a Banach space $X$; \emph{i.e.}, $X^{(0)}=X$ and $X^{(k+1)} = (X^{(k)})^*$. 
Contreras and D\'\i az (\cite[Corollaries 3.7 and 3.9]{Contreras:99}) continued the work of Bourgain \cite{Bourgain:83} as follows: 

\begin{example}
The duals of even order $A^{(2k)}$ and $H^{\infty (2k)}$ of the disk algebra $A$ and $H^\infty$ are Grothendieck spaces.
None of them is a $\Lc_\infty$-space.
\end{example}

From Corollary \ref{cor:2<p} we get:

\begin{example}
For $2<p<\infty$, the space $\ell_\infty \projtp \ell_p\equiv (c_0 \projtp \ell_p)^{**}$ is Grothendieck. 
\end{example}

Next we present a construction of Leung \cite[Examples 11 and 15]{Leung:88a} which is a variation of the Schreier space. 
\begin{prop}\label{prop:ex-Leung}
There exists a Banach space $F$ with a shrinking unconditional basis $(e_i)_{i=1}^\infty$, which satisfies the following properties:
\begin{enumerate}
\item Both sequences $(e_i)_{i=1}^\infty$ and the  coefficient functionals $(e^*_i)_{i=1}^\infty$ are weakly null. Hence $F$ fails the Dunford--Pettis property. 
\item $F$ is hereditarily $c_0$ and $F^*$ is hereditarily $\ell_1$.
\item $F^{**}$ is Grothendieck, but it is not a $\Lc_\infty$-space because it fails the Dunford--Pettis property. 
\item $F^{**}$ has the {surjective Dunford--Pettis property}: every surjective operator from $F^{**}$ onto a reflexive space maps weakly convergent sequences to convergent sequences. 
\end{enumerate}
\end{prop}
Let us briefly describe the construction of $F$. Given an integer $i\geqslant 0$, we say that a~subset $A$ of $\N$ is $i$-\emph{admissible}, whenever $\textrm{card\,}A=2^i$ and $\min A\geqslant 2^i$. For $1\leqslant p\leqslant\infty$, we denote by $\|\cdot\|_p$ the $\ell_p$-norm in the space $c_{00}$ of all finitely supported scalar sequences.
\newline\indent
Given $x=(a_i)_{i=1}^\infty\in c_{00}$ and an $i$-admissible set $A$, we set $q_A(x)=\|x\|_{\sqrt{i}}$ (with the convention $\|x\|_{\sqrt{0}}=\|x\|_\infty)$, and define
$$\|x\|_F =\sup\{q_A(x)\colon i=0,1,2, \ldots \text{ and } A\text{ is }i\text{-admissible}\}.$$

The space $F$ is the completion of $(c_{00}, \|\cdot\|_F)$. It is not difficult to check that the standard unit vector basis $(e_i)_{i=1}^\infty$ is a monotone unconditional basis of $F$. We refer to \cite{Leung:88a} for the proof of the properties of $F$ and its dual spaces.\medskip

A simple but interesting example is the following one. 
\begin{example}
The space $E = (\bigoplus_{n\in\mathbb N}\ell_2^n)_{\ell_\infty}$ (isometric to the bidual of  $(\bigoplus_{n\in\mathbb N}\ell_2^n)_{c_0}$) has the Grothendieck property, yet it is not a $\Lc_\infty$-space and fails the Dunford--Pettis property, because it contains a complemented copy of $\ell_2$.
\end{example}
Indeed, $E$ is a complemented subspace of the direct sum $\ell_\infty(\ell_2)$, which in turn is a quotient of $\ell_\infty(\ell_\infty)\equiv \ell_\infty$ so it is a Grothendieck space. In order to construct a projection $P$ onto $E$ with range isomorphic to $\ell_2$, we fix a non-principal ultrafilter $\U$ on $\N$. Then, for a~sequence $a=((a_{n,k})_{k=1}^n)_{n=1}^\infty\in E$, we denote $a_{\infty,k}=\lim_{n\to\U}a_{n,k}$, and define $Pa$ by $(Pa)_{n,k}=a_{\infty,k}$.

\chapter{Stability properties of Grothendieck spaces}\label{chap:stability}

Here we will show how to produce  new examples of Grothendieck spaces beginning with the ones we know.

\section{Subspaces of separable codimension}\label{sect:sep-codim}
The following result was proved in  \cite[Proposition 3.1]{GLR:21}.

\begin{prop}\label{X/M-sep}
Let $X$ be a Grothendieck space. If $M$ is a closed subspace of $X$ and $X/M$ is separable, then $M$ is a Grothendieck space.
\end{prop}
\begin{proof}
Let $S\colon M\to c_0$ be an operator. Since $c_0$ is separably injective (\cite[Theorem 2.3]{ACCGM:16}) and the quotient $X/M$ is separable, $S$ admits an extension $T\colon X\to c_0$ (\cite[Proposition 2.5]{ACCGM:16}), which
is weakly compact by Theorem~\ref{Groth-sp}.
Then $S$ is weakly compact, and Theorem~\ref{Groth-sp} allows us to conclude that $M$ is Grothendieck.
\end{proof}

Proposition \ref{prop:Linfty-quot} provides further information about the examples captured by Proposition~\ref{X/M-sep}.

\begin{cor}\label{cor:not-l-infty}
Let $X$ be a Grothendieck $\mathcal{L}_\infty$-space. If $Y$ is a closed subspace of $X$ with $X/Y$ infinite-dimensional and separable, then $Y$ is Grothendieck, but it is not a~$\mathcal{L}_\infty$-space. In particular, $Y$ is not isomorphic to $X$.
\end{cor}
\begin{proof}
If $Y$ were a $\mathcal{L}_\infty$-space, then $X/Y$ would be too, but it is not since  it is reflexive.   
\end{proof}

Corollary~\ref{cor:not-l-infty} suggests the following problem.

\begin{quest}
Let $X$ be a Grothendieck space without complemented separable, infinite-di\-men\-sio\-nal subspaces and let $Y$ be a closed subspace of $X$ with $X/Y$ infinite-di\-men\-sio\-nal separable.
Is it possible for $Y$ to be  isomorphic to $X$?
\end{quest}

Let us invoke another consequence of Proposition \ref{X/M-sep} (\cite[Theorem 3.3]{GLR:21}):

\begin{cor}\label{cor:uncountable} There exists an uncountable family of pairwise non-isomorphic Grothendieck subspaces $M$ of $\ell_\infty$ with $\ell_\infty/M$ infinite-dimensional and separable.
\end{cor}
\begin{proof}
For $1<p\leqslant 2$, the space  $L_1\equiv L_1[0,1]$ has a subspace $N_p$ isomorphic to $L_p$  \cite[Corollary 2.f.5]{LT1}. Then $N_p^\perp$ is a subspace of $L_\infty$ with $L_\infty/N_p^\perp$ isomorphic to $L_p^*$. Let $U\colon L_\infty\to\ell_\infty$ be an isomorphim and $M_p=U(N_p^\perp)$.
Then $\ell_\infty/M_p\cong L_p^*$ is infinite-dimensional and separable, hence $M_p$ is Grothendieck by Proposition \ref{X/M-sep}.

Let $1<p,q\leqslant 2$ and suppose that $T\colon M_p\to M_q$ is an isomorphism. Since $\ell_\infty/M_p$ and $\ell_\infty/M_q$ are reflexive, by \cite[Theorem 2.f.12]{LT1} there exists an extension $\hat T\colon \ell_\infty\to \ell_\infty$ of $T$ which is a Fredholm operator. Hence the operator $\hat T$ induces a Fredholm operator $S\colon \ell_\infty/M_p\to \ell_\infty/M_q$ implying that $L_p^*$ and $L_q^*$ are isomorphic (because either space is isomorphic to their subspaces of finite codimension); hence $p=q$.
\end{proof}

Can we improve Proposition \ref{X/M-sep} to assuming only that ${\rm dens}\,X / M < \mathfrak{p}$?

\begin{quest}\label{q:card-quot}
Let $X$ be a Grothendieck space. Suppose that $M$ is a closed subspace of $X$ such that ${\rm dens}\, X/M < \mathfrak{p}$. Is $M$ a Grothendieck space?
\end{quest}

Since separable quotients of Grothendieck spaces are reflexive, the following problem also arises.

\begin{quest}\label{q:reflex-quot}
Let $M$ be a closed subspace of a Grothendieck space $X$ such that $X/M$ is reflexive. Is $M$ a Grothendieck space?
%
\end{quest}

\noindent \emph{\textbf{Added in proof.}} Problems \ref{q:card-quot}
and \ref{q:reflex-quot}
have been recently solved affirmatively in \cite{martinezcervantez:2021}. In particular, even in the case when $X$ is Grothendieck and ${\rm dens}\, X/M < \mathfrak{s}$, where $\mathfrak{s}$ is a certain cardinal called the splitting number which satisfies  $\mathfrak{p}\leqslant \mathfrak{s}$, the subspace $M$ is Grothendieck. \medskip


The space $\ell_\infty$ contains isometric copies of every separable Banach space (in particular, reflexive), so the number of pair-wise non-isomorphic subspaces of $\ell_\infty$ that are Grothendieck is at least continuum. Among reflexive subspaces of $\ell_\infty$ this is also an upper bound as every reflexive subspace of $\ell_\infty$ is separable. 

\begin{quest}
What is the number of pairwise non-isomorphic Grothendieck subspaces of $\ell_\infty$;
can it be $2^{\mathfrak{c}}$?
\end{quest}

Dow, Gubbi, and Szyma\'nski constructed in \cite{Dowetal:1988} a family of  $2^{\mathfrak c}$ pairwise non-homeomor\-phic extremally disconnected, separable compact spaces that are topologically rigid. It is known that every separable extremally disconnected compact space $K$ embeds into $\beta \mathbb N$ (\cite[Corollary 3.2]{Woods:76}). As $K$ is separable, $C(K)$ embeds into $\ell_\infty$. Since the spaces are pairwise non-homeomorphic, by the Banach--Stone theorem, the corresponding spaces of continuous functions are pairwise non-isometric. Since each space $K$ is extremally disconnected, $C(K)$ is injective, hence Grothendieck. Thus, there are $2^{\mathfrak{c}}$ isometry types of Grothendieck subspaces of $\ell_\infty$. However, each $C(K)$ is in this case isomorphic to $\ell_\infty$, because as observed by Lindenstrauss \cite{Lindenstrauss:67}, infinite-dimensional complemented subspaces of $\ell_\infty$ are isomorphic to $\ell_\infty$.
\medskip

Sait\^{o} and Maitland Wright \cite[Corollary 20]{SaitoWright:2007} constructed $2^{\mathfrak{c}}$ sub-$C^*$-algebras of $\ell_\infty$ that are pairwise non-isomorphic as $C^*$-algebras, and each of these algebras is a monotone $\sigma$-complete quotient of the algebra of bounded Borel functions on $[0,1]$ by a suitable ideal. In particular, they are Grothendieck spaces (\emph{cf}.~Remark~\ref{rem:list} and Section~\ref{sect:c-star-tensor} for more details). By the Gelfand--Naimark theorem and the Banach--Stone theorem, they are not isometric as Banach spaces. However, we do not know whether they are non-isomorphic as Banach spaces.



\section{Non-trivial twisted sums}
Recall that the notion of twisted sum of two Banach spaces $Y$ and $Z$ was introduced in Section \ref{sect:3SP}.
By Proposition \ref{prop:3SP}, a twisted sum of Grothendieck spaces is Grothendieck. This fact will allow us to obtain further, perhaps exotic, examples of Grothendieck spaces. For this, we shall employ push-out and pull-back diagrams that are standard tools for building twisted sums from given spaces. 
We refer to \cite{CC:21} and  \cite{Castillo-Gonzalez:97} for information on these techniques. 

\begin{example}\label{ex-PB}
There exists a Banach space ${\rm PB}_1$ which is a non-trivial twisted sum of $\ell_2$ and $\ell_\infty$, yet it is not isomorphic to a direct sum of a Hilbert space and a $C(K)$ space.
\end{example}

\begin{proof}
We consider the Banach space $Z_2$ constructed by Kalton and Peck in \cite{KaltonP:79}, which is not isomorphic to a Hilbert space but it is a twisted sum thereof.
Thus we have an exact sequence
\begin{equation}
\begin{CD}\label{Z2-seq}
0@>>> \ell_2 @>j>> Z_2 @>q>>\ell_2@>>> 0.
\end{CD}
\end{equation}

Then we take a surjective operator $Q\colon\ell_\infty\to\ell_2$, and we can consider the pull-back diagram
\begin{equation}
\begin{CD}
0@>>> \ell_2 @>j_1>> {\rm PB}_1 @>q_1>> \ell_\infty @>>> 0\\
&&@| @VQ_1VV @VVQV\\
0@>>> \ell_2 @>j>> Z_2 @>q>>\ell_2@>>> 0
\end{CD}
\end{equation}
which is commutative and has both rows exact. In particular, $Q_1$ is a surjective operator because so is $Q$.

Every operator from $C(K)$ into its dual space  factors through a Hilbert space (\cite[Corollary 2.15 and Theorem 3.5]{DiestelJT:95}), hence the same applies to the direct sum $C(K)\oplus H$ of $C(K)$ and a~Hilbert space $H$. However, since $Z_2$ is isomorphic to its dual space (\cite[Theorem 6.1]{KaltonP:79}), we can consider the operator $Q_1^*Q_1\colon {\rm PB}_1\to {\rm PB}_1^*$ that does not factor through a Hilbert space. Hence ${\rm PB}_1$ is \emph{not} isomorphic to $C(K)\oplus H$. In particular, the upper exact sequence is non-trivial.
\end{proof}

Readers unacquainted with the pull-back diagrams may observe that $$
{\rm PB}_1=\{(z,y)\in Z_2\oplus_\infty \ell_\infty\colon qz=Qy\},
$$
which is a closed subspace of $Z_2\oplus_\infty \ell_\infty$, and the induced maps $Q_1, q_1$, and $j_1$ are simply defined by $Q_1(z,y)=z$, $q_1(z,y)=y$, and $j_1x=(x,0)$, respectively. \medskip

Next we present a more general construction of twisting a  Grothendieck $C(K)$-spaces with a Hilbert space.

\begin{example}
Let $K_2$ be an infinite G-space. Then there exists a Banach space ${\rm PB}_2$ which is a non-trivial twisted sum of $\ell_2$ and $C(K_2)$ and it is not isomorphic to a direct sum of a Hilbert space $H$ and a $C(K)$-space. Moreover, it may be arranged that ${\rm PB}_2$ does not contain isomorphic copies of $\ell_\infty$.
\end{example}
\begin{proof}
For the first part, since $C(K_2)$ contains a copy of $\ell_1$ by Corollary \ref{cor:l1}, it admits a quotient isomorphic to $\ell_2$. So we can repeat the argument given in Example \ref{ex-PB}.

For the second part, if we use as $C(K_2)$ the space given by Haydon in \cite{Haydon:81} (see Example \ref{ex:Haydon}), then ${\rm PB}_2$ does not contain copies of $\ell_\infty$. For this, it is enough to observe that containing no copies of $\ell_\infty$ is a three-space property \cite{Castillo-Gonzalez:97}.
\end{proof}

The following example of non-trivial twisted sum is based in Proposition \ref{X/M-sep}. 

\begin{example}\label{ex-PO}
Let $M$ and $N$ be closed subspaces of $\ell_\infty$ with $N\subset M$ such that both $M/N$ and $\ell_\infty/M$ are infinite-dimensional and separable. Then there exists a Banach space ${\rm PO}_1$ which is a non-trivial twisted sum of $M$ and $\ell_\infty/N$.
\end{example}
\begin{proof}
Note that, by Proposition \ref{X/M-sep}, both $M$ and $\ell_\infty/N$ are Grothendieck spaces.

We consider the exact sequence
\begin{equation}
\begin{CD}\label{C(Kh)-seq}
0@>>> N @>j>> \ell_\infty @>q>>\ell_\infty/N@>>> 0
\end{CD}
\end{equation}
and the embedding $J\colon N\to M$ and form the associated push-out diagram
\begin{equation}
\begin{CD}
0@>>> N @>j>> \ell_\infty @>q>> \ell_\infty/N @>>> 0\\
&& @VVJV @VJ_1VV @|\\
0@>>> M @>j_1>> {\rm PO}_1 @>q_1>> \ell_\infty/N@>>> 0
\end{CD}
\end{equation}
which is a commutative diagram in which both rows are exact.

To show that the lower exact sequence is non-trivial; \emph{i.e.}, that $j_1(M)$ is uncomplemented in ${\rm PO}_1$, it is enough to show that $J\colon N\to M$ does not admit an extension $\widehat J\colon \ell_\infty\to M$. See Lemma 20 of \cite[Appendix A]{ACCGM:16}. Indeed, since $\ell_\infty/N$ is reflexive, by \cite[Theorem 2.f.12]{LT1}, each extension of $J$ to an operator $T\colon \ell_\infty\to \ell_\infty$ is a Fredholm operator (has finite-dimensional kernel and finite-codimensional range). Hence, its range cannot be contained in $M$.\end{proof}

Readers unacquainted with the push-out diagrams may observe that
$$
{\rm PO}_1= (M\oplus_1 \ell_\infty)/W, \textrm{ where } W =\{(Jx,-jx)\in M \oplus_1 \ell_\infty\colon x\in N\};
$$
moreover denoting by $Q\colon M\oplus_1 \ell_\infty \to {\rm PO}_2$ the canonical quotient map, the induced maps $j_1, J_1$, and $q_1$ are defined by $j_1y =Q(y,0)$, $J_1z =Q(0,z)$, and $q_1Q(y,z)=qz$, respectively.

\begin{example}\label{ex-PO2}
There exists a Banach space ${\rm PO}_2$ which is a non-trivial twisted sum of Haydon's space $C(K_2)$ in Example \ref{ex:Haydon} and $\ell_\infty/c_0$.
\end{example}
\begin{proof} We consider the exact sequence
\begin{equation}
\begin{CD}\label{C(Kh)-seq2}
0@>>> c_0 @>j>> \ell_\infty @>q>>\ell_\infty/c_0@>>> 0
\end{CD}
\end{equation}
and the embedding $J\colon c_0\to C(K_2)$ and form the associated push-out diagram
\begin{equation}
\begin{CD}
0@>>> c_0 @>j>> \ell_\infty @>q>> \ell_\infty/c_0 @>>> 0\\
&& @VVJV @VJ_2VV @|\\
0@>>> C(K_2) @>j_2>> {\rm PO}_2 @>q_2>> \ell_\infty/c_0@>>> 0
\end{CD}
\end{equation}
To show that the lower exact sequence is non-trivial, it is enough to observe  that the embedding $J\colon c_0\to C(K_2)$ does not admit an extension $\widehat J\colon \ell_\infty\to C(K_2)$  (see Lemma 20 of \cite[Appendix A]{ACCGM:16}), and this follows from the fact that such an extension would fix a copy of $\ell_\infty$, because $J$ is not weakly compact \cite[Proposition 2.f.4]{LT1}, and this is not possible because $C(K_2)$ does not contain copies of  $\ell_\infty$.
\end{proof}
The fact that property $(V_\infty)$ fails the three-space property gives another non-trivial example of Grothendieck space. We refer to \cite[Proposition 6.3]{ACCGM:16} for the details of this construction.

\begin{example}\label{ex-noVinfty}
There exists a Banach space $Z$ failing property $(V_\infty)$ which is a non-trivial twisted sum of two spaces $\ell_\infty(\Gamma,\ell_2)$ and $\ell_\infty$ which have property $(V_\infty)$.
\end{example}
Indeed, there exists an exact sequence
\begin{equation}
\begin{CD}\label{noVinfty}
0@>>> \ell_\infty(\Gamma,\ell_2) @>j>> Z @>q>>\ell_\infty@>>> 0
\end{CD}
\end{equation}
such that $q$ is an isomorphism on no copy of $\ell_\infty$. Thus $Z$ fails property $(V_\infty)$.
\newline\indent
Note that $\ell_\infty(\Gamma,\ell_2)$ has property $(V_\infty)$ because it is a quotient of the space $\ell_\infty(\Gamma,\ell_\infty)$, which is isometrically isomorphic to $\ell_\infty(\Gamma\times\N)$;  moreover, $\ell_\infty(\Gamma, \ell_2)\oplus \ell_\infty \cong\ell_\infty(\Gamma,\ell_2)$  and property $(V_\infty)$ is stable under direct sums.

\section{Tensor products and spaces of bounded operators}
Here we study the possible inheritance of the Grothendieck property by Banach-space tensor products, which is a~rather rare phenomenon, and the Grothendieck property for $\mathscr{B}(X)$.

\subsection*{Projective tensor product} We denote by $X\projtp Y$ the projective tensor product of two Banach spaces $X$ and $Y$. We start with the following result ({\cite[Proposition 6]{GonzalezG:95}}):

\begin{prop}\label{prop:ptp}
If $X\projtp Y$ is Grothendieck, then both $X$ and $Y$ are Grothendieck, and one of them is reflexive.
\end{prop}
\begin{proof}
The first part is clear, since $X\projtp Y$ is Grothendieck, and contains complemented copies of $X$ and $Y$.

Suppose that both $X$ and $Y$ are non-reflexive. By Corollary \ref{cor:l1}, both spaces contain $\ell_1$. Then both of them have a quotient isomorphic to $\ell_2$ \cite[Corollary 4.16]{DiestelJT:95}; hence $X\projtp Y$ has a quotient isomorphic to $\ell_2\projtp \ell_2$ \cite[Proposition 2.5]{Ryan:02}, which is separable and non-reflexive; thus $X\projtp Y$ is not Grothendieck.
\end{proof}
Proposition~\ref{prop:ptp} has a partial converse (\cite[Propositions 8 and 9]{GonzalezG:95}):
\begin{prop}
Suppose that $X$ is Grothendieck and $Y$ is reflexive.
\begin{enumerate}
\item If $\mathscr B(X,Y^*)=\mathscr K(X,Y^*)$, then $X\projtp Y$ is Grothendieck.
\item If $X\projtp Y$ is Grothendieck and $Y^*$ has the Bounded Compact Approximation property, then $\mathscr B(X,Y^*)=\mathscr K(X,Y^*)$.
\end{enumerate}
\end{prop}
%
In particular, we have the following result  (\cite[Corollary  10]{GonzalezG:95}):
\begin{cor}\label{cor:2<p}
Let $p\in [1,\infty]$. Then $\ell_\infty \projtp \ell_p$ is a Grothendieck space if and only if $2<p<\infty$.
\end{cor}

The dual space of $X\projtp Y$ can be naturally identified with $\mathscr B(X,Y^*)$. The following result of Holub (\cite[Theorem 2]{Holub:73}) shows the limitations of Proposition \ref{prop:ptp} in producing non-reflexive Grothendieck spaces. 
\begin{prop} 
Suppose that $X$ and $Y$ are reflexive Banach spaces an one of them has the Aproximation property. Then $X\projtp Y$ is reflexive if and only if $\mathscr B(X,Y^*)=\mathscr K(X,Y^*)$. 
\end{prop}

\subsection*{Injective tensor product}\label{Injten} We  denote by $X\injtp Y$ the injective tensor product of two Banach spaces $X$ and $Y$. Our first result concerns the inheritance of the Grothendieck property by the injective tensor product (\cite[Theorem 4]{Bu-Ji:10}).

\begin{prop}\label{prop:itp1} Suppose that $X$ is a  Grothendieck space, $Y$ is reflexive, $X^{**}$ or $Y^{**}$ has the Approximation property, and $\mathscr B(X^*,Y)=\mathscr K(X^*,Y)$.
Then $X\injtp Y$ is Grothendieck.
\end{prop}

In view of Proposition \ref{prop:ptp}, the following problem was raised in \cite[p. 1156]{Bu-Ji:10}.

\begin{quest}
Suppose that  $X\injtp Y$ is Grothendieck. Is $X$ or  $Y$  reflexive?
\end{quest}

We have the following partial answer (\cite[Theorem 7]{Bu-Ji:10}).
\begin{prop}\label{prop:itp2}
Suppose that $Y$ is a reflexive space and has an unconditional FDD.
Then $X\injtp Y$ is Grothendieck if and only if $\mathscr B(X^*,Y)= \mathscr K(X^*,Y)$ and the space $X$ is Grothen\-dieck.
\end{prop}

It follows from the previous result that, for $1<p<\infty$,  $\ell_\infty\injtp\ell_p$ is not Grothendieck (see \cite[p. 1159]{Bu-Ji:10}). When $Y$ is a reflexive space with the Approximation Property, we have $\mathscr K(X,Y)=X^* \injtp Y$. Let us then record the following consequence (\cite[Corollary 8]{Bu-Ji:10}) of that fact.

\begin{cor}\label{cor:Bu} Suppose that $Y$ is a reflexive space with an unconditional FDD. Then $\mathscr K(X,Y)$ is Grothendieck if and only if $X^*$ is Grothendieck and $\mathscr B(X,Y)= \mathscr K(X,Y)$.
\end{cor}

Let $1<p,q<\infty$. It follows from Corollary \ref{cor:Bu} that  $\ell_p^*\injtp\ell_q$ is Grothendieck if and only if $\mathscr{B}(\ell_p,\ell_q)= \mathscr{K}(\ell_p,\ell_q)$, but in this case $\ell_p^*\injtp\ell_q$ is  reflexive. We may then reiterate the following problem  (\cite[p. 1158]{Bu-Ji:10}).

\begin{quest}\label{Grothnonreflex}
Is there a pair of infinite-dimensional Banach spaces $X$ and $Y$ for which the tensor product $X\injtp Y$ is Grothendieck yet it is not reflexive?
\end{quest}

Khurana \cite[Theorem 2]{Khurana:78} proved that $C(K)\injtp X$ (which is isometric to $C(K, X)$, the space of $X$-valued continuous functions on $K$) is Grothendieck  only in two  cases:
\begin{itemize}
    \item $K$ is finite and $X$ is a Grothendieck space,
    \item $K$ is a G-space and $X$ is finite-dimensional.
\end{itemize}

Cembranos \cite{Cembranos:84} applied Khurana's ideas to prove that for every infinite compact Hausdorff space $K$ and every  infinite-dimensional Banach space $X$, the space $C(K,X)$ contains a~complemented copy of $c_0$,  from which Khurana's result follows immediately.

Thus, the answer to  Problem~\ref{Grothnonreflex} is negative when at least one of the spaces $X$ or $Y$ is of the form $C(K)$ for an infinite compact Hausdorff space. Indeed, if $X = C(K)$, then $X\injtp Y\equiv C(K, Y)$ contains a complemented copy of $c_0$. Actually, the key ingredient of Cembranos' proof is the fact that $C(K)$-spaces have property $(V)$. We can emulate that proof to establish the following partial result.

\begin{prop}\label{alacembranos}Suppose that $X$ is a non-reflexive space with a subspace isomorphic to $c_0$. If $X\injtp Y$ is a Grothendieck space, then $Y$ is finite-dimensional. \end{prop}

\begin{proof}
Suppose that $Y$ is infinite-dimensional and $X$ contains a~sequence $(e_n)_{n=1}^\infty$ which is $M$-equivalent to the standard basis of $c_0$ for some $M > 0$. Let $(f_n)_{n=1}^\infty$ be the sequence of fixed Hahn--Banach extensions of the coordinate functionals associated to $(e_n)_{n=1}^\infty$ on the closed linear span $X_0$ of this sequence.

Since $Y$ is infinite-dimensional, by the Josefson--Nissenzweig theorem (see \cite[Chapter XII]{Diestel:84}), there is a sequence $(y_n^*)_{n=1}^\infty$ in the unit sphere of $Y^*$ that converges to 0 in the weak* topology, and by \cite[Remark III.1]{JohnsonR:72} we can find a bounded sequence $(y_n)_{n=1}^\infty$ in $Y$ such that $\langle y^*_n, y_k\rangle = \delta_{n,k}$ ($n,k\in \mathbb N$). 
Then the map $T\colon X\injtp Y\to \ell_\infty$ given by
$$T\xi = ( \langle (f_n \otimes y_n^*), \xi\rangle)_{n=1}^\infty\qquad (\xi \in X\injtp Y) $$
is a bounded linear operator. Moreover, $T$ takes values already in $c_0$ because
$$|\langle f_n, x\rangle\cdot \langle y, y_n^*\rangle | \leqslant M\cdot \|x\|\cdot |\langle y, y_n^*\rangle|\qquad (x\in X, y\in Y)$$
for some constant $M$ and the sequence $(y_n^*)_{n=1}^\infty$ is weak*-null.
The operator $T$ is not unconditionally converging because $T(e_n\otimes y_n)(k) = \delta_{n,k}$ ($n,k\in \mathbb N$), and this is a contradiction, because $X\injtp Y$ is Grothendieck, so every operator from that space into $c_0$ is weakly compact, hence by the Orlicz--Pettis theorem, unconditionally converging.
\end{proof}

Let us observe that a positive answer to  Problem~\ref{Grothc_0} implies positive answer to Problem~\ref{Grothnonreflex}. Indeed, $X$ is isomorphic to a~complemented subspace of $X\injtp Y$, so as such it is a~Grothendieck space. Being non-reflexive, it would contain a copy of $c_0$, hence by Proposition~\ref{alacembranos}, $Y$ must be finite-dimensional.\medskip

We close this section by noting that the aforestated results concerning preservation of the Grothendieck property by injective/projective tensor products resonate in the theory of Banach lattices and their tensor products (positive injective, Fremlin, Wittstock tensor products, etc.). Defining these notions formally is beyond the scope of the present paper---instead we refer directly to the relevant papers: \cite{Bu:20,JiCraddockBu:2010, LiBu:2017,ZhangGuLi:2020}.

\subsection*{$C^*$-tensor products}\label{sect:c-star-tensor}
The second-named author \cite{Kania:15}, inspired by Cembranos' result \cite{Cembranos:84}, proved that an in\-fi\-ni\-te-di\-men\-sio\-nal $C^*$-algebra $A$, which is a Grothendieck space cannot be decomposed as a~$C^*$-tensor product $B_1\otimes_\gamma B_2$ of two infinite-dimensional $C^*$-algebras $B_1, B_2$ (\cite[Theorem 1.1]{Kania:15}).

\begin{prop}
Let $A$ be a $C^*$-algebra, and let  $B_1\otimes_\gamma B_2$ be a $C^*$-tensor product of two infinite-dimensional $C^*$-algebras. If $A\cong B_1\otimes_\gamma B_2$, then $A$ is not Grothendieck.
\end{prop}
\begin{proof}[Sketch of the proof] 
There is a minimal $C^*$-tensor product $\otimes_{\min}$, which for commutative $C^*$-algebras coincides with the injective tensor product of Banach spaces. 

The *-homomorphisms between $C^*$-algebras have always closed range, so there is always a surjective homomorphism $B_1\otimes_\gamma B_2\to B_1\otimes_{\min} B_2$. In particular, if $B_1\otimes_\gamma B_2$ is Grothendieck, then so is $B_1\otimes_{\min} B_2$. The minimal tensor product `respects' subalgebras, and the injective tensor product of Banach spaces `respects' closed subspaces. However, every infinite-dimensional $C^*$-algebra contains a subspace isomorphic to $c_0$ (every infinite-dimensional $C^*$-algebra contains a self-adjoint element $a$ with infinite spectrum (\cite[Ex. 4.6.12]{KadisonRingrose:1983}, so by the spectral theorem, $C^*(\sigma(a))\subset A$ is an infinite-dimensional commutative sub-$C^*$-algebra), hence Proposition~\ref{alacembranos} applies.\end{proof}
In particular, $\mathscr B(\ell_2)$ and the Calkin algebra $\mathscr B(\ell_2) / \mathscr K(\ell_2)$ do not admit such decompositions. It is still not known whether there exists a reasonable Banach-space cross-norm $\gamma$ and two Banach spaces $B_1, B_2$ such that $\mathscr{B}(\ell_2)$ is isomorphic as a Banach space to the tensor product $B_1 \otimes_\gamma B_2$.

\subsection*{Spaces of bounded operators} The space $\mathscr B(E)$ of all bounded operators on $E$ contains $E^*\injtp E$ as a closed subspace, that can be identified with the approximable operators, and also contains complemented copies of $E$ and $E^*$. So the next problem is a~special case of Problem \ref{XandX*-Gr}.

\begin{quest}
Suppose that $\mathscr B(E)$ is Grothendieck. Is $E$ reflexive?
\end{quest}

The converse implication fails:

\begin{example}
For $1<p<\infty$, the spaces $E=(\bigoplus_{n\in\mathbb N}\ell_1^n)_{\ell_p}$, $E = T$, the Tsirelson space, or $E=B_p$, the $p$\textsuperscript{th} Bearnstein space are reflexive, yet $B(E)$ is not a~Grothendieck space. We refer to \cite{BKL:19, Kania:13} for details. 
\end{example}

\begin{quest}\label{Q:superreflexive}
Suppose that $E$ is super-reflexive. Is $\mathscr B(E)$ Grothendieck?

What happens for $E=\ell_p$, $1<p<\infty$, $p\neq 2$?
\end{quest}

As Banach spaces, $\mathscr{B}(\ell_p)$ and $\mathscr{B}(L_p)$ are isomorphic; in fact, for every separable, infinite-dimensional 
$\mathscr{L}_p$-space $E$,
$$
\mathscr{B}(\ell_p) \cong \mathscr{B}(E) \cong \big(\bigoplus_{n=1}^\infty \mathscr{B}(\ell_p^n)\big)_{\ell_\infty};
$$
see \cite[Section 2]{AriasFarmer:96}. For $p=2$, this result is attributed to Lindenstrauss and Haagerup (see  \cite{ChristensenSinclair:1989}), and the proof relies on the Pe{\l}czy\'nski decomposition method. This implies that the Grothendieck property of $\mathscr{B}(\ell_p)$ and that of of $\mathscr{B}(L_p)$ are equivalent.
\medskip

We observe that it is an open problem whether $\mathscr{B}(E)$ may be reflexive for an infinite-dimensional (reflexive) space $E$. It is known that $\mathscr{B}(E)$ is non-reflexive when $E$ has the Bounded Approximation Property, because in this case $\mathscr{K}(E)^{**}\equiv \mathscr{B}(E)$ with the inclusion $\mathscr{K}(E)\to \mathscr{B}(E)$ being the canonical embedding. For this reason, should  Problem~\ref{Q:superreflexive} have positive answer, there would be no super-reflexive analogue of the Argyros--Haydon space \cite{ArgyrosHaydon:2011}, that is, there would be no super-reflexive space $E$ with a~Schauder basis on which every operator is of the form $\lambda I_E+S$, where $\lambda$ is a scalar, $I_E$ is the identity operator on $E$, and $S\in \mathscr{K}(E)$.

\section{Direct sums and spaces of vector-valued functions}
Bombal \cite{Bombal:89} noticed that for $1<p<\infty$, if $X$ is a Grothendieck space, then so is $\ell_p(X)$. Let us state a more general result that involves direct sums with respect to more general unconditional bases. 

\begin{prop}\label{prop:unc-sum} Let $E$ be a reflexive Banach space with a 1-unconditional basis $(e_\gamma)_{\gamma\in \Gamma}$, and let $X_\gamma$ $(\gamma\in \Gamma)$ be a family of Banach spaces. Then the following conditions are equivalent:

\begin{enumerate}
    \item\label{sum1} each space $X_\gamma$ $(\gamma\in \Gamma)$ is Grothendieck;
    \item\label{sum2} $(\bigoplus_{\gamma\in \Gamma} X_\gamma)_E$ is Grothendieck.
\end{enumerate}\end{prop}

\begin{proof} Let us denote $X = (\bigoplus_{\gamma\in \Gamma} X_\gamma)_E$. The implication \eqref{sum2} $\Rightarrow$ \eqref{sum1} is clear as each space $X_\gamma$ is isomorphic to a complemented subspace of $X$. For the converse implication, let us fix a~weak*-null sequence $(f^n)_{n=1}^\infty$ in $X^*$. Then 
$$\lim_{n\to\infty}\big\langle (x_\gamma)_{\gamma\in \Gamma}, \Lambda_E (f^n_\gamma)_{\gamma\in \Gamma}\big\rangle = \sum_{\gamma\in \Gamma} \langle x_\gamma, f^n_\gamma\rangle = 0\qquad \big((x_\gamma)_{\gamma\in \Gamma}\in X), $$
where $\Lambda_E$ is the map defined in Section \ref{sect:direct-sums}. 

Thus, $\lim_{n\to\infty}\langle x_\gamma, f^n_\gamma\rangle = 0 $ for each $\gamma\in \Gamma$ and all $x_\gamma\in X_\Gamma$. As $X_\gamma$ is Grothendieck, $\lim_{n\to\infty}\langle x^{**}_\gamma, f^n_\gamma\rangle = 0 $ for each $\gamma\in \Gamma$ and all $x^{**}_\gamma\in X_\Gamma^{**}$. As $E$ is reflexive, the basis $(e_\gamma)_{\gamma\in \Gamma}$ is both shrinking and boundedly complete, hence $(\bigoplus_{\gamma\in \Gamma} X_\gamma)_E$ is naturally isometrically isomorphic to $(\bigoplus_{\gamma\in \Gamma} X_\gamma^{**})_E$.

The susbspace $c_{00}(\Gamma, X^{**})$  of all finitely supported $X^{**}$-valued tuples is norm-dense in $(\bigoplus_{\gamma\in \Gamma} X_\gamma^{**})_E$. 
So we may restrict ourselves to the elements $G\in c_{00}(\Gamma, X^{**})$ for testing the Grothendieck property. Indeed, it follows from the above arguments that for such $G$ we have $\lim_{n\to\infty}\langle G, (f^n_\gamma)_{\gamma\in \Gamma}\rangle = 0 $. Then, by the uniform boundedness principle applied to $(f^n_\gamma)_{\gamma\in \Gamma}$ viewed as an element of $(\bigoplus_{\gamma\in \Gamma} X_\gamma^{***})_{E^*}$, we conclude the result.
\end{proof}

A Banach space $X$ contains {uniformly complemented copies of} $\ell_1^n$ ($n\in \mathbb N$) if and only if $X$ has non-trivial type (\cite[Theorem 13.3]{DiestelJT:95}).
We have already seen in Remark~\ref{nonsums} a~special case of the following result.

\begin{prop}
If $(\bigoplus_{n\in\mathbb N} X_n)_{\ell_\infty}$ is Grothendieck, then there is no $\lambda>0$ such that each space $X_n$ contains a $\lambda$-complemented copy of $\ell_1^n$.
\end{prop}
\begin{proof}
If each $X_n$ contains a $\lambda$-complemented copy of $\ell_1^n$, then $(\bigoplus_{n\in\mathbb N} X_n)_{\ell_\infty}$ contains a~complemented copy of $\ell_1$ \cite[p.~303]{Johnson:84}.
\end{proof}

In \cite[Theorem 3]{Leung:88}, Leung proved that if $E$ is a countably order-complete Banach lattice satisfying certain technical conditions, then $\ell_\infty(\Gamma, E)$ is a Grothendieck space. As a~consequence, he derived the following fact (\cite[Theorem 6]{Leung:88}):

\begin{example}
If $\phi$ is an Orlicz function such that the Orlicz space $L^\phi(\mu)$ has weakly sequentially complete dual space, then $\ell_\infty(\Gamma, L^\phi(\mu))$ is Grothendieck.
\end{example}

The special case of $\ell_\infty(\Gamma, L^p(\mu))$ was proved in \cite{Rabiger:85}. \medskip

A cardinal number $\lambda$ is \emph{real-valued}, whenever there exists an atomless probability measure on the power-set of $\lambda$. The existence of such a cardinal number cannot be proved in ZFC as it implies that ZFC is consistent. Assuming that a real-valued cardinal number exists, we denote by $\mathfrak{m}_r$ the smallest real-valued measurable cardinal. 

Leung and R\"abiger \cite[Theorem p.~55]{LeungR:90} obtained the following result.

\begin{theorem}
Let $\Gamma$ be a set with $|\Gamma|< \mathfrak{m}_r$ and let $(E_\gamma)_{\gamma\in\Gamma}$ be a family of Banach spaces. Suppose that $E = (\bigoplus_{\gamma\in \Gamma} E_\gamma)_{\ell_\infty(\Gamma)}$ has property $(V)$. Then $E$ is a Grothendieck space if and only if each space $E_\gamma$ is Grothendieck.

In particular, $E$ is Grothendieck, whenever
\begin{enumerate}
    \item\label{lr1} each space $E_\gamma$ $(\gamma\in\Gamma)$ is a Grothendieck Lindenstrauss space; or
    \item\label{lr2} each space $E_\gamma$ $(\gamma\in\Gamma)$  is a Grothendieck $C^*$-algebra.
\end{enumerate}
Moreover, if no real-valued cardinal exists, the result holds for any index set $\Gamma$.
\end{theorem}
Clause \eqref{lr1} is \cite[Corollary 2]{LeungR:90}, whereas clause \eqref{lr2} follows from the fact that $E$ is naturally a $C^*$-algebra when so is each $E_\gamma$ ($\gamma\in \Gamma$) and Pfitzner's theorem asserting that $C^*$-algebras have property $(V)$ (Theorem~\ref{thm:pf}). 

Let us reiterate the question asked by Kucher (\cite[Conjecture 1]{Kucher:99}).
\begin{quest}
Suppose that $E$ is super-reflexive. Is $\ell_\infty(E)$ a Grothendieck space?
\end{quest}

The case of spaces of vector-valued function on atomless measure spaces appears to be more restrictive.

\begin{prop} \label{prop:vvalued}
Suppose that  $\mu$ is an atomless measure.
\begin{enumerate}
\item If $1<p<\infty$ and the Bochner space $L_p(\mu,X)$ is Grothendieck, then the space $X$ is reflexive.
\item If $L_\infty(\mu,X)$ is Grothendieck, then $X$ has non-trivial type and it is reflexive.
\end{enumerate}
\end{prop}

\begin{proof}
We give the proof when $\mu$ is the Lebesgue measure on the unit interval, which is essentially valid in the general case.

Let $1<p\leqslant \infty$. If $X$ is a non-reflexive Grothendieck space, then $X^*$ contains a~sequence $(f_n)_{n=1}^\infty$ equivalent to the unit vector basis of $\ell_1$ (Corollary~\ref{cor:l1}). Let $(r_n)_{n=1}^\infty$ be a~sequence of Rademacher functions and $1/p+1/q=1$ ($q=1$ for $p=\infty$). We consider the simple functions $\varphi_n(t)=r_n(t)f_n$ in $L_q(\mu,X^*)$, regarded naturally a subspace of $L_p(\mu,X)^*$. It is not difficult to check that $(\varphi_n)_{n=1}^\infty$ is a weak$^*$ null sequence in $L_p(\mu,X)^*$ equivalent to the unit vector basis of $\ell_1$ (see the proof of \cite[Theorem 1]{Diaz:95}), hence $L_p(\mu,X)$ is not Grothendieck.

In the case $p=\infty$, observe that $L_\infty(\mu,X)$ contains a complemented copy of $\ell_\infty(X)$. Having non-trivial type is equivalent to non-containment of uniformly complemented copies of $\ell_1^n$ ($n\in \mathbb N$). Thus $X$ must have non-trivial type when $L_\infty(\mu,X)$ is Grothendieck (\emph{cf}. Remark~\ref{nonsums}).
\end{proof}

The former clause of Proposition~\ref{prop:vvalued} is due to D\'{\i}az \cite{Diaz:95}, whereas the latter one may be found in \cite[Theorem 13.3]{DiestelJT:95}.

\section{Ultrapowers and ultraproducts}

A Banach space $E$ is \emph{super-reflexive}, whenever every Banach space that is finitely representable in $E$ is reflexive.
It is well known that a Banach space $E$ is super-reflexive if and only if every ultrapower of $E$ with respect to a countably incomplete ultrafilter is reflexive \cite[Proposition 6.4]{Heinrich:80}.
This fact provides non-trivial examples of ultrapowers that are Grothendieck spaces. 

\begin{quest}
 Can an ultrapower of a reflexive space be Grothendieck without being reflexive?
 \end{quest}

By Proposition \ref{prop:scriptLinfty}, there exist ultraproducts of finite dimensional spaces which are Grothendieck non-reflexive spaces: just take $\{\ell_\infty^n\colon n\in\N\}$.\medskip

If $\{X_\gamma\colon \gamma\in \Gamma \}$ is a family of Banach spaces such that $(\bigoplus_{\gamma\in \Gamma} X_\gamma)_{\ell_\infty}$ is Grothendieck and $\U$ is an ultrafilter on $\Gamma$, then the ultraproduct $[X_\gamma]_\U$ is Grothendieck being a quotient of $(\bigoplus_{\gamma\in \Gamma} X_\gamma)_{\ell_\infty(\Gamma)}$. \medskip

Let us revisit further examples of Banach spaces whose ultraproducts with respect to countably incomplete ultrafilters are Grothendieck spaces (every non-principal ultrafilter on a countably infinite set is automatically countably incomplete). Observe that, by the Principle of Local Reflexivity,  the bidual $X^{**}$ of a Banach space $X$ is 1-complemented in some ultrapower of $X$, so $X^{**}$ has  the Grothendieck property when that ultrapower has it (\emph{cf}.~Problem~\ref{bidual-Gr}).

\begin{prop}\label{prop:scriptLinfty} Let $\U$ be a countably incomplete ultrafilter over a set $\Gamma$ and let $X_\gamma$ $(\gamma\in \Gamma)$ be Banach spaces.
\begin{enumerate}
 \item\label{c-star-1} If $X_\gamma$ are $C^*$-algebras $(\gamma\in \Gamma)$, then $[X_\gamma]_\U$ is a~Grothendieck space.  \item\label{c-star-2} If the ultraproduct $[X_\gamma]_\U$ is a $\widetilde{\mathcal{OL}}_{\infty}$-space, then it is a~Grothendieck space.
 \item\label{c-star-3} If the ultraproduct $[X_\gamma]_\U$ has property $(V)$, then it is a~Grothendieck space.
\end{enumerate}
\end{prop}

\begin{proof}
Clause~\eqref{c-star-1} has been already noticed in \cite[Proposition 1.2(ii)]{Kania:15}. It follows from the conjunction of Theorem~\ref{thm:pf} (ultraproducts of $C^*$-algebras are naturally $C^*$-algebras) and the fact that ultraproducts of Banach spaces over countably incomplete ultrafilters do not contain complemented copies of $c_0$ (\cite[Proposition 3.3]{ACCGM:13}). The conclusion then follows from Proposition~\ref{Groth-V-c0}.
\smallskip

As for \eqref{c-star-2}, by  Proposition~\ref{basic2}, it is enough to show that every separable subspace of the ultraproduct is contained in a Grothendieck subspace. The proof closely emulates that of \cite[Proposition 3.2]{ACCGM:13}.

Suppose that the ultraproduct $[X_\gamma]_{\U }$ is a $\widetilde{\mathcal{OL}}_{\infty, \lambda+\varepsilon}$-space for some $\lambda\geqslant 1$ and all $\varepsilon >0$,  and let $W$ be a separable subspace of $[X_\gamma]_{\U }$. Let $D$ be a countable linearly dense, linearly independent subset of $W$. For each $d\in D$, let $(d_\gamma)_{\gamma\in \Gamma}\in (\bigoplus_{\gamma\in \Gamma} X_\gamma)_{\ell_\infty(\Gamma)}$ be a~representative of $d$. We write $D$ as a strictly increasing union of finite sets $D_n$ and set $W_n = {\rm span}\, D_n$. Thus, there is a finite-dimensional subspace $F_n$ of $[X_\gamma]_{\U }$ that contains $W_n$ and a $C^*$-algebra $A_n$ such that $d_{{\rm BM}}(F_n, A_n) \leqslant \lambda + \tfrac{1}{n}$.

We fix a basis $B_n$ of $F_n$ that contains the set $D_n$. For $n\in\mathbb N$ and $b\in B_n$, we choose a~representative $(b_\gamma)_{\gamma\in \Gamma}\in (\bigoplus_{\gamma\in \Gamma} X_\gamma)_{\ell_\infty(\Gamma)}$, however we insist that $(b_\gamma)_{\gamma\in \Gamma} = (d_\gamma)_{\gamma\in \Gamma}$ as long as $b\in D$. We then define $(F_n)_\gamma$ as  ${\rm span}\,\{b_\gamma\colon b\in B_n\}\subset X_\gamma$.

As the ultrafilter $\U $ is countably incomplete, we may find a sequence of sets $I_n$ all in $\mathscr{U}$ whose intersection is empty. For each $n$ set
$$ \tilde{J}_n = \Big\{\gamma\in \Gamma\colon d_{{\rm BM}}((F_n)_\gamma, A) \leqslant \lambda + \tfrac{2}{n}\text{ for some }C^*\text{-algebra }A\Big\}\cap I_n$$
and $J_n = \tilde{J}_1\cap \ldots \cap \tilde{J}_n$ ($n\in \mathbb N$). We have $J_n\in \U $ for all $n$. Now,
$$S\subseteq \big[F_{\sup\{k\in \mathbb N\colon \gamma\in J_k\}}\big]_{\U }$$
and the latter space has the Banach--Mazur distance at most $\lambda$ to the ultraproduct of $C^*$-algebras, which is a Grothendieck space itself.

For \eqref{c-star-3}, note that the ultraproduct $[X_\gamma]_\U$ contains no complemented copy of $c_0$ by \cite[Proposition 3.3]{ACCGM:13}.
\end{proof}

Contreras and D\'\i az \cite[Section 3]{Contreras:99} proved that all the ultrapowers of the disk algebra $A$ and $H^\infty$ have property $(V)$.  Since  ultrapowers of Banach spaces over countably incomplete ultrafilters do not contain complemented copies of $c_0$ (\cite[Proposition 3.3]{ACCGM:13}), the following result is a consequence of  Proposition \ref{Groth-V-c0}. 
\begin{prop}
All ultrapowers of $A$ and $H^\infty$ over countably incomplete ultrafilters  have the Grothendieck property.
\end{prop}

Ultrapowers can be applied to obtain non-separable reflexive quotients of $\ell_\infty$. Indeed, since $\ell_q$ is a quotient of $\ell_\infty$ for $2\leqslant q<\infty$ (see the proof of Corollary \ref{cor:uncountable}), $\ell_\infty(\ell_q)$ is a quotient of $\ell_\infty\equiv \ell_\infty(\ell_\infty)$. Therefore, if $\U$ is a non-principal ultrafilter on $\N$, then the ultrapower $[\ell_p]_\U$ is a non-separable $L_p(\mu)$-space 
\cite{Heinrich:80} and a quotient of $\ell_\infty$. Thus, Problem \ref{q:reflex-quot} admits the following special case.  



\chapter{Miscellanea}\label{chap:miscellanea}

Here we briefly describe some additional results that could be worth to know for people interested in the Grothendieck property, but we have chosen not to treat in detail.

\section{Grothendieck operators}
We say that an operator $T\colon X\to Y$ is \emph{Grothendieck} if the adjoint operator $T^*$ takes weak$^*$-convergent sequences in $Y^*$ to weak-convergent sequences in $X^*$. Obviously, $X$ is Grothendieck if and only if the identity on $X$ is a~Grothendieck operator.

Pietsch \cite[3.2.6]{PietschBook:80} proved that an operator $T\in \mathscr{B}(X,Y)$ is Grothendieck if and only if for every $S\in\mathscr{B}(Y,Z)$ with separable range, the product $ST$ is weakly compact. Note also that the class of Grothendieck operators is an operator ideal in the sense of \cite{PietschBook:80}, which is surjective ($S$ surjective operator and $ST$ Grothendieck implies $T$ Grothendieck) and closed (the norm-limit of a convergent sequence of Grothendieck operators is Grothendieck). This is a consequence of the fact that the class of weakly compact operators shares these properties. \medskip

A subset $C$ of a Banach space $Y$ is said to be a \emph{Grothendieck subset} if $T(C)$ is relatively weakly compact for every operator $T\colon X\to c_0$. It is not difficult to prove the following result (see, \emph{e.g.}, \cite[Proposition 1]{Ghenciu:2017}).
\begin{prop}
An operator $T\colon X\to Y$ is Grothendieck if and only if it takes $B_X$ into a Grothendieck subset of $Y$.
\end{prop}

We do not know whether the ideal of Grothendieck operators has the factorisation property.
\begin{quest}
Does every Grothendieck operator factorise through a Banach space with the Grothendieck property?   
\end{quest}

Doma\'nski, Lindstr\"om, and Schl\"uchterman proved that if $T$ is a Grothendieck operator and $S$ is a compact operator, then the tensor-product operator $T\injtp S$ defined on the injective tensor product of the domains of the respective operators is still Grothendieck. Since the injective tensor product of two Grothendieck spaces need not be Grothendieck (for example, $\ell_2 \injtp \ell_2$ contains a complemented copy of $c_0$), tensoring the identity operators on such spaces provides counterexamples to tensorial stability of the ideal of Grothendieck operators.\smallskip

In \cite{GonzalezG:00}, \emph{Grothendieck holomorphic functions} between complex spaces $E$ and $F$ where defined as those holomorphic functions $f\colon E\to F$ for which each $x\in E$ has a neighbourhood $V_x$ such that $f(V_x)$ is a Grothendieck subset, and they where characterised in terms of factorisation as follows:
\begin{prop} \emph{(\cite[Theorem 6]{GonzalezG:00})}
A holomorphic function $f\colon E\to F$ is Grothendieck if and only if there exists a Banach space $G$, a holomorphic function $g\colon E\to G$, and a~Grothendieck operator $T\colon G\to F$ such that $f=T\circ g$.
\end{prop}

\section{Quantification of the Grothendieck property} 
Kruli\v{s}ov\'a (n\'ee Bendov\'a) introduced in \cite{Bendova:2014} a quantitative version of the Grothendieck property using the following two measures $\delta_{{\rm w}}$ and $\delta_{{\rm w}^*}$ of `non-weak-Cauchyness' and `non-weak*-Cauchyness', respectively, defined for bounded sequences $(f_n)_{n=1}^\infty$ in the dual of a~Banach space $X$:
\begin{itemize}
    \item $\delta_{{\rm w}}( (f_n)_{n=1}^\infty ) = \sup\limits_{x^{**}\in B_{X^{**}}} \inf\limits_{n\in\mathbb N} \sup\limits_{k,l \geqslant n} |\langle x^{**}, f_k - f_l \rangle|,$
    \item $\delta_{{\rm w}^*}( (f_n)_{n=1}^\infty ) = \sup\limits_{x\in B_{X}} \inf\limits_{n\in\mathbb N} \sup\limits_{k,l \geqslant n} |\langle f_k - f_l, x \rangle|.$    
\end{itemize}

\begin{definition}Let $\lambda \geqslant 1$. A Banach space $X$ is a $\lambda$-\emph{Grothendieck space}, whenever for every for every bounded sequence $(f_n)_{n=1}^\infty$ in $X^*$ we have
$$\delta_{{\rm w}}( (f_n)_{n=1}^\infty )\leqslant \lambda \cdot \delta_{{\rm w}^*}( (f_n)_{n=1}^\infty ). $$
\end{definition}

Kruli\v{s}ov\'a proved (\cite[Theorem 4.1]{Bendova:2014}) that $\ell_\infty(\Gamma)$ is a 1-Grothendieck space for any set $\Gamma$ (hence every 1-injective space is 1-Grothendieck too), however not every Grothendieck space is $\lambda$-Grothendieck for some $\lambda \geqslant 1$ (\cite[Theorem 2]{Bendova:2014}); this was obtained by forming an $\ell_2$-sum of $\lambda_n$-Grothendieck spaces with $\lambda_n\to \infty$ as $n\to \infty$. Lechner extended in \cite{Lechner:17} the results about the 1-Grothendieck property of 1-injective spaces by proving that for every subsequentially complete Boolean algebra $\mathscr{A}$ (see Section~\ref{sect:ck}) the space $C({\rm St}\, \mathscr{A})$ is 1-Grothendieck.\smallskip

Moreover, Kruli\v{s}ov\'a introduced in \cite{Krulisova:2017} a quantitative version of property $(V)$ as follows: given  $\lambda\geqslant 1$, a Banach space has \emph{quantitative property $(V)$} (with constant $\lambda$), whenever for every Banach space $Y$ and every operator $T\colon X\to Y$, one has $\gamma(T)\leqslant \lambda\cdot {\rm uc}(T)$, where
\begin{itemize}
  \item $\gamma(T)$ is the measure relative weak non-compactness of $T$:
  $$\gamma(T) = \sup\big\{ |\lim\limits_{n\to \infty} \lim\limits_{m\to \infty} \langle f_m, x_n\rangle - \lim\limits_{m\to \infty} \lim\limits_{n\to \infty} \langle f_m, x_n\rangle|\colon x_n\in T(B_X)\, (n\in \mathbb N) $$ $$\text{ and } (f_m)_{m=1}^\infty\text{ is a sequence in }B_{X^*}\text{ so that both limits exist}\big\}$$
  \item ${\rm uc}(T)$ measures how far is the operator $T$ from being unconditionally converging:
 $${\rm uc}(T) =\sup\big\{{\rm ca} ((\sum_{i=1}^n Tx_i )_{n=1}^\infty \colon (x_n)_{n=1}^\infty \subset X, \sup_{f\in B_{X^*}} \sum_{n=1}^\infty |\langle f, x_n \rangle | \leqslant 1 \big\},$$
where
 $${\rm ca}((y_n)_{n=1}^\infty) = \inf\limits_{n\in\mathbb N} \sup\limits_{k,l\geqslant n} \|y_k - y_l\|\qquad \big((y_n)_{n=1}^\infty \subset Y\big).$$
\end{itemize}
Using these notions, Kruli\v{s}ov\'a refined Pfitzner's theorem (Theorem~\ref{thm:pf}) by showing that $C^*$-algebras have the quantitative property $(V)$ (\cite[Theorem 4.1]{Krulisova:2017}). Moreover, she proved the following counterpart of Proposition~\ref{prop:V-dual} (\cite[Theorem 5.1]{Krulisova:2017}).
\begin{prop}
Every dual space with the quantative property $(V)$ is a $\lambda$-Grothendieck space for some $\lambda \geqslant 1$.
\end{prop}
An alternative approach to quantify the Grothendieck property was taken in \cite{ChenKaniaRuan:2021}, where the following quantity $G(X)$ was defined for any Banach space $X$:
$$G(X)=\sup_{\substack{(f_{n})_{n=1}^\infty \subseteq B_{X^{*}} \\ \operatorname{weak}^{*}\textrm{null}}}\inf_{(g_{n})_{n=1}^\infty \in \operatorname{cbs}((f_{n})_{n=1}^\infty )}\limsup_{n\to \infty}\|g_{n}\|.$$
Here $\operatorname{cbs}((f_{n})_{n=1}^\infty )$ stands for the family of all convex block subsequences of $(f_{n})_{n=1}^\infty$.)

One can use Proposition~\ref{prop:convblock} to prove that $X$ is a Grothendieck space if and only if $G(X) = 0$ (\cite[Theorem 4.5]{ChenKaniaRuan:2021}). It is easy to see that $G(c_0) = 1$.

\section{Positive Grothendieck property}
A Banach lattice (or more generally, an ordered Banach space) $E$ has the \emph{positive Grothendieck property}, if every positive weakly* convergent sequence in $E^*$ is weakly convergent. 

As in the case of the  Grothendieck property, a Banach lattice $E$ has the positive Groth\-endieck property if and only if every positive operator $T\colon E\to c_0$ is weakly compact. Moreover, the positive Grothendieck property is preserved by positive surjective operators. \smallskip

The space $c_0$ is a paradigm example of a Banach lattice failing the positive Grothen\-dieck property.
The positive Grothendieck property is much weaker than the usual one as for example $c$, the space of convergent sequences (in which $c_0$ has codimension one) has the positive Grothendieck property because every positive functional on $c$ attains its norm at $\mathds{1}_{\mathbb N}\in c$.

Wnuk (\cite[Proposition 2.12]{Wnuk:2013}) proved the following characterisation of Banach lattices with the positive Grothendieck property.
\begin{prop}
For a Banach lattice $E$, the following assertions are equivalent:
\begin{enumerate}
\item $E$ has the positive Grothendieck property;
\item for every non-reflexive Banach lattice $F$ with order-continuous norm there is no positive surjective operator $T\colon E\to F$;
\item There is no positive surjective operator $T\colon E\to c_0$.
\end{enumerate}
\end{prop}
K\"uhn proved that for Archimedean ordered Banach spaces (for example, Banach lattices with order unit) with countable Riesz interpolation property the Grothendieck and positive Grothendieck properties are equivalent (\cite[1.~Proposition]{Kuhn:80}; see also \cite[Theorem 5.3.13]{Meyer-Nieberg:91}).\smallskip

Koszmider and Shelah (\cite[Lemma 2.2]{KoszmiderShelah:13}) obtained a handy condition characterising Stone spaces $K$ for which $C(K)$ has the positive Grothendieck property.
\begin{prop}
Let $\mathscr{A}$ be a Boolean algebra. Then the space $C({\rm St}\,\mathscr{A})$ has the positive Grothen\-dieck property if and only if given an antichain
$\{A_n\colon n\in\mathbb N\}$  in $\mathscr{A}$, $\varepsilon > 0$, and a bounded sequence $(\mu_n)_{n=1}^\infty$ of bounded, finitely additive signed measures on $\mathscr{A}$ with $|\mu_n(A_n)| > \varepsilon$,
there exists $A\in \mathscr{A}$ such that the scalar sequence $(\mu_n(A))_{n=1}^\infty$ fails to converge.
\end{prop}


\section{{$C_0$}-semigroups of operators on Grothendieck spaces} 

In the present section, we discuss results that 
may be considered extensions of Proposition \ref{prop:SchD2}.

\begin{definition}
A Banach space $X$  has the \emph{Lotz property}, whenever every $C_0$-semigroup of operators on $X$ is uniformly continuous. 
\end{definition}
The following result is due to Lotz \cite{Lotz:85}.
\begin{prop}\label{prop:lotz}
Let $X$ be a Grothendieck space with the Dunford--Pettis property. Then $X$ has Lotz property.
\end{prop}
Therefore, Grothendieck $C(K)$-spaces have the Lotz property. In \cite{vanNeerven:92}, van Neerven proved a partial converse to the above theorem for Banach lattices.

\begin{prop}\emph{(\cite[Theorem 2]{vanNeerven:92})}
Let $E$ be a Banach lattice with a quasi-interior point. Then the following assertions are equivalent:
\begin{enumerate}
\item $E$ has the Lotz property,
\item $E$ is a Grothendieck space with the Dunford--Pettis property,
\item $E$ is isomorphic to a $C(K)$-space with the Grothendieck property.
\end{enumerate}
\end{prop}

In \cite[Examples 13 and 15]{Leung:88a} (see also Proposition \ref{prop:ex-Leung}), Leung constructs  a Banach space $E$ with an unconditional basis---hence $E^{**}$ is a Banach lattice---such that $E^{**}$ has the Lotz property \cite[Corollary 11]{Leung:88a} and it is Grothendieck, but fails the Dunford--Pettis property.  

Atalla proved that for a contractive operator $T\colon E\to E$ on a Grothendieck space $E$, the sequence of Ces\`aro means $\big(\tfrac{1}{n} \sum_{k=1}^n T^k\big)_{n=1}^\infty$ converges strongly if and only if the norm closure and the weak* closure of the range of $I_{E^*}-T^*$ coincide (\cite[Theorem 2.2]{Atalla:1976}) and noted that the hypothesis of the Grothendieck property of $E$ cannot be dropped (\cite[Examples 2.3]{Atalla:1976}) as witnessed by certain contractive Markov operators on $C[0,1]$. Shaw (\cite[Theorem 2]{Shaw:1983}) extended and improved this result in the setting of locally integrable semigroups of operators on Grothendieck spaces.
\begin{prop}\label{prop:shaw}
Let $E$ be a Grothendieck space and let $T(\cdot)$ be a locally integrable semigroup of operators on $E$. Set $S(t)x = \int_0^t T(s)x\,{\rm d}s$ $(x\in E, t > 0)$. Then $T(\cdot)$ is strongly ergodic, which means that
\begin{itemize}
    \item $\lim\limits_{t\to\infty} \tfrac{1}{t}S(t)x$ for all $x\in E$,
\end{itemize}
if and only if the following conditions are satisfied:
\begin{itemize}
\item $\limsup\limits_{t\to\infty} \tfrac{1}{t} \|S(t)\| <\infty$.
\item $\lim\limits_{t\to\infty} \tfrac{1}{t}T(t)S(u)x = 0$ for all $x\in E$ and $u>0$.
\item the norm closure and the weak* closure of 
$${\rm span} \bigcup_{t> 0} (T(t)^* - I_{E^*})(E^*)$$
coincide.
\end{itemize}

If $T(\cdot)$ is a $C_0$-semigroup with infinitesimal generator $A$, then the final condition may be replaced by  coincidence of the norm and the weak*-closures of the range of $A^*$.
\end{prop}

The final condition in Proposition~\ref{prop:shaw} is certainly redundant when $E$ is reflexive and indeed the result recovers Masani's result in this setting \cite{Masani:1976}.

\section{Locally convex Grothendieck spaces} The definition of a Grothendieck space naturally extends to topological vector spaces, however it is perhaps more natural in the context of (Hausdorff) locally convex spaces as they have non-trivial (continuous) bidual spaces.

Given a topological vector space $E$, we denote by $E^{*}$ the (continuous) dual space of $E$, and by $E^{**}$ the bidual of $E$; that is, the dual of $(E^*,\beta(E^*,E))$, where $\beta(E^*,E)$ is the strong topology.

\begin{definition} A (locally convex) topological vector space $E$ is \emph{Grothendieck} whenever the $\sigma(E^*,E)$--sequential and $\sigma(E^*, E^{**})$-sequential convergences coincide in equicontinuous subsets of $E^*$.\end{definition}

Research concerning locally convex Grothendieck  spaces has been conducted since 1980s in parallel to the Banach-space framework; we only highlight the most basic properties of Grothendieck spaces in this context and some results related to spaces of vector-valued continous functions on topological spaces that are counterparts of the results for $C(K,E)$-spaces presented in Section~\ref{sect:ck}. \smallskip

As observed by Freniche (\cite[Proposition 2.3]{Freniche:1985}), locally convex Grothendieck spaces have the following stability properties:
\begin{itemize}
\item A locally convex space $E$ is Grothendieck if and only if so is every dense linear subspace $F\subset E$.
\item Let $T\colon E\to F$ be a continuous linear operator such that for every bounded subset $B$ of $F$ there is a bounded subset $C$ of $E$ so that $B\subseteq \overline{T(C)}$. If $E$ is Grothendieck, then so is $F$.
\item If $E$ is an inductive limit (in the category of locally convex spaces) of a sequence $(E_n)_{n=1}^\infty$ of Grothendieck spaces, and if every bounded subset of $E$ is contained in some $E_n$, then $E$ is Grothendieck.
\end{itemize}

In \cite{Freniche:1985}, the definition of a G-space was extended to arbitrary topological spaces: a~topological space $X$ is a \emph{G-space}, whenever for every compact subset $K\subset X$ the Banach space $C(K)$ is Grothendieck. 

For a locally convex space $E$, $C(X, E)$ (the space of $E$-valued continuous functions on $X$ endowed with the compact-open topology) is Grothendieck if and only if for every compact subset $K\subset X$ the space $C(K,E)$ is Grothendieck (\cite[Theorem 2.4]{Freniche:1985}). Moreover, if $X$ is a G-space and $E$ is a strict inductive limit of Fr\'echet--Montel spaces, then $C(X, E)$ is Grothendieck. Additionally, Khurana proved that for a~compact G-space $K$ and a~Montel space $E$ the space $C(K,X)$ is Grothendieck \cite{Khurana:1991}. 

For completely regular spaces $X$ containing infinite compact subsets, further characterisations of Grothendieck $C(X,E)$-spaces have been obtained in \cite{KhuranaVielma:1993}. Further variations of this result were obtained in \cite{DomanskiDrewnowski:92}, where it was proved that for a completely regular space $E$ containing an infinite compact subset and a non-Montel Fr\'echet space, the space $C(X,E)$ contains a complemented copy of $c_0$.\smallskip

Valdivia proved that Corollary~\ref{cor:l1} extends to Fr\'echet spaces (\cite[Theorem 1]{Valdivia:1993}). 

\begin{prop}
Let $E$ be a Fr\'echet space. If $E$ is a non-reflexive Grothendieck space, then it contains an isomorphic copy of $\ell_1$.
\end{prop}

In the literature concerning Fr\'echet spaces, the  Grothendieck property is often considered in tandem with the Dunford--Pettis property; spaces satisfying both properties are termed \emph{GDP spaces}. Their systematic treatment may be found, \emph{e.g.}, in \cite{AlbaneseBonetRicker:2009,AlbaneseMangino:2011, BonetRicker:2007}; for non-Fr\'echet spaces see, \emph{e.g.}, \cite{GabriyelyanKakol:2020}.

\chapter{List of open problems}\label{chap:problems}

We conclude the paper by reiterating the accumulated open problems in the form of a~concise list. The numbering below corresponds to the numbering of Problems used earlier in the paper. Problems marked with ${}^\dagger$ have been either solved or their solutions have been announced during the preparation of this manuscript.\medskip

\begin{enumerate}

\item Does there exist an internal characterisation of Grothendieck spaces?
\smallskip

\item What class of Banach spaces $Y$ is characterised by the equality  $\mathscr{B}(X,Y)= \mathscr{W}(X,Y)$ for every Grothendieck space $X$?\smallskip

\item Does a non-reflexive Grothendieck space contain a copy of $c_0$?
\smallskip

\item Do Grothendieck spaces have property $(V)$?
\smallskip

\item
Do dual spaces with the Grothendieck property have property $(V)$?
\smallskip

\item Suppose that $X$ and $X^*$ are  Grothendieck. Is $X$ reflexive?
\smallskip

\item Let $X$ be a Grothendieck space. Is $X^{**}$ Grothendieck?
\smallskip

\item ${}^\dagger$ Let $X$ be a Grothendieck space. Is ${\rm At}(X)$ a weak$^*$-$G_\delta$ subset of $X^*$?
\smallskip


\item Characterise filters $\mathscr{F}$ for which $c_{0,\mathscr{F}}$ is a Grothendieck space.
\smallskip

\item
Is there an intrinsic characterisation of G-spaces? More precisely, can G-spaces be characterised topologically?
\smallskip

\item Let $\mathscr{A}$ be a Boolean algebra whose Stone space is a G-space. Does there exist a~Boolean subalgebra $\mathscr{B}\subset \mathscr{A}$ whose Stone space $K$ fails to be a G-space yet every weakly* convergent sequence of purely atomic measures on $K$ converges weakly?
\smallskip

\item
Characterise Banach spaces $E$ for which $E_w$ is Grothendieck. Is $E_w$ Grothen\-dieck when so is $E$?
\smallskip

\item Let $A$ be a $C^*$-algebra. Is $A_w$ a Grothendieck space?
\smallskip

\item
Let $E$ be a  $\mathcal{L}_\infty$-space which is Grothendieck. Is $E_w$ a  Grothendieck space?
\smallskip

\item
Is every $\mathcal{L}_\infty$-space without infinite-dimensional separable complemented subspace a Grothendieck space?
\smallskip

\item
Let $X$ be a Grothendieck $\mathcal{L}_\infty$-space. Does $X$ have property $(V)$?
\smallskip

\item
Suppose that a Banach space $X$ is  $\lambda$-separably injective for some $\lambda<2$. Is $X$ Grothendieck?
\smallskip

\item Let $n\geqslant 2$. Set $E_n = C^1([0,1]^n)$. Is $E_n^{**}$ a Grothendieck space?
\smallskip

\item Can we replace `the countable interpolation property' by `the countable monotone interpolation property' in Proposition \ref{prop:cones}? What happens in the case $E$ is a~$C^*$-algebra?
\smallskip

\item  Find a characterisation of dual Grothendieck spaces.
\smallskip

\item Let $X$ be a Grothendieck space without complemented separable, infinite-di\-men\-sio\-nal subspaces and let $Y$ be a closed subspace of $X$ with $X/Y$ infinite-di\-men\-sio\-nal separable.
Is it possible for $Y$ to be  isomorphic to $X$?
\smallskip

\item ${}^\dagger$ Let $X$ be a Grothendieck space. Suppose that $M$ is a closed subspace of $X$ such that ${\rm dens}\, X/M < \mathfrak{p}$. Is $M$ a Grothendieck space?
\smallskip

\item ${}^\dagger$ Let $M$ be a closed subspace of a Grothendieck space $X$ such that $X/M$ is reflexive. Is $M$ a Grothendieck space? 
\smallskip

\item
What is the number of pairwise non-isomorphic Grothendieck subspaces of $\ell_\infty$;
can it be $2^{\mathfrak{c}}$?
\smallskip

\item
Suppose that  $X\injtp Y$ is Grothendieck. Is $X$ or  $Y$  reflexive?
\smallskip

\item
Is there a pair of infinite-dimensional Banach spaces $X$ and $Y$ for which the tensor product $X\injtp Y$ is Grothendieck yet it is not reflexive?
\smallskip

\item Suppose that $\mathscr B(E)$ is Grothendieck. Is $E$ reflexive?
\smallskip

\item Suppose that $E$ is super-reflexive. Is $\mathscr B(E)$ Grothendieck? What happens when $E=\ell_p$, $1<p<\infty$, $p\neq 2$?
\smallskip

\item Suppose that $E$ is super-reflexive. Is $\ell_\infty(E)$ a Grothendieck space?\smallskip

\item  Can an ultrapower of a reflexive space be Grothendieck without being reflexive?
\smallskip



\item Does every Grothendieck operator factorise through a Banach space with the Grothendieck property?   

\end{enumerate}

\subsection*{Acknowledgements} We  thank the referee for a  careful reading of the manuscript and for providing an interesting solution to Problem~\ref{pr:8}.


\begin{thebibliography}{999}
\bibitem{AcostaKadets:2011} A.D.~Acosta; V.~Kadets. \emph{A characterization of reflexive spaces}. Math. Ann. 349, No. 3 (2011), 577--588.

\bibitem{AlbaneseBonetRicker:2009} A.A.~Albanese; J.~Bonet; W.J.~Ricker. \emph{Grothendieck spaces with the Dunford–Pettis property}. Positivity 14 (2010), 145--164.

\bibitem{AlbaneseMangino:2011} A.A.~Albanese; E.M.~Mangino. \emph{Some permanence results of the Dunford--Pettis and Grothendieck properties in lcHs}, Funct. Approximatio, Comment. Math. 44, No. 2 (2011), 243--258.

\bibitem{Ando:61} T.~Andô. \emph{Convergent sequences of finitely additive measures}. Pacific J. Math. 11 (1961), 395--404.

\bibitem{ArgyrosHaydon:2011} S.A. Argyros; R.~Haydon. \emph{A hereditarily indecomposable $\mathscr{L}_\infty$-space that solves the scalar-plus-compact problem}. Acta Math. 206 (2011), 1--54.

\bibitem{AriasFarmer:96} A.~Arias; J.F.~Farmer.
\emph{On the structure of tensor products of $\ell_p$-spaces}. Pacific J. Math., 175 (1996), 13--37.

\bibitem{Atalla:1976} R.E. Atalla. \emph{On the ergodic theory of contractions}. Revista Colombiana de Matem\'aticas, 10 (1976), 75--81.

\bibitem{ACCGM:13b}
A.~Avil\'es; F.~Cabello S\'anchez; J.M.F. Castillo; M. Gonz\'alez; Y. Moreno. \emph{On separably injective Banach spaces.} Adv.~Math. 234 (2013), 192--216.

\bibitem{ACCGM:13}
A.~Avil\'es; F.~Cabello S\'anchez; J.M.F. Castillo; M. Gonz\'alez; Y. Moreno.
\emph{On ultrapowers of Banach spaces of type $\mathcal{L}_\infty$.} Fund.~ Math. 222 (2013), 195--212.

\bibitem{ACCGM:16}
A.~Avil\'es; F.~Cabello S\'anchez; J.M.F. Castillo; M. Gonz\'alez; Y. Moreno.
\emph{Separably injective Banach Spaces.} Lecture Notes in Math. 2132. Springer-Verlag, 2016.

\bibitem{BarcenasMarmol:2005} D.~B\'arcenas; L.G.~M\'armol. \emph{On C(K) Grothendieck spaces}. Rend. Circ. Mat. Palermo 54 (2005), 209--216.

\bibitem{BKL:19} K.~Beanland; T.~Kania; N.J.~Laustsen.
\emph{The algebras of bounded operators on the Tsirelson and Baernstein spaces are not Grothendieck spaces.} Houston J. Math. 45 (2019), 553--566.

\bibitem{Bendova:2014} H.~Bendov\'a. \emph{Quantitative Grothendieck property}. J. Math. Anal. 412 (2014), 1097--1104.

\bibitem{Bielas:11} W.~Bielas. \emph{On convergence of sequences of Radon measures}, Praca semestralna nr 2 (semestr zimowy
2010/11),\newline \url{ssdnm.mimuw.edu.pl/pliki/prace-studentow/st/pliki/wojciech-bielas-2.pdf}.

\bibitem{Bombal:89} F.~Bombal.
\emph{Operators on vector sequence spaces.} London Math. Soc. Lecture Notes 140 (1989), 94--106.

\bibitem{BonetRicker:2007} J. Bonet; W.J. Ricker. \emph{Schauder decompositions and the Grothendieck and Dunford--Pettis properties in K\"othe echelon spaces of infinite order}. Positivity, 11 (2007), 77--93.

\bibitem{Bourgain:83} J.~Bourgain.
\emph{$H^\infty$ is a Grothendieck space.} Studia Math. 75 (1983), 193--216.

\bibitem{Bourgain:83a} J.~Bourgain. \emph{On weak completeness of the dual of spaces of analytic and smooth functions}. Bull. Soc. Math.
Belg. S\'er. B, 35 (1983), 111--118.

\bibitem{BD:80} J. Bourgain; F. Delbaen. \emph{A class of special $\mathcal{L}_\infty$-spaces.} Acta Math. 145 (1980), 155--176.

\bibitem{Brech:06} C.~Brech.
\emph{On the density of Banach spaces $C(K)$ with the Grothendieck property.} Proc. Amer. Math. Soc.~134 (2006), 3653--3663.

\bibitem{BrooksSaitoWright:2005} J.K.~Brooks; K.~Sait\^{o}; J.D.M. Wright. \emph{Operators on $\sigma$-complete $C^*$-algebras}. Quart.~J. Math. (Oxford) 56, Issue 3 (2005), 301--310.

\bibitem{Bu:20} Q.~Bu.
\emph{On Kalton's theorem for regular compact operators and Grothendieck property for positive projective tensor products.} Proc. Amer. Math. Soc. 148  (2020), 2459--2467.

\bibitem{Bu-Emma:05} Q.~Bu; G.~Emmanuele.
\emph{The projective and injective tensor products of $L^p[0,1]$ and $X$ being Grothendieck spaces.} Rocky Mount. J. Math. 35 (2005), 713--726.

\bibitem{Bu-Ji:10} Q.~Bu; D.~Ji; X.~Xue.
\emph{The Grothendieck property for injective tensor products of Banach spaces.} Czech. Math. J. 60 (2010), 1153--1159.

\bibitem{Bu-Li:17} Q.~Bu; Y.~Li.
\emph{New examples of non-reflexive Grothendieck spaces.} Houston J. Math. 43 (2017), 569--575.

\bibitem{CC:21} F. Cabello S\'anchez; J.M.F. Castillo. \emph{Homological methods in Banach space theory}, Cambridge Studies in Advanced Mathematics 193, Cambridge Univ. Press 2021.

\bibitem{Cast-Gonz:94} J.M.F.~Castillo; M. Gonz\'alez. \emph{Properties $(V)$ and $(u)$ are not three-space properties.} Glasgow Math. J. 36 (1994), 297--299.

\bibitem{Castillo-Gonzalez:97} J.M.F.~Castillo; M. Gonz\'alez. \emph{Three-space problems in Banach space theory.}
Lecture Notes in Math. 1667, Springer-Verlag 1997.

\bibitem{Cembranos:84} P.~Cembranos, \emph{$C(K, E)$ contains a complemented copy of $c_0$}, Proc. Amer. Math. Soc. 91 (1984), 556--558.

\bibitem{ChenKaniaRuan:2021} D.~Chen; T.~Kania; Y.~Ruan. \emph{Quantifying properties $(K)$ and $(\mu^s)$}. Preprint \url{arXiv:2102.00857} (2021), 19 pp.

\bibitem{ChetcutiHamhalter:2009} E. Chetcuti; J. Hamhalter. \emph{A noncommutative Brooks-Jewett theorem.} J. Math. Anal. Appl. 355, No. 2 (2009), 839--845.

\bibitem{ChristensenSinclair:1989} E.~Christensen; A.M.~Sinclair. \emph{Completely bounded isomorphisms of injective von Neumann algebras}. Proc. Edinburgh Math. Soc. (2) 32 (1989), 317--327.

\bibitem{CiliaEmmanuele:2015} R. Cilia; G.~Emmanuele. \emph{Pelczynski's property $(V)$ and weak* basic sequences}. Quaest. Math. 38, No. 3 (2015), 307--316.

\bibitem{Contreras:99} M.D. Contreras; D. D\'\i az.
\emph{Some Banach space properties of the duals of the disk algebra and  $H^\infty$.} Michigan  Math. J. 46 (1999), 123--141.

\bibitem{Coulhon:84} Th.~Coulhon. \emph{Suites d’opérateurs sur un espace $C(K)$ de Grothendieck}. C.R. Acad. Sci. Paris, 298 (1984) 13--15.

\bibitem{Dales:16} H.G. Dales; F.K. Dashiell, Jr.; A.T.-M. Lau; D. Strauss. \emph{Banach Spaces of Continuous Functions as Dual Spaces}. CMS Books Math., Springer, Cham, 2016.

\bibitem{Dashiell:81} F.K. Dashiell, Jr.
\emph{Nonweakly compact operators from order-Cauchy complete $C(S)$ lattices, with applications to Baire classes.} Trans. Amer. Math. Soc. 266  (1981), 397--416.

\bibitem{DFJP:74} W.J. Davis; T. Figiel; W.J. Johnson; A. Pe{\l}czy\'nski. \emph{Factoring weakly compact operators.} J. Funct. Anal. 35 (1974), 397--411.

\bibitem{Dean:67} D.W. Dean.
\emph{Schauder decompositions in $(m)$.} Proc. Amer. Math. Soc. 18 (1967), 619--623.

\bibitem{DGSR:1995} G.~Debs; G.~Godefroy; J.~Saint Raymond. \emph{Topological properties of the set of norm-attaining linear functionals}. Can. J. Math. 47 (1995), 318--329.

\bibitem{Diaz:95} D. D\'\i az.
\emph{Grothendieck's property in $L^p(\mu,X)$.}
Glasgow Math. J. 37 (1995), 379--382.

\bibitem{Diestel:73} J. Diestel.
\emph{Grothendieck spaces and vector measures.} In J. Diestel. ``Vector and operator valued measures and applications'' (Proc. Sympos., Alta, Utah, 1972), pp. 97--108. Academic Press, 1973.

\bibitem{Diestel:80} J. Diestel.
\emph{A survey of results related to the Dunford--Pettis property.} Proceedings of the Conference on Integration, Topology, and Geometry in Linear Spaces (Univ. North Carolina, Chapel Hill, N.C., 1979), pp. 15--60, Contemp. Math. 2, Amer. Math. Soc. Providence, 1980.

\bibitem{Diestel:84} J. Diestel.
\emph{Sequences and series in Banach Spaces.} Springer-Verlag, 1984.

\bibitem{DiestelJT:95} J. Diestel; H. Jarchow; A. Tonge.
\emph{Absolutely summing operators.} Cambridge Univ. Press, 1995.

\bibitem{DiestelS:78} J. Diestel; C.J. Seifert.
\emph{The Banach--Saks ideal, I. Operators acting on $C(\Omega)$.} Commentationes Math. (Tomus especialis in honorem W. Orlicz) I (1978), 109--118.

\bibitem{DiestelU:77} J. Diestel; J.J. Uhl, Jr. \emph{Vector measures.}  Math. Surveys, 15. Amer. Math. Soc., 1977.

\bibitem{DomanskiDrewnowski:92} P.~Domański; L.~Drewnowski. \emph{Fr\'echet spaces of continuous vector-valued functions: Complementability in dual Fr\'echet spaces and injectivity}. Studia Math. 102, No. 3 (1992), 257--267.

\bibitem{Domanskietal:97} P.~Domański; M.~Lindstr\"om; G.~Schlüchtermann. \emph{Grothendieck operators on tensor products}. Proc. Am. Math. Soc. 125, No. 8 (1997), 2285--2291.

\bibitem{Dowetal:1988} A.~Dow; A.V.~Gubbi; A.~Szyma\'nski, \emph{Rigid Stone spaces with ZFC}. Proc. Amer. Math. Soc. 102, No. 3 (1988), 745--748.



\bibitem{FPPeralta:2009} F.J.~Fern\'{a}ndez-Polo; A.M.~Peralta. \emph{Weak compactness in the dual space of a JB*-triple is commutatively determined}, Math. Scand. 105 (2009), 307--319.

\bibitem{FPPeralta:2010} F.J.~Fern\'{a}ndez-Polo; A.M.~Peralta. \emph{A short proof of a theorem of Pfitzner}, Quart.~J.~Math. Oxford 61 (2010), 329--336.

\bibitem{FigielJohnsonTzafriri:1975} T. Figiel; W.B. Johnson; L. Tzafriri. \emph{On Banach lattices and spaces having local unconditional structure, with applications to Lorentz function spaces,} J. Approx. Theory 13 (1975), 395--412.

\bibitem{Fremlin:1984} D.~Fremlin. \emph{Consequences of Martin's Axiom}. Cambridge Tracts in Math. 84, Cambridge University Press (1984). 

\bibitem{Freniche:1985} F.J.~Freniche. \emph{Grothendieck locally convex spaces of continuous vector valued functions}. Pacific J. Math. 120 (1985), 345--355.

\bibitem{GabriyelyanKakol:2020} S.~Gabriyelyan; J.~K\k{a}kol. \emph{Dunford--Pettis type properties and the Grothendieck property for function spaces}. Rev. Mat. Complut. 33, No. 3 (2020), 871--884.

\bibitem{Ghenciu:2017} I.~Ghenciu. \emph{The weak Gelfand--Phillips property in spaces of compact operators}. Comm. Math. Univ. Carolinae 58, No. 1 (2017), 35--47.

\bibitem{GhenciuLewis:2012} I.~Ghenciu; P.~Lewis. \emph{Completely continuous operators}. Colloq. Math. 126, No. 2 (2012), 231--256.

\bibitem{GodSab} G.~Godefroy; P.~Saab. \emph{Quelques espaces de Banach ayant les propri\'et\'es (V) ou (V*) de A. Pe{\l}czy\'nski}. C. R. Acad. Sci. Paris Sér. I Math. 303 (1986), 503--506.

\bibitem{Gonzalez:93} M. Gonz\'alez. 
\emph{Dual results of factorization for operators.} Ann. Acad. Sci. Fennicae 18 (1993), 3--11.

\bibitem{GonzalezG:95} M. Gonz\'alez; J.M. Guti\'errez. \emph{Polynomial Grothendieck properties.} Glasgow Math. J. 37 (1995), 211--219.

\bibitem{GonzalezG:00} M. Gonz\'alez; J.M. Guti\'errez. \emph{Surjective factorization of holomorphic mappings.} Comment.  Math. Univ. Carolinae 43 (2000), 469--476.

\bibitem{GLR:21} M.~Gonz\'alez; F.~Le\'on-Saavedra; M.P.~Romero de la Rosa.
\emph{On $\ell_\infty$-Grothendieck subspaces.} J.Math. Anal. Appl. 497 (2021) 124857, 5 pp.

\bibitem{GO:86} M. Gonz\'alez; V.M. Onieva. \emph{Lifting results for sequences in Banach spaces.} Math. Proc. Cambridge Philos. Soc. 105 (1989), 117--121.

\bibitem{GravesWheeler:1983} W.H. Graves; R.F. Wheeler. \emph{On the Grothendieck and Nikodym properties for algebras of Baire, Borel and universally measurable sets}. Rocky Mountain J. Math. 13 (1983), no. 2, 333--354.

\bibitem{Grothendieck:53} A. Grothendieck. \emph{Sur les applications lin\'eaires faiblement compactes d'espaces du type $C(K)$.} Canad. J. Math. 5 (1953), 129--173.

\bibitem{HaglerJ:77} J. Hagler; W.B. Johnson. \emph{On Banach spaces whose dual balls are not weak$^*$ sequentially compact.} Israel J. Math. 28 (1977), 325--330.

\bibitem{HajekMVZ:08} P.~H\'ajek; V.~Montesinos; J.~Vanderwerff; V.~Zizler. \emph{Biorthogonal systems in Banach spaces.} CMS Books in Math.\ 26; Springer-Verlag, 2008.

\bibitem{HarmandWW:93} P.~Harmand; D.~Werner; W.~Werner.
\emph{M-ideals in Banach spaces and Banach algebras.} Lecture Notes in Math.\ 1547; Springer-Verlag, 1993.

\bibitem{Haydon:78} R.~Haydon.
\emph{On dual $L_1$-spaces and injective bidual Banach  spaces.} Israel J. Math. 31 (1978), 142--152.

\bibitem{Haydon:81} R.~Haydon. \emph{A non-reflexive Grothendieck space that does not contain $\ell_\infty$.}
Israel J. Math. 40 (1981), 65--73.

\bibitem{Haydon:87} R.~Haydon. \emph{An unconditional result about Grothendieck spaces}. Proc. Amer. Math. Soc. 100. (1987), no. 3, 511--516.

\bibitem{Haydon:2001} R.~Haydon. \emph{Boolean rings that are Baire spaces}. Serdica Math. J. 27 (2001), 91--106.

\bibitem{HaydonLevyOdell:87} R.~Haydon; M.~Levy; E.~Odell.
\emph{On sequences without weak* convergent convex block subsequences}. Proc. Amer. Math. Soc., 100 (1987) 94--98.

\bibitem{Heinrich:80} S.~Heinrich. \emph{Ultraproducts in Banach space theory.} J. Reine Angew. Math. 313 (1980), 72--104.

\bibitem{Heinrich:80b} S.~Heinrich. \emph{Closed operator ideals and interpolation.} J. Funct. Anal. 35 (1980), 397--411.

\bibitem{Holub:73} J.R. Holub. \emph{Reflexivity of $L(E,F)$.} Proc. Amer. Math. Soc. 39 (1973), 175--177.

\bibitem{JiCraddockBu:2010} D.~Ji; M.~Craddock; Q.~Bu. \emph{Reflexivity and the Grothendieck property for positive tensor products of Banach lattices-I}. Positivity 14 (2010), 59--68.

\bibitem{Johnson:70} W.B. Johnson. \emph{No infinite-dimensional $P$ space admits a Markuschevich basis}. Proc. Amer. Math. Soc., 26 (1970), 467--468.

\bibitem{Johnson:84} W.B. Johnson. \emph{A complementary universal conjugate Banach space and its relation to the approximation problem}. Israel J. Math.~13 (3-4) (1972), 301--310.

\bibitem{JohnsonKaniaSchechtman:2016} W.B. Johnson; T.~Kania; G.~Schechtman. \emph{Closed ideals of operators on and complemented subspaces of Banach spaces of functions with countable support}. Proc. Amer. Math. Soc., 144 (2016), 4471--4485.

\bibitem{JohnsonR:72} W.B. Johnson; H.P. Rosenthal. \emph{On w$^*$-basic sequences and their application to the study of Banach spaces.} Studia Math. 43 (1972), 77--92.

\bibitem{Junge:2003}M.~Junge; N.~Ozawa; Z.-J. Ruan. \emph{On $\mathcal{OL}_\infty$ structures of nuclear C*-algebras}. Math. Ann., 325 (2003), 449--483.

\bibitem{KadisonRingrose:1983} R.V. Kadison; J.R. Ringrose. \emph{Fundamentals of the Theory of Operator Algebras, Vol. I, Elementary Theory}, Pure and Applied Math., Vol. 100 Academic Press, New York, 1983.

\bibitem{KakolLeiderman:2021} J.~K\k{a}kol; A.~Leiderman. \emph{A characterization of $X$ for which spaces $C_p(X)$ are distinguished and its applications}. Proc. Amer. Math. Soc. Ser.~B 8 (2021), 86--99.

\bibitem{KakolMarciszewskiSobotaZdomsky:2020} J.~K\k{a}kol; W.~Marciszewski; D.~Sobota; L.~Zdomskyy. \emph{On complemented copies of the space $c_0$ in spaces $C_p(X\times Y)$}. Preprint \url{arxiv.org/abs/2007.14723} (2020), 29 pp.

\bibitem{KakolMolto:2020} J.~K\k{a}kol; A.~Moltó. \emph{Witnessing the lack of the Grothendieck property in $C(K)$-spaces via convergent sequences} Rev. R. Acad. Cienc. Exactas Fís. Nat., Ser. A Mat., RACSAM 114, No. 4, Paper No. 179 (2020), 7 pp.

\bibitem{KakolSobotaZdomsky:2020} J.~K\k{a}kol; D.~Sobota; L.~Zdomskyy. \emph{The Josefson--Nissenzweig theorem, Grothendieck property, and finitely-supported measures on compact spaces}. Preprint \url{arxiv.org/abs/2009.07552} (2020), 57 pp.

\bibitem{KaltonP:79} N.J.~Kalton; N.T.~Peck. \emph{Twisted sums of sequence spaces and the three space problem.} Trans. Amer. Math. Soc. 255 (1979), 1--30.

\bibitem{Kania:13} T.~Kania.
\emph{A reflexive Banach space whose algebra of operators is not a Grothendieck space.} J. Math. Anal. Appl. 401 (2013), 242--243.

\bibitem{Kania:15} T.~Kania.
\emph{On C*-algebras which cannot be decomposed into tensor products with both factors infinite-dimensional},
Quart.~J.~Math. (Oxford) 66 (2015), 1063--1068.

\bibitem{Kania:2019} T.~Kania. \emph{A letter concerning Leonetti's paper ‘Continuous Projections onto Ideal Convergent Sequences’},
Results Math., 12 (2019), 4 pp.



\bibitem{Khurana:78} S.S.~Khurana. \emph{Grothendieck spaces}, Illinois J. Math., 22 (1978), 79--80.

\bibitem{Khurana:1991} S.S.~Khurana. \emph{Grothendieck spaces. II.} J. Math. Anal. Appl. 159, No. 1 (1991), 202--207.

\bibitem{KhuranaVielma:1993} S.S.~Khurana; J.~Vielma. \emph{Grothendieck spaces. III.} Simon Stevin 67 (1993), 81--85.

\bibitem{Knight:93} R.W.~Knight, \emph{$\Delta$-Sets}. Trans. Amer. Math. Soc. 339 (1993), 45--60.

\bibitem{Koszmider:04} P.~Koszmider. \emph{Banach spaces of continuous functions with few operators}, Math. Ann. 330 (2004), 151--183.

\bibitem{Koszmider:2010} P.~Koszmider. \emph{Set-theoretic methods in Banach space theory}. University of Wroc{\l}aw lecture notes, 2010,
\url{ssdnm.mimuw.edu.pl/pliki/wyklady/skrypt_PKoszmider.pdf}.

\bibitem{KoszmiderShelah:13} P.~Koszmider; S.~Shelah. \emph{Independent families in Boolean algebras with some separation properties}. Algebra Univers. 69 (2013), 305--312.

\bibitem{KoszmiderShelahSwietek:2018} P.~Koszmider; S.~Shelah; M.~\'{S}wi\k{e}tek. \emph{There is no bound on sizes of indecomposable Banach spaces} Adv.~Math. 323 (2018), 745--783.

\bibitem{Krulisova:2017} H.~Kruli\v{s}ov\'a. \emph{
$C^*$-algebras have a quantitative version of Pe{\l}czy\'nski’s property $(V)$}. Czech. Math. J. 67, No. 4 (2017), 937--951.

\bibitem{Kucher:99} O.V. Kucher. \emph{The Grothendieck property in the space $\ell_\infty(E)$ and the weak Banach--Saks property in $c_0(E)$.} J. Math. Sci. 96 (1999), 2828--2833.

\bibitem{Kuhn:80} B.~K\"uhn. \emph{Schwache Konvergenz in Banachverb\"anden}. Arch. Math. 35 (1980), 554--558.


\bibitem{Laustsen:2001} N.J. Laustsen. \emph{Matrix multiplication and composition of operators on the direct sum of an infinite
sequence of Banach spaces}, Math. Proc. Cambridge Philos. Soc. 131 (2001) 165--183.

\bibitem{Lechner:17} J.~Lechner. \emph{1-Grothendieck $C(K)$-spaces}. J. Math. Anal. Appl. 446 (2017), 1362--1371.

\bibitem{Leonetti:2018} P.~Leonetti. \emph{Continuous projections onto ideal convergent sequences}. Results Math., 73 (2018), 5~pp.

\bibitem{Leung:88} D.H. Leung.
\emph{Weak$^*$ convergence on higher duals of Orlicz spaces.} Proc. Amer. Math. Soc. 103 (1988), 797--800.

\bibitem{Leung:88a} D.H. Leung. \emph{Uniform convergence of operators and Grothendieck spaces with the Dunford--Pettis property}. Math. Z. 197 (1988), 21--32.

\bibitem{LeungR:90} D.H. Leung; F. R\"abiger. \emph{Complemented copies of $c_0$ in $\ell^\infty$-sums of Banach spaces.} Illinois J. Math. 34 (1990), 52--58.

\bibitem{LiBu:2017} Y. Li; Q. Bu. \emph{New examples of non-reflexive Grothendieck spaces}, Houston J. Math., 43 (2017), 569--575.

\bibitem{Lindenstrauss:64} J.~Lindenstrauss. \emph{On the extension of operators with range in a $C(K)$-space}, Proc. Amer. Math. Soc. 15 (1964), 218--225.

\bibitem{Lindenstrauss:67} J.~Lindenstrauss. \emph{On complemented subspaces of $m$}. Israel J. Math., 5 (1967), 153--156.

\bibitem{LT1} J.~Lindenstrauss; L.~Tzafriri. \emph{Classical Banach Spaces~I.} Springer-Verlag, 1977.

\bibitem{LT2} J.~Lindenstrauss; L.~Tzafriri. \emph{Classical Banach Spaces~II.} Springer-Verlag, 1979.

\bibitem{Lone:2017} N.A.~Lone. \emph{On the weak-Riemann integrability of weak*-continuous functions}. Mediterr.~J. Math. 14 (2017), 7 pp.

\bibitem{Lotz:85} H.P.~Lotz. \emph{Uniform convergence of operators on $L^\infty$ and similar spaces}, Math. Z. 190 (1985), no. 2, 207--220.

\bibitem{Lotz:10} H.P.~Lotz.
\emph{Weak convergence in the dual of weak $L^p$.} Israel J. Math. 176 (2010), 209--220.

\bibitem{martinezcervantez:2021} G.~Mart\'{\i}nez-Cervantes; J.~Rodr\'{\i}guez. \emph{On weak$^*$-extensible subspaces of Banach spaces}, Preprint \url{arxiv:2103.03590} (2021), 5.~pp.

\bibitem{Masani:1976} P.~Masani. \emph{Ergodic theorems for locally integrable semigroups of continuous linear operators on a Banach space}. Adv. Math., 21 (1976), 202--228.


\bibitem{mcwilliams:62}  R.D.~McWilliams. \emph{A note on weak sequential convergence}. Pacific J. Math 12 (1962), 333--335.

\bibitem{Megginson} R.E. Megginson.  \emph{An Introduction to Banach Space Theory.}
Springer-Verlag, 1998.

\bibitem{Meyer-Nieberg:91} P. Meyer-Nieberg. \emph{Banach lattices.} Springer-Verlag, 1991.

\bibitem{Molto:81} A. Molt\'o. \emph{On the Vitali--Hahn--Saks theorem}. Proc. Roy. Soc. Edinburgh Sect. A 90 (1981), 163--173.

\bibitem{Murphy:90} G.J.~Murphy. \emph{ $C^*$-algebras and operator theory}, Academic Press, Inc., Boston, MA, 1990.

\bibitem{vanNeerven:92} J.M.A.M.~van Neerven. \emph{A converse of Lotz's theorem on uniformly continuous semigroups.} {Proc. Am. Math. Soc.} 116 (1992), 525--527.

\bibitem{OdellR:75} E. Odell; H.P. Rosenthal. \emph{A double dual characterization of separable Banach spaces containing $\ell_1$}. Israel J. Math.~20 (3-4) (1975), 375--384.

\bibitem{dePagter:20} B.~de Pagter; F.A.~Sukochev. \emph{The Grothendieck property in Marcinkiewicz spaces.} Indag. Math. 31 (2020), 791--808.

\bibitem{Pedersen:2018} G.K.~Pedersen. \emph{$C^*$-algebras and their automorphism groups}. Edited by S. Eilers and D. Olesen. 2nd edition. Amsterdam: Elsevier/Academic Press, 2018.


\bibitem{Pelczynski:77} A.~Pe{\l}czy\'nski. \emph{Banach spaces of analytic functions and absolutely summing operators.} CBMS Reg. Conf. 30, Amer. Math. Soc. 1977.

\bibitem{PelczynskiSudakov:1962} A.~Pe{\l}czy\'nski; V.N.~Sudakov. \emph{Remark on non-complemented subspaces of the space $m(S)$}. Colloq. Math. 19 (1962), 85--88.


\bibitem{Pfitzner:94} H.~Pfitzner. \emph{Weak compactness in the dual of a C$^*$-algebra is determined commutatively.} Math. Ann. 298 (1994), 349--371.

\bibitem{PietschBook:80} A.~Pietsch.
\emph{Operator ideals.} North-Holland, Amsterdam, 1980.

\bibitem{Plebanek:04} G. Plebanek. \emph{A construction of a Banach space $C(K)$ with few operators.} Topology Appl. 143 (2004), 217--239.

\bibitem{PolyrakisXanthos:2011} I.A. Polyrakis; F.~Xanthos. \emph{Cone characterization of Grothendieck spaces and Banach spaces containing $c_0$}. Positivity, 15, No. 4 (2011), 677--693.

\bibitem{PolyrakisXanthos:2013} I.A. Polyrakis; F.~Xanthos. \emph{Grothendieck ordered Banach spaces with an interpolation property}. Proc. Amer. Math. Soc. 141 (2013), 1651--1661.

\bibitem{Rabiger:85} F. R\"abiger.
\emph{Beitr\"age zur Strukturtheorie der Grothendieck-R\"aume.} Sitzungsberichte der Heidelberger Akademie der Wissenschaften. Mathematisch-Naturwissenschaftliche Klasse 85-4, 78 pp. Springer-Verlag, 1985.

\bibitem{Rosenthal:70} H.P.~Rosenthal.
\emph{On relatively disjoint families of measures, with some applications to Banach space theory.} Studia Math. 37 (1970), 13--36.

\bibitem{Rosenthal:70b} H.P.~Rosenthal.
\emph{On injective Banach spaces and the spaces $L^\infty(\mu)$ for finite measures $\mu$.} Acta Math. 124 (1970), 205--248.

\bibitem{Ryan:02} R.A.~Ryan.
\emph{Introduction to tensor products of Banach Spaces.} Springer-Verlag, 2002.

\bibitem{SaitoWright:2003} K.~Sait\^o; J.D.~Maitland~Wright. \emph{$C^*$-algebras which are Grothendieck spaces}. Rend. Circ. Mat. Palermo 52 (2003), 141--144.

\bibitem{SaitoWright:2007} K.~Sait\^o; J.D.~Maitland~Wright. \emph{On classifying monotone complete algebras of operators}. Ricerche Mat. 56 (2007), 321--355.

\bibitem{Schachermayer:82} W.~Schachermayer. \emph{On some classical measure-theoretic theorems for non-sigma-complete Boolean algebras.} Dissertationes Math. (Rozprawy Mat.) 214 (1982), 33~pp.

\bibitem{Schlackow:2008} I.~Schlackow. \emph{Centripetal operators and Koszmider spaces}. Topology Appl. 155 (2008), 1227--1236.

\bibitem{Seever:68} G.L.~Seever.
\emph{Measures on $F$-spaces.} Trans. Amer. Math. Soc. 133 (1968), 267--280.

\bibitem{Semadeni:1964} Z.~Semadeni. \emph{On weak convergence of measures and $\sigma$-complete Boolean algebras}. Colloq. Math., 12 (1964), 229--233.

\bibitem{Shaw:1983} S.-Y. Shaw. \emph{Ergodic theorems for semigroups of operators on a Grothendieck space.} Proc. Japan Acad. 59 (A) (1983), 132--135.

\bibitem{SmithWilliams:86} R.R.~Smith; D.P.~Williams. \emph{The decomposition property for $C^*$-algebra}. J. Oper. Theory 16 (1986), 51--74.

\bibitem{SmithWilliams:88} R.R.~Smith; D.P.~Williams. \emph{Separable injectivity for $C^*$-algebras}. Indiana Univ. Math. J. 37, No. 1 (1988), 111--133.

\bibitem{Talagrand:80} M.~Talagrand.
\emph{Un nouveau $\mathcal{C}(K)$ qui poss\`ede la propri\'et\'e de  Grothendieck.}
Israel J. Math. 37 (1980), 181--191.

\bibitem{Valdivia:1993} M.~Valdivia. \emph{Fr\'echet spaces with no subspaces isomorphic to $\ell_1$}. Math. Japon. 38 (1993), 397--411. 

\bibitem{Wnuk:2013} W.~Wnuk. \emph{On the dual positive Schur property in Banach lattices}. Positivity 17 (2013), 759--773.

\bibitem{Woods:76} R.G.~Woods. \emph{Characterizations of some $C^*$-embedded subspaces of $\beta \mathbb N$}. Pacific J. Math. 65 (1976), 573--579.

\bibitem{ZhangGuLi:2020} S.~Zhang; Z.~Gu; Y.~Li. \emph{On Positive Injective Tensor Products Being Grothendieck Spaces}. Indian J. Pure Appl. Math. 51 (2020), 1239–1246.

\bibitem{Zippin:1977} M.~Zippin. \emph{The separable extension problem}, Israel J.~Math.~26 (1977), 372--387.

\end{thebibliography}
\end{document}